\documentclass[a4paper]{amsart}

\usepackage{caption}
\usepackage{subcaption}
\usepackage{pgfplots}
\usepgfplotslibrary{groupplots}
\usepackage{amsmath,amssymb,amsfonts, amsthm}
\usepackage[foot]{amsaddr}
\usepackage{nicefrac}
\usepackage{bbm}
\usepackage{bm}

\usepackage{enumitem}
\usepackage[]{algorithm2e}
\usepackage{euscript}
\usepackage[bookmarks=true,hidelinks]{hyperref}
\numberwithin{equation}{section}		

\newcommand{\R}{\mathbb{R}}
\newcommand{\Z}{\mathbb{Z}}
\newcommand{\N}{\mathbb{N}}
\newcommand{\Q}{\mathbb{Q}}
\newcommand{\eps}{\varepsilon}
\newcommand{\Ph}{U}
\newcommand{\E}{\ensuremath{\mathbb{E}}}

\newcommand{\D}{D}

\newcommand{\TD}{\R_+}

\newcommand{\Meas}{\ensuremath{\mathcal{M}}}
\newcommand{\dd}{\ensuremath{\;d}}
\newcommand{\Dx}{{\ensuremath{\Delta}}}

\newcommand{\Ypair}[2]{\bigl\langle #1, #2\bigr\rangle}
\newcommand{\NumericalEvolution}[1]{\ensuremath{\mathcal{S}^{#1}}}

\newcommand{\Borel}{\mathcal{B}}
\newcommand{\Indicator}[1]{\mathbf{1}_{#1}}
\newcommand{\cm}{\ensuremath{\bm{\nu}}}

\newcommand{\Prob}{\ensuremath{\mathcal{P}}}
\newcommand{\weakto}{\ensuremath{\rightharpoonup}}
\newcommand{\weaksto}{\ensuremath{\overset{*}{\weakto}}}

\renewcommand{\phi}{\varphi}
\newcommand{\loc}{{\ensuremath{\text{loc}}}}

\newcommand{\Corrmeas}{\EuScript{L}}
\newcommand{\Caratheodory}{\EuScript{H}}
\newcommand{\Lgfuncs}{\mathcal{C}}

\newcommand{\norm}[1]{\left\|#1\right\|}

\renewcommand{\vec}[1]{\ensuremath{\mathbf{#1}}}
\newcommand{\bi}{{\vec{i}}}
\newcommand{\bj}{{\vec{j}}}
\newcommand{\be}{{\vec{e}}}

\newcommand{\pdpd}[2]{\ensuremath{\frac{\partial #1}{\partial #2}}}
\newcommand{\ddt}[1]{\ensuremath{\frac{\text{d}}{\text{d}t}#1}}
\newcommand{\hf}{{\ensuremath{\nicefrac{1}{2}}}}
\newcommand{\iphf}{{i+\hf}}
\newcommand{\imhf}{{i-\hf}}
\newcommand{\jphf}{{j+\hf}}
\newcommand{\jmhf}{{j-\hf}}
\usepackage{subcaption}
\newcommand{\pushforward}[2]{#1\##2}
\renewcommand{\leq}{\leqslant}
\renewcommand{\geq}{\geqslant}

\DeclareMathOperator{\Lip}{Lip}
\newcommand{\AvgS}[1]{\ensuremath{#1}}
\newcommand{\Avg}[2]{\ensuremath{\AvgS{#1}_{#2}}}
\newcommand{\AvgEvolved}[4]{\ensuremath{\AvgS{#1}_{#4}^{#2,#3}}}
\newcommand{\GridIndexSetD}{\ensuremath{\Z^d}}
\newcommand{\GridIndexSetDK}[1]{\ensuremath{(\Z^d)^{#1}}}

\newcommand{\InitialData}[1]{\ensuremath{\bar{#1}}}

\DeclareMathOperator{\supp}{supp}
\DeclareMathOperator*{\esssup}{ess\,sup}
\newcommand{\MeshResolution}[1]{#1 $\times$ #1}

\def\Xint#1{\mathchoice
{\XXint\displaystyle\textstyle{#1}} 
{\XXint\textstyle\scriptstyle{#1}} 
{\XXint\scriptstyle\scriptscriptstyle{#1}} 
{\XXint\scriptscriptstyle\scriptscriptstyle{#1}} 
\!\int}
\def\XXint#1#2#3{{\setbox0=\hbox{$#1{#2#3}{\int}$ }
\vcenter{\hbox{$#2#3$ }}\kern-.56\wd0}}
\def\intavg{\Xint-}

\usepackage{todonotes}

\newtheorem{theorem}{Theorem}

\newtheorem{proposition}[theorem]{Proposition}
\newtheorem{lemma}[theorem]{Lemma}
\newtheorem{remark}[theorem]{Remark}
\newtheorem{notation}[theorem]{Notation}

\numberwithin{theorem}{section}
\theoremstyle{definition}
\newtheorem{definition}[theorem]{Definition}
\newtheorem{example}[theorem]{Example}
\newtheorem{Algorithm}{Algorithm}
\usepackage{cleveref}

\newlength\figureheight
\newlength\figurewidth

\usepackage{booktabs}
\usepackage{algorithmicx}
\usepackage{algpseudocode}
\usepackage{graphicx}

\usepackage{etoolbox}

\newtoggle{usetikz}
\togglefalse{usetikz}
\newcommand{\InputImage}[3]{}
\iftoggle{usetikz}{%
	\renewcommand{\InputImage}[3]{%
		\setlength\figureheight{#2}%
		\setlength\figurewidth{#1}%
		\input{img_tikz/#3.tikz}%
	}%
} { %
  \renewcommand{\InputImage}[3]{%
		\includegraphics[width=#1]{img_downscaled/#3_notitle}%
	}%
}
\pgfplotsset{every tick label/.append style={font=\tiny}}
\pgfplotsset{
	tick label style = {font = {\fontsize{6 pt}{12 pt}\selectfont}},
	label style = {font = {\fontsize{6 pt}{12 pt}\selectfont}},
	legend style = {font = {\fontsize{6 pt}{12 pt}\selectfont}},  
}

\title[Numerical approximation of statistical solutions]{Statistical solutions of hyperbolic systems of conservation laws: numerical approximation}
\author{U. S. Fjordholm}
\address[U. S. Fjordholm]{Department of Mathematics, University of Oslo, Postboks 1053 Blindern, 0316 Oslo, Norway}
\email{ulriksf@math.uio.no}
\author{K. Lye \and S. Mishra}
\address[K. Lye, S. Mishra]{Seminar for Applied Mathematics, ETH Z\"urich, R\"amistrasse 101, 8092 Z\"urich, Switzerland}
\email[K. Lye]{kjetil.lye@sam.math.ethz.ch}
\email[S. Mishra]{siddhartha.mishra@sam.math.ethz.ch}
\author{F. Weber}
\address[F. Weber]{Department of Mathematical Sciences, Carnegie Mellon University, 5000 Forbes Avenue, Pittsburgh, PA 15213, USA}
\email{franzisw@andrew.cmu.edu}	
\date{\today}
\begin{document}
\maketitle


\begin{abstract}
Statistical solutions are time-parameterized probability measures on spaces of integrable functions, that have been proposed recently as a framework for global solutions and uncertainty quantification for multi-dimensional hyperbolic system of conservation laws. By combining high-resolution finite volume methods with a Monte Carlo sampling procedure, we present a numerical algorithm to approximate statistical solutions. Under verifiable assumptions on the finite volume method, we prove that the approximations, generated by the proposed algorithm, converge in an appropriate topology to a statistical solution. Numerical experiments illustrating the convergence theory and revealing interesting properties of statistical solutions, are also presented.
\end{abstract}

\section{Introduction}
Systems of conservation laws are a large class of nonlinear partial differential equations of the generic form
\begin{subequations}\label{eq:cauchy}
\begin{align}
\partial_t u + \nabla_x \cdot f(u) &= 0  \label{eq:cl} \\
u(x,0) &= \bar{u}(x).
\end{align}
\end{subequations}
Here, the unknown $u = u(x,t) : D\times\R_+\to U$ is the vector of \emph{conserved variables} and $f = (f^1, \dots, f^d) : \R^N \to \R^{N\times d}$ is the \emph{flux function}. Here, we denote $\R_+ := [0,\infty)$ and $U:= \R^N$, and we let the physical domain $D \subset \R^d$ be some open, connected set. The system \eqref{eq:cl} is \emph{hyperbolic} if the flux Jacobian $\partial_u(f \cdot n)$ has real eigenvalues for all $n \in \R^d$ with $|n| = 1$.

Many important models in physics and engineering are described by hyperbolic systems of conservation laws. Examples include the compressible Euler equations of gas dynamics, the shallow water equations of oceanography, the Magneto-Hydro-Dynamics (MHD) equations of plasma physics, and the equations of nonlinear elastodynamics \cite{Dafermos}.  

\subsection{Entropy Solutions.} It is well known that even if the initial data $\overline{u}$ in \eqref{eq:cauchy} is smooth, solutions of \eqref{eq:cauchy} develop discontinuities, such as \emph{shock waves} and \emph{contact discontinuities}, in finite time. Therefore, solutions to
\eqref{eq:cauchy} are sought in the sense of distributions: A function $u\in L^\infty(\R^d\times\R_+,\R^N)$ is an \emph{weak solution} of \eqref{eq:cauchy} if it satisfies 
\begin{equation}\label{eq:wsoln}
\int_{\R_+}\int_{\R^d} \partial_t \phi(x,t) u(x,t) + \nabla_x\phi(x,t) \cdot f(u(x,t))\ dxdt + \int_{\R^d}\phi(x,0) \overline{u}(x) \ dx = 0
\end{equation}
for all test functions $\phi\in C^1_c(\R^d\times\R_+)$.

As weak solutions are not unique \cite{Dafermos}, it is necessary to augment them with additional admissibility criteria or \emph{entropy conditions} to recover uniqueness. These entropy conditions are based on the existence of a so-called \emph{entropy pair} --- a pair of functions $\eta:\R^N\to\R$, $q:\R^N\to\R^d$, with $\eta$ convex and $q$ satisfying the compatibility condition $q' = \eta' \cdot f'$ (where $f'$ and $q'$ are the Jacobian matrices of $f$ and $q$). An \emph{entropy solution} of \eqref{eq:cauchy} is a weak solution that also satisfies the so-called \emph{entropy inequality}
\begin{equation}\label{eq:entrcond}
\int_{\R_+}\int_{\R^d} \partial_t \phi(x,t) \eta(u(x,t)) + \nabla_x\phi(x,t) \cdot q(u(x,t))\ dxdt + \int_{\R^d}\phi(x,0) \eta(\bar{u}(x)) \ dx \geq 0
\end{equation}
for all nonnegative test functions $\phi\in C^1_c(\R^d\times\R_+)$. Depending on the availability of entropy pairs $(\eta,q)$, the entropy condition leads to various \emph{a priori} bounds on $u$: If, say, $\eta(u)=|u|^p$ (or some perturbation thereof) for some $p\geq1$ then \eqref{eq:entrcond} leads to
\begin{equation}\label{eq:lpbound}
\int_{\R^d}|u(x,t)|^p\,dx \leq \int_{\R^d}|\bar{u}(x)|^p\,dx \qquad \forall\ t>0,
\end{equation}
see e.g.~\cite{Dafermos,GR,HR}.

The global well-posedness of entropy solutions of \eqref{eq:cauchy} has been addressed both  for (multi-dimensional) scalar conservation laws \cite{Kruz1} and for systems in one space dimension (see \cite{Glimm1,BB1,Bress1,HR} and references therein). However, there are no global existence results for entropy solutions of multi-dimensional systems of conservation laws with generic initial data. On the other hand, it has been established recently in \cite{CDL1,CDK1} that entropy solutions for some systems of conservation laws (such as isentropic Euler equations in two space dimensions) may not be unique. This is a strong indication that the paradigm of entropy solutions is not the correct framework for the well-posedness of multi-dimensional systems of hyperbolic conservation laws.

\subsection{Numerical schemes} A wide variety of numerical methods have been developed to approximate entropy solutions of \eqref{eq:cauchy} in a robust and efficient manner. These include finite volume, (conservative) finite difference, discontinuous Galerkin (DG) finite element and spectral (viscosity) methods, see textbooks \cite{GR,HEST1} for further details. Rigorous convergence results of numerical methods to entropy solutions are only available for scalar conservation laws (see e.g.~\cite{GR} for \emph{monotone schemes} and \cite{USFthesis} for high-order schemes) and for some specific numerical methods for one-dimensional systems (\cite{Glimm1} for Glimm's scheme and \cite{HR} for front tracking).

There are no rigorous convergence results to entropy solutions for any numerical schemes approximating multi-dimensional systems of conservation laws. To the contrary, several numerical experiments, such as those presented recently in \cite{fkmt,FMTacta}, strongly suggest that there is no convergence of approximations generated by standard numerical schemes for \eqref{eq:cauchy}, as the mesh is refined. This has been attributed to the emergence of turbulence-like structures at smaller and smaller scales upon mesh refinement (see Figure 4 of \cite{fkmt}). 

\subsection{Measure-valued and Statistical solutions.} Given the lack of well-posedness of entropy solutions for multi-dimensional systems of conservation laws and the lack of convergence of numerical approximations to them, it is natural to seek alternative solution paradigms for \eqref{eq:cauchy}. A possible solution framework is that of \emph{entropy measure-valued solutions}, first proposed by DiPerna in \cite{diperna}, see also \cite{diperna-2}. Measure-valued solutions are \emph{Young measures} \cite{Young}, that is, space-time parameterized probability measures on phase space $\R^N$ of \eqref{eq:cauchy}. Global existence of entropy measure-valued solutions has been considered in \cite{diperna,GWZ1} and in \cite{fkmt,FMTacta} where the authors constructed entropy measure-valued solutions by proving convergence of a Monte Carlo type ensemble-averaging algorithm, based on underlying entropy stable finite difference schemes. 

Although entropy measure-valued solutions for multi-dimensional systems of conservation laws exist globally, it is well known that they are not necessarily unique; see \cite{Sch1, FMTacta} and references therein. In particular, one can even construct multiple entropy measure-valued solutions for scalar conservation laws for the same measure-valued initial data \cite{FMTacta}. Although generic measure-valued solutions might not be unique, numerical experiments presented in \cite{fkmt} indicate that measure-valued solutions of \eqref{eq:cauchy}, computed with the ensemble-averaging algorithm of \cite{fkmt}, are stable with respect to initial perturbations and to the choice of underlying numerical method. This suggests imposing additional constraints on entropy-measure valued solutions in order to recover uniqueness. 

In \cite{FLM17}, the authors implicated the lack of information about (multi-point) statistical correlations in Young measures as a possible cause of the non-uniqueness of entropy measure-valued solutions. Consequently, they introduced a \emph{stronger} solution paradigm termed \emph{statistical solutions} for hyperbolic systems of conservation laws \eqref{eq:cauchy}. Statistical solutions are time-parameterized probability measures on some Lebesgue space $L^p(D;U)$ satisfying \eqref{eq:cl} in an averaged sense. The choice of the exponent $p\geq1$ depends on the available \emph{a priori} bounds for solution of \eqref{eq:cauchy}, such as \eqref{eq:lpbound}. It was shown in \cite{FLM17} that probability measures on $L^p(D;U)$ can be identified with (and indeed are equivalent to) \emph{correlation measures} --- a hierarchy of Young measures defined on tensorized versions of the domain $D$ and the phase space $U$ in \eqref{eq:cauchy}. Statistical solutions have also been introduced in the context of the incompressible Navier--Stokes equations by Foia\c{s} et al.; see \cite{FMRT1} and references therein.

 In \cite{FLM17}, the authors defined statistical solutions of systems of conservation laws \eqref{eq:cl} by requiring that the moments of the time-parameterized probability measure on $L^p(D)$ (or equivalently, of the underlying correlation measure) satisfy an infinite set of (tensorized) partial differential equations (PDEs), consistent with \eqref{eq:cl}. 

The first member of the hierarchy of correlation measures for a statistical solution, is a (classical) Young measure and it can be shown to be an entropy measure valued solution of \eqref{eq:cauchy}, in the sense of DiPerna \cite{diperna}. The $k$-th member ($k \geq 2$) of the hierarchy represents $k$-point spatial correlations. Thus, a statistical solution can be thought of as a measure-valued solution, augmented with information about all possible (multi-point) spatial correlations \cite{FLM17}. Consequently, statistical solutions contain much more information than measure-valued solutions. 

In \cite{FLM17}, the authors constructed a \emph{canonical} statistical solution for scalar conservation laws in terms of the data-to-solution semi-group of Kruzkhov \cite{Kruz1} and showed that this statistical solution is unique under a suitable entropy condition. Numerical approximation of statistical solutions of scalar conservation laws was considered in \cite{FLyeM1}, where the authors proposed Monte Carlo and multi-level Monte Carlo (MLMC) algorithms to compute statistical solutions and showed their convergence. 

\subsection{Aims and scope of this paper.}
Given this background, our main aim in this paper is to study statistical solutions for multi-dimensional systems of conservation laws. To this end, we obtain the following results:
\begin{itemize}
\item We propose a Monte Carlo ensemble averaging based algorithm for computing statistical solutions of systems of conservation laws. This algorithm is a variant of the Monte Carlo algorithms presented in \cite{fkmt} and \cite{FLyeM1}. 
\item Under reasonable assumptions on the underlying numerical scheme, we prove convergence of the ensemble-averaging algorithm to a statistical solution. It is highly non-trivial to identify an appropriate topology on time parameterized probability measures on $L^p(D)$ in order to prove convergence of the computed statistical solutions. To this end, we find a suitable topology and prescribe {novel} sufficient conditions that ensure convergence in this topology.
\item We present several numerical experiments that illustrate the robustness of our proposed algorithm and also reveal interesting properties of statistical solutions of \eqref{eq:cl}.
\end{itemize} 
As a consequence of our convergence theorem, we establish a \emph{conditional} global existence result for multi-dimensional systems of conservation laws. Moreover, we also propose an {entropy condition} under which we prove a \emph{weak-strong} uniqueness result for statistical solutions, that is, we prove that if there exists a statistical solution of sufficient regularity (in a sense made precise in Section \ref{sec:3}), then all entropy statistical solutions agree with it. 

The rest of the paper is organized as follows: In Section \ref{sec:2}, we provide the mathematical framework by describing the concepts of correlation measures and statistical solutions. We also provide characterizations of the topology on probability measures on $L^p(D)$, in which our subsequent numerical approximations will converge. The entropy condition and the weak-strong uniqueness of statistical solutions are presented in Section \ref{sec:3} and the Monte Carlo ensemble-averaging algorithm (and its convergence) is presented in Section \ref{sec:4}. Numerical experiments are presented in Section \ref{sec:5} and the results of the paper are summarized and discussed in Section \ref{sec:6}. 
\section{Probability measures on $L^p(D;U)$ and Statistical solutions}
\label{sec:2}
In the usual, deterministic interpretation of \eqref{eq:cl}, one attempts to find a function $u = u(t):D\to U$ satisfying \eqref{eq:cl} in a weak or strong sense. (Here, as in the Introduction, we let $D\subset\R^d$ be an open, connected set and we denote $U:=\R^N$.) By contrast, a statistical solution of \eqref{eq:cl} is a probability measure $\mu=\mu_t$ distributed over such functions $u$ and satisfying \eqref{eq:cl} in some averaged sense. Solutions of \eqref{eq:cl} are most naturally found in (some subspace of) $L^p(D;U)$, so $\mu_t$ is required to be a probability measure on $L^p(D;U)$ at each time $t$. In order to write down constitutive equations for $\mu$, it is more natural to work with finite-dimensional projections or \emph{marginals} of $\mu$; these are the so-called \emph{correlation measures} \cite{FLM17}. In this section we provide a self-contained description of correlation measures, probability measures over $L^p(D;U)$, and statistical solutions of \eqref{eq:cl}. 

In order to link probability measures on $L^p$ to their finite-dimensional marginals, we prove in Section~\ref{sec:probmeasLp} that a sequence of such measures converges weakly if and only if it converges with respect to a certain class $\Lgfuncs_p$ of finite-dimensional observables. In Section \ref{sec:corrmeas} we introduce correlation measures and we show that these are in a one-to-one relationship with probability measures over $L^p$, and that they are linked precisely through the finite-dimensional observables $\Lgfuncs_p$. We also prove a compactness result for sequences of correlation measures. In Section~\ref{sec:timedeppromeas} we treat time-parametrized probability and correlation measures, and we prove measurability and compactness results. Finally, in Section~\ref{sec:statsoln} we provide the definition of statistical solutions of \eqref{eq:cl}.

For the sake of clarity, many of the proofs in this rather technical section have been moved to Appendices \ref{app:weakconvergenceequivalenceproof}, \ref{app:proofcompactness} and \ref{app:time-dep-cm}.

\begin{notation}
If $X$ is a topological space then we let $\Borel(X)$ denote the Borel $\sigma$-algebra on $X$, we let $\Meas(X)$ denote the set of signed Radon measures on $(X,\Borel(X))$, and we let $\Prob(X)\subset\Meas(X)$ denote the set of all probability measures on $(X, \Borel(X))$, i.e.~all non-negative $\mu\in\Meas(X)$ with $\mu(X)=1$ (see e.g.~\cite{ambrosio_gradient_flows,Bil08,Kle}). For $k\in\N$ and a multiindex $\alpha\in \{0,1\}^k$ we write $|\alpha|=\alpha_1+\dots+\alpha_k$ and $\bar{\alpha} = \mathbbm{1}-\alpha=(1-\alpha_1,\dots,1-\alpha_k)$, and we let $x_{\alpha}$ be the vector of length $|\alpha|$ consisting of the elements $x_i$ of $x$ for which $\alpha_i$ is non-zero. For a vector $x=(x_1,\dots,x_k)$ we write $\hat{x}^i = (x_1,\dots,x_{i-1},x_{i+1},\dots,x_k)$. For a vector $\xi=(\xi_1,\dots,\xi_k)$ we write $|\xi^\alpha| = |\xi_{\alpha_1}|\cdots|\xi_{\alpha_k}|$ with the convention $\xi_{\alpha_i}=1$ if $\alpha_i=0$.
\end{notation}

\subsection{Probability measures on $L^p(D)$ and weak convergence}\label{sec:probmeasLp}

If $X$ is any topological space and we are given a sequence $\mu_1,\mu_2,\dots \in \Prob(X)$ and some $\mu\in\Prob(X)$, then we say that $\{\mu_n\}_{n\in\N}$ \emph{converges weakly to $\mu$}, written $\mu_n\weakto \mu$, if
\begin{equation}\label{eq:muconvergence}
\Ypair{\mu_n}{F} \to \Ypair{\mu}{F} \qquad\text{as } n\to\infty
\end{equation}
for every $F\in C_b(X)$. (Here and elsewhere, $\Ypair{\mu}{F}=\int_X F(x)\,d\mu(x)$ denotes the expectation of $F$ with respect to $\mu$.) We will be particularly interested in the case $X=L^p(D;U)$, so to study weak convergence in this space we need to work with the space $C_b(L^p(D;U))$. In this section we will see that it is sufficient to prove \eqref{eq:muconvergence} for a much smaller class of functionals $F$, namely those which depend only on finite-dimensional projections of $u\in L^p(D;U)$.

If $E$ and $V$ are Euclidean spaces then a measurable function $g:E\times V\to\R$ is called a \emph{Carath\'eodory function} if $\xi\mapsto g(x,\xi)$ is continuous for a.e.~$x\in E$ and $x\mapsto g(x,\xi)$ is measurable for every $\xi\in V$ (see e.g.~\cite[Section 4.10]{AliBor}). For a number $k\in\N$ and a Carath\'eodory function $g=g(x,\xi):D^k\times U^k\to\R$ we define the functional $L_g : L^p(D;U) \to \R$ by
\begin{equation}\label{eq:Lgdef}
L_g(u) := \int_{D^k} g(x_1,\dots,x_k,u(x_1),\dots,u(x_k))\,dx.
\end{equation}
(Here, $D^k$ denotes the product space $D^k=D\times\dots\times D$, and similarly for $U^k$.) The above integral is clearly not well-defined for every Carath\'eodory function $g$, so we restrict our attention to the following class.

\begin{definition}\label{def:caratheodory}
For every $k\in\N$, we let $\Caratheodory^{k,p}(D;U)$ denote the space of Carath\'eodory functions $g:D^k\times U^k \to \R$ satisfying
\begin{equation}\label{eq:gbounded}
|g(x,\xi)| \leq \sum_{\alpha\in\{0,1\}^k} \phi_{|\bar{\alpha}|}(x_{\bar{\alpha}})|\xi^\alpha|^p \qquad \forall\ x\in D^k,\ \xi\in U^k
\end{equation}
for nonnegative functions $\phi_i\in L^1(D^i)$, $i=0,1,\dots,k$ (with the convention that $L^1(D^0)\cong \R$; see also Example \ref{ex:caratheodory}). We let $\Caratheodory^{k,p}_1(D;U)\subset \Caratheodory^{k,p}(D;U)$ denote the subspace of functions $g$ which are locally Lipschitz continuous, in the sense that there is some $r>0$ and some nonnegative $h \in \Caratheodory^{k-1,p}(D;U)$ such that
\begin{equation}\label{eq:glipschitz}
\begin{split}
\big|g(x,\zeta)-g(y,\xi)\big| \leq \sum_{i=1}^k|\zeta_i-\xi_i|\max\big(|\xi_i|,|\zeta_i|\big)^{p-1} h(\hat{x}^i,\hat{\xi}^i)
\end{split}
\end{equation}
for every $x\in D^k$, $y\in B_r(x)$ and $\xi,\zeta\in U^k$. Last, we denote
\begin{align*}
\Lgfuncs^p(D;U) := \big\{L_g\ :\ g\in \Caratheodory^{k,p}(D;U),\ k\in\N\big\} \\
\Lgfuncs^p_1(D;U) := \big\{L_g\ :\ g\in \Caratheodory^{k,p}_1(D;U),\ k\in\N\big\}
\end{align*}
where $L_g$ is defined in \eqref{eq:Lgdef}.
\end{definition}

\begin{example}\label{ex:caratheodory}
For $k=1$ the condition \eqref{eq:gbounded} asserts that
\[
|g(x,\xi)| \leq \phi_1(x) + \phi_0|\xi|^p
\]
for $0\leq \phi_1\in L^1(D)$ and $\phi_0\in[0,\infty)$, and for $k=2$ that
\[
|g(x_1,x_2,\xi_1,\xi_2)| \leq \phi_2(x_1,x_2) + \phi_1(x_1)|\xi_2|^p + \phi_1(x_2)|\xi_1|^p + \phi_0|\xi_1|^p|\xi_2|^p
\]
for $0\leq \phi_2\in L^1(D^2)$, $0\leq\phi_1\in L^1(D)$ and $\phi_0\in[0,\infty)$.
\end{example}
We will simply denote $\Caratheodory^{k,p}=\Caratheodory^{k,p}(D;U)$, etc.\ when the domain and image $D, U$ are clear from the context.

\begin{lemma}\label{lem:lgcont}
Every functional $L_g\in\Lgfuncs^p$ is well-defined and finite on $L^p(D;U)$. Every functional $L_g\in\Lgfuncs^p_1$ is continuous and is Lipschitz continuous on bounded subsets of $L^p(D;U)$.
\end{lemma}

\begin{theorem}\label{thm:weakconvergenceequivalence}
Let $\mu_n,\mu\in \Prob(L^p(D;U))$ for $n\in\N$ satisfy $\supp\mu, \supp\mu_n\subset B$ for all $n\in\N$, for some bounded set $B\subset L^p(D;U)$. Then $\mu_n\weakto\mu$ \emph{if and only if} $\Ypair{\mu_n}{F}\to\Ypair{\mu}{F}$ for all $F\in \Lgfuncs^p_1(D;U)$.
\end{theorem}
The proofs of the above results can be found in Appendix \ref{app:weakconvergenceequivalenceproof}. The ``only if'' part of \Cref{thm:weakconvergenceequivalence} is trivial, since every $F\in \Lgfuncs^p_1$ belongs to $C_b$ when restricted to a bounded set; the converse relies on an approximation argument found in \cite{ambrosio_gradient_flows}.

\subsection{Correlation measures}\label{sec:corrmeas}

In short, a \emph{correlation measure} prescribes the joint distribution of some uncertain quantity $u$ at any finite collection of spatial points $x_1,\dots,x_k$. Below we provide the rigorous definition of correlation measures and then state the result from \cite{FLM17} on the equivalence between correlation measures and probability measures on $L^p(D;U)$.

We denote $\Caratheodory^k_0(D;U) := L^1\big(D^k,C_0(U^k)\big)$. By identifying the expressions $g(x)(\xi)$ and $g(x,\xi)$, we can view $\Caratheodory^k_0(D;U)$ as a subspace of $\Caratheodory^{k,p}(D;U)$ for any $p\geq1$ (with the choice $\phi_0,\dots,\phi_{k-1}\equiv0$ and $\phi_k(x)=\|g(x)\|_{C_0(U^k)}$ in \eqref{eq:gbounded}). 

\begin{theorem}
The dual of $\Caratheodory^k_0(D;U)$ is the space $\Caratheodory^{k*}_0(D;U) := L^\infty_w\big(D^k,\Meas(U^k)\big)$, the space of bounded, weak* measurable maps from $D^k$ to $\Meas(U^k)$, under the duality pairing
\[
\Ypair{\nu^k}{g}_{\Caratheodory^k} = \int_{D^k} \Ypair{\nu^k_x}{g(x)}\,dx
\]
(where $\Ypair{\nu^k_x}{g(x)} = \int_{U^k} g(x,\xi)\,d\nu^k_x(\xi)$ is the usual duality pairing between Radon measures $\Meas(U^k)$ and continuous functions $C_b(U^k)$).
\end{theorem}
For more details and references for the above result, see \cite{ball}.

\begin{definition}[Fjordholm, Lanthaler, Mishra \cite{FLM17}]\label{def:correlationmeasure}
A \emph{correlation measure} is a collection $\cm = (\nu^1,\nu^2,\dots)$ of maps $\nu^k \in \Caratheodory^{k*}_0(D;U)$ satisfying for all $k=1,2,\dots$
\begin{enumerate}[label=\it (\roman*)]
\item $\nu^k_x \in \Prob(U^k)$ for a.e.\ $x\in D^k$, and the map $x \mapsto \Ypair{\nu^k_x}{f}$ is measurable for every $f\in C_b(U^k)$. (In other words, $\nu^k$ is a Young measure from $D^k$ to $U^k$.)
\item\textit{Symmetry:} if $\sigma$ is a permutation of $\{1,\dots,k\}$ and $f\in C_0(U^k)$ then $\Ypair{\nu^k_{\sigma(x)}}{f(\sigma(\xi))} = \Ypair{\nu^k_x}{f(\xi)}$ for a.e.\ $x\in D^k$. 
\item\textit{Consistency:} If $f\in C_b(U^k)$ is of the form $f(\xi_1,\dots,\xi_k) = g(\xi_1,\dots,\xi_{k-1})$ for some $g\in C_0(U^{k-1})$, then $\Ypair{\nu^k_{x_1,\dots,x_k}}{f} = \Ypair{\nu^{k-1}_{x_1,\dots,x_{k-1}}}{g}$ Lebesgue-a.e.~$(x_1,\dots,x_k)\in D^k$.
\item \textit{$L^p$ integrability:}
\begin{equation}\label{eq:corrlpbound}
\int_{D} \Ypair{\nu^1_x}{|\xi|^p}\,dx < +\infty.
\end{equation}
\item\textit{Diagonal continuity:} $\lim_{r\to0} \omega^p_r(\nu^2) = 0$, where
\begin{equation}\label{eq:defmodcont}
\omega_r^p(\nu^2) := \int_D \intavg_{B_r(x)} \Ypair{\nu^2_{x,y}}{|\xi_1-\xi_2|^p}\,dydx.
\end{equation}
\end{enumerate}
Each element $\nu^k$ will be called a \emph{correlation marginal}. The functional $\omega_r^p$ is called the \emph{modulus of continuity} of $\cm$. We let $\Corrmeas^{p}(D;U)$ denote the set of all correlation measures.
\end{definition}

The next result shows that there is a duality relation between correlation marginals and the probability measures $\mu\in\Prob(L^p)$ discussed in the previous section.

\begin{theorem}[Fjordholm, Lanthaler, Mishra \cite{FLM17}]\label{thm:corrmeasduality}
For every correlation measure $\cm\in\Corrmeas^p(D;U)$ there is a unique probability measure $\mu\in\Prob(L^p(D;U))$ whose $p$-th moment is finite,
\begin{equation}\label{eq:mufinitemoment}
\int_{L^p} \|u\|_{L^p}^p\,d\mu(u) < \infty
\end{equation}
and such that $\mu$ is dual to $\cm$: the identity
\begin{equation}\label{eq:corrmeasduality}
\int_{D^k}\Ypair{\nu^k}{g(x)}\,dx = \int_{L^p} \int_{D^k}g(x,u(x))\,dxd\mu(u)
\end{equation}
holds for every $g\in\Caratheodory^k_0(D;U)$ and all $k\in\N$. Conversely, for every $\mu\in\Prob(L^p(D;U))$ satisfying \eqref{eq:mufinitemoment} there is a unique correlation measure $\cm\in\Corrmeas^p(D;U)$ that is dual to~$\mu$.
\end{theorem}

\begin{remark}
By using Lebesgue's dominated convergence theorem, it is not hard to show that the identity \eqref{eq:corrmeasduality} can be extended to all $g\in \Caratheodory^{k,p}(D;U)$, as long as both integrals are well-defined. In particular, this is true if $\mu$ is supported on a bounded subset of $L^p(D;U)$.
\end{remark}

Later on, we will be particularly interested in those $\mu\in\Prob(L^p)$ that have bounded support. The following lemma shows how the property of having bounded support can be expressed in terms of the corresponding correlation measure.

\begin{lemma}
Let $\cm\in\Corrmeas^p(D;U)$ and $\mu\in\Prob(L^p(D;U))$ be dual to one another. Then
\begin{equation}\label{eq:esssupidentity}
\esssup_{u\in L^p} \|u\|_{L^p} = \limsup_{k\to\infty} \left(\int_{D^k} \Ypair{\nu^k_x}{|\xi_1|^p\cdots|\xi_k|^p}\,dx\right)^{1/kp}
\end{equation}
where the ``$\esssup$'' is taken with respect to $\mu$.
\end{lemma}
\begin{proof}
From the identity $\|f\|_{L^\infty(X;\mu)} = \lim_{k\to\infty}\|f\|_{L^k(X;\mu)}$, valid for any finite measure $\mu$, we get
\begin{align*}
\esssup_{u\in L^p} \|u\|_{L^p(D;U)}^p &= \lim_{k\to\infty}\left(\int_{L^p(D;U)} \|u\|_{L^p(D;U)}^{pk}\,d\mu(u)\right)^{1/k} \\
&= \lim_{k\to\infty}\left(\int_{L^p(D;U)} \int_{D^k} |u(x_1)|^p\cdots|u(x_k)|^p\,dx\,d\mu(u)\right)^{1/k} \\
&= \lim_{k\to\infty}\left(\int_{D^k} \Ypair{\nu^k_x}{|\xi_1|^p\cdots|\xi_k|^p}\,dx\right)^{1/k}.
\end{align*}
\end{proof}

\begin{definition}\label{def:corrmeasboundedsupp}
We let $\Corrmeas^p_b(D;U)$ denote the subset of correlation measures $\cm\in\Corrmeas^p(D;U)$ with bounded support, in the sense that there is an $M>0$ such that
\begin{equation}\label{eq:corrmeasboundedsupp}
\limsup_{k\to\infty} \left(\int_{D^k} \Ypair{\nu^k_x}{|\xi_1|^p\cdots|\xi_k|^p}\,dx\right)^{1/kp} \leq M.
\end{equation}
\end{definition}

\begin{definition}
If $\cm_n,\cm\in\Corrmeas^p(D;U)$ for $n\in\N$ then we say that $\cm_n$ converges weak* to $\cm$ as $n\to\infty$ (written $\cm_n\weaksto\cm$) if $\nu^k_n \weaksto \nu^k$ as $n\to\infty$, that is, if $\Ypair{\nu^k_n}{g}_{\Caratheodory^k}\to\Ypair{\nu^k}{g}_{\Caratheodory^k}$ for all $g\in\Caratheodory^k_0(D;U)$ and all $k\in\N$.

If $\cm_n,\cm\in\Corrmeas^p_b(D;U)$ for $n\in\N$ then we say that $(\cm_n)_{n\in\N}$ converges weakly to $\cm$ as $n\to\infty$ (written $\cm_n\weakto\cm$) if $\Ypair{\nu^k_n}{g}_{\Caratheodory^k}\to\Ypair{\nu^k}{g}_{\Caratheodory^k}$ for every $g\in\Caratheodory^{k,p}_1(D;U)$.
\end{definition}

Note that $\cm\in\Corrmeas^p_b$ implies that $\Ypair{\nu^k}{g}_{\Caratheodory^k}$ is well-defined and finite for any $g\in\Caratheodory^{k,p}$ (cf.~Definition \ref{def:caratheodory}).

We next show a compactness result which can be thought of as Kolmogorov's compactness theorem (cf.~\cite[Theorem A.5]{HR}) for correlation measures.

\begin{theorem}\label{thm:cmcompactness}
Let $\cm_n \in \Corrmeas^{p}(D;U)$ for $n=1,2,\dots$ be a sequence of correlation measures such that
\begin{align}
\sup_{n\in\N}\Ypair{\nu^1_n}{|\xi|^p}_{\Caratheodory^1} &\leq c^p	\label{eq:uniformlpbound} \\
\lim_{r\to0}\limsup_{n\to\infty}\omega_r^p\bigl(\nu^2_n\bigr) &= 0 \label{eq:uniformdc}
\end{align}
for some $c>0$ (where $\omega_r^p$ is defined in Definition \ref{def:correlationmeasure}(v)). Then there exists a subsequence $(n_j)_{j=1}^\infty$ and some $\cm\in\Corrmeas^{p}(D;U)$ such that
\begin{enumerate}[label=\it (\roman*)]
\item\label{prop:wsconv} $\cm_{n_j} \weaksto \cm$ as $j\to\infty$, that is, $\Ypair{\nu^k_{n_j}}{g}_{\Caratheodory^k} \to \Ypair{\nu^k}{g}_{\Caratheodory^k}$ for every $g\in\Caratheodory^k_0(D;U)$ and every $k\in\N$
\item\label{prop:lpbound} $\Ypair{\nu^1}{|\xi|^p}_{\Caratheodory^1} \leq c^p$
\item\label{prop:dc} $\omega_r^p(\nu^2) \leq \liminf_{n\to\infty} \omega_r^p(\nu^2_n)$ for every $r>0$
\item\label{prop:liminf} for $k\in\N$, let $\phi\in L^1_{\loc}(D^k)$ and $\kappa\in C(U^k)$ be nonnegative, and let $g(x,\xi) := \phi(x)\kappa(\xi)$. Then
\begin{equation}\label{eq:nuklimitbound}
\Ypair{\nu^k}{g}_{\Caratheodory^k} \leq \liminf_{j\to\infty} \Ypair{\nu_{n_j}^k}{g}_{\Caratheodory^k}.
\end{equation}
\item\label{prop:strongconv}
Assume moreover that the domain $D\subset\R^d$ is bounded and that $\cm_n$ have uniformly bounded support, in the sense that \eqref{eq:corrmeasboundedsupp} holds for all $\cm_n$ for a fixed $M>0$, or equivalently,
\begin{equation}
\|u\|_{L^p} \leq M \qquad \text{for $\mu_n$-a.e.\ } u\in L^p(D;U) \text{ for every } n\in\N. \label{eq:LpLinftybound}
\end{equation}
Then \emph{observables converge strongly:}
\begin{equation}\label{eq:nuklimit}
\lim_{j\to\infty}\int_{D^k} \big|\Ypair{\nu^k_{n_j,x}}{g(x)} - \Ypair{\nu^k_{x}}{g(x)}\big| \,dx = 0
\end{equation}
for every $g \in \Caratheodory^{k,p}_1(D;U)$. In particular, $\cm_{n_j} \weakto \cm$.
\end{enumerate}
\end{theorem}

The proof is given in Appendix \ref{app:proofcompactness}.
\begin{remark}
\eqref{eq:nuklimit} implies in particular that $\Ypair{\nu^k_{n_j}}{g}_{\Caratheodory^k} = \Ypair{\mu_{n_j}}{L_g}$ converges for any $g\in\Caratheodory^{k,p}_1$, where $\mu_n\in\Prob(L^p)$ is dual to $\cm_n$ (see Theorem~\ref{thm:corrmeasduality}). By Theorem~\ref{thm:weakconvergenceequivalence}, this is equivalent to saying that $\mu_{n_j}$ converges weakly to $\mu$. Since, by hypothesis, the $p$th moment of $\mu_n$ is uniformly bounded, the sequence $\mu_n$ converges to $\mu$ in the Wasserstein distance; see Definition \ref{def:wasserstein} and \cite[Chapter 7]{Vil1}.
\end{remark}

\begin{remark}
Theorem \ref{thm:cmcompactness} can most likely be extended to provide a complete characterization of compact subsets of $\Corrmeas^p(D;U)$. Since we only require sufficient conditions for compactness, we do not pursue this generalization here.
\end{remark}

\subsection{Time-parameterized probability measures on $L^p$}\label{sec:timedeppromeas}


Let $T\in (0,\infty]$. To take into account the evolutionary nature of the PDE \eqref{eq:cl}, we will add time-dependence to the probability measures considered in Section \ref{sec:probmeasLp} by considering maps $\mu : [0,T) \to \Prob(L^p(D;U))$.
Note the distinction between time-parametrized maps $\mu : [0,T)\to\Prob(L^p(D;U))$ and probability measures $\gamma$ on, say, the space $L^\infty([0,T); L^p(D;U))$. Every such measure $\gamma$ would correspond to a unique $\mu$, but not \emph{vice versa}; when ``projecting'' $\gamma$ onto $\mu$, any information about correlation between function values $u(t_1)$, $u(t_2)$ at different times $t_1$, $t_2$ is lost. Given the evolutionary nature of the PDE \eqref{eq:cl}, we have chosen to work with ``$\mu$'' measures in order to preserve the direction of time in the underlying PDE.  



\begin{notation}
We denote the set of Carath\'eodory functions depending on space and time by $\Caratheodory^k_0([0,T),D;U) := L^1([0,T)\times D^k;C_0(U^k))$ and its dual space by $\Caratheodory^{k*}_0([0,T),D;U) := L^\infty_w([0,T)\times D^k;\Meas(U^k))$.
\end{notation}
Analogously to Definition~\ref{def:caratheodory}, we let $\Caratheodory^{k,p}([0,T),D;U)$ denote the space of Carath\'eodory functions $g:[0,T)\times D^k\times U^k \to \R$ satisfying
	\begin{equation}\label{eq:gboundedt}
|g(t,x,\xi)| \leq \sum_{\alpha\in\{0,1\}^k} \phi_{|\bar{\alpha}|}(t,x_{\bar{\alpha}})|\xi^\alpha|^p \qquad \forall\ x\in D^k,\ \xi\in U^k
\end{equation}
for nonnegative functions $\phi_i\in L^\infty([0,T); L^1(D^i))$, $i=0,1,\dots,k$. We let $\Caratheodory^{k,p}_1([0,T),D;U)\subset \Caratheodory^{k,p}([0,T),D;U)$ denote the subspace of functions $g$ satisfying the local Lipschitz condition
\begin{equation}\label{eq:glipschitzt}
\begin{split}
\big|g(t,x,\zeta)-g(t,y,\xi)\big| \leq \psi(t)\sum_{i=1}^k|\zeta_i-\xi_i|\max\big(|\xi_i|,|\zeta_i|\big)^{p-1} h(t,\hat{x}^i,\hat{\xi}^i)
\end{split}
\end{equation}
for every $x\in D^k$, $y\in B_r(x)$ for some $r>0$, for some nonnegative $h \in \Caratheodory^{k-1,p}([0,T),D;U)$ and $0\leq \psi(t)\in L^{\infty}([0,T))$.

The following lemma shows that it is meaningful to ``evaluate'' an element $\nu^k \in \Caratheodory^{k*}_0([0,T),D;U)$ at (almost) any time $t\in[0,T)$.
 \begin{lemma}\label{lem:cmtimeslice}
Let $\nu^k \in \Caratheodory^{k*}_0([0,T),D;U)$. Then there exists a map $\rho:[0,T)\to\Caratheodory^{k*}_0(D;U)$, uniquely defined for a.e.~$t\in[0,T)$, such that $t \mapsto \Ypair{\rho(t)}{g}_{\Caratheodory^k_0}$ is measurable for all $g\in\Caratheodory^k_0(D;U)$, and
\[
\Ypair{\nu^k}{g}_{\Caratheodory^k} = \int_0^T \Ypair{\rho(t)}{g(t,\cdot)}_{\Caratheodory^k}\,dt \qquad \forall\ g\in\Caratheodory^k_0([0,T),D;U).
\] 
\end{lemma}
The proof of this lemma is given in Appendix \ref{app:time-dep-cm}. Henceforth we will not make distinctions between these two representations of elements of $\Caratheodory^{k*}_0([0,T),D;U)$, and denote them both by $\nu^k$.

\begin{definition}\label{def:timdeptrunccorrmeas}
A \emph{time-dependent correlation measure} is a collection $\cm=(\nu^1,\nu^2,\dots)$ of functions $\nu^k\in \Caratheodory^{k*}_0([0,T),D;U)$ such that
\begin{enumerate}[label=\it (\roman*)]
\item $(\nu^1_t,\nu^2_t,\dots) \in \Corrmeas^p(D;U)$ for a.e.~$t\in[0,T)$
\item \textit{$L^p$ integrability:}
\begin{equation}\label{eq:corrlpboundtimedep}
\esssup_{t\in[0,T)}\int_{D} \Ypair{\nu^1_{t,x}}{|\xi|^p}\,dx \leq c^p < +\infty
\end{equation}
\item\textit{Diagonal continuity (DC):}
\begin{equation}\label{eq:timedepdc}
\int_0^{T'}\omega_r^p\big(\nu^2_t\big)\,dt \to 0 \qquad \text{ as } r\to0 \text{ for all } T'\in(0,T)
\end{equation}
where $\omega_r^p$ was defined in \eqref{eq:defmodcont}.
\end{enumerate}
We denote the set of all time-dependent correlation measures by $\Corrmeas^{p}([0,T),D;U)$.
\end{definition}

\begin{remark}
By Lemma \ref{lem:cmtimeslice}, the objects $\nu^k_t$ are well-defined for a.e.~$t\in[0,T)$. Assertion (ii) requires that the $L^p$ bound should be uniform in $t$, and assertion (iii) requires that the modulus of continuity in the diagonal continuity requirement should be integrable in $t$.
\end{remark}

Next, we prove a time-dependent version of the duality result Theorem \ref{thm:corrmeasduality}. 
\begin{theorem}
For every time-dependent correlation measure $\cm\in\Corrmeas^p([0,T),D;U)$ there is a unique {\rm(}up to subsets of $[0,T)$ of Lebesgue measure $0${\rm)} map $\mu:[0,T)\to\Prob(L^p(D;U))$ such that
\begin{enumerate}[label=\it (\roman*)]
\item the map
\begin{equation}\label{eq:mumeasurable}
t\mapsto\Ypair{\mu_t}{L_g} = \int_{L^p} \int_{D^k}g(x,u(x))\,dxd\mu_t(u)
\end{equation}
is measurable for all $g\in\Caratheodory^k_0(D;U)$,
\item $\mu$ is $L^p$-bounded:
\begin{equation}\label{eq:mulpbound}
\esssup_{t\in[0,T)}\int_{L^p}\|u\|_{L^p}^p\,d\mu_t(u) \leq c^p < \infty
\end{equation}
\item $\mu$ is dual to $\cm$: the identity
\begin{equation}\label{eq:corrmeastimeduality}
\int_{D^k}\Ypair{\nu_t^k}{g(x)}\,dx = \int_{L^p} \int_{D^k}g(x,u(x))\,dxd\mu_t(u)
\end{equation}
holds for a.e.~$t\in[0,T)$, every $g\in\Caratheodory^k_0(D;U)$ and all $k\in\N$. 
\end{enumerate}
Conversely, for every $\mu:[0,T)\to\Prob(L^p(D;U))$
 satisfying (i) and (ii), there is a unique correlation measure $\cm\in\Corrmeas^p([0,T),D;U)$ satisfying (iii).
\end{theorem}
\begin{proof}
Let $\cm$ be given. Then for a.e.~$t\in[0,T)$ we have $\cm_t:=(\nu^1_t,\nu^2_t,\dots)\in \Corrmeas^p(D;U)$, so by Theorem \ref{thm:corrmeasduality} there exists a unique $\mu_t\in\Prob(L^p(D;U))$ that is dual to $\cm_t$, in the sense that (iii) holds. From the previous remark we know that $t\mapsto\Ypair{\nu^k_t}{g}_{\Caratheodory^k}$ is measurable for every $g\in\Caratheodory^k_0(D;U)$, which (using (iii)) is precisely (i). Property (ii) follows by approximating $\xi\mapsto|\xi|^p$ by functions in $C_0(U)$.

Conversely, given $\mu$ satisfying (i) and (ii), Theorem \ref{thm:corrmeasduality} gives, for a.e.~$t\in[0,T)$, the existence and uniqueness of $\cm_t\in\Corrmeas^p(D;U)$ satisfying (iii) as well as the $L^p$-bound \eqref{eq:corrlpboundtimedep}. We claim that $(\cm_t)_{t\in[0,T)}$ defines a time-dependent correlation measure $\cm\in\Corrmeas^p([0,T),D;U)$. Indeed, define the linear functional $\nu^k$ by
\[
\Ypair{\nu^k}{\theta\otimes g}_{\Caratheodory^k} := \int_0^T \theta(t)\Ypair{\nu^k_t}{g}_{\Caratheodory^k}\,dt \qquad \forall\ \theta\in L^1([0,T)),\ g\in \Caratheodory^k_0(D;U),\ k\in\N.
\]
Then $\nu^k$ is well-defined on tensor product test functions $\theta(t)g(x)$, and
\begin{align*}
\big|\Ypair{\nu^k}{\theta\otimes g}_{\Caratheodory^k}\big| &\leq \|\theta\|_{L^1([0,T))}\big\|\Ypair{\mu_\cdot}{L_g}\big\|_{L^\infty([0,T))}
\leq \|\theta\|_{L^1}\|L_g\|_{C^0(L^p)} \\
&= \|\theta\|_{L^1}\|g\|_{\Caratheodory^k_0}
= \|\theta \otimes g\|_{L^1([0,T)\times D^k;C_0(U^k))}.
\end{align*}
Extending $\nu^k$ by linearity to all of $L^1([0,T)\times D^k;C_0(U^k))$ produces a unique element $\nu^k \in L^1([0,T)\times D^k;C_0(U^k))^* \cong L^\infty_w([0,T)\times D^k;\Meas(U^k))$. Defining the collection $\cm = (\nu^1,\nu^2,\dots)$, it only remains to show that $\nu^2$ satisfies the diagonal continuity requirement \eqref{eq:timedepdc}. Indeed, since 
\[
\omega_r^p\big(\nu^2_t\big) = \int_{L^p}\omega_r^p(u)\,d\mu_t(u) \to 0 \qquad \text{as } r\to0
\]
for a.e.~$t\in[0,T)$, the requirement \eqref{eq:timedepdc} follows from the dominated convergence theorem.
\end{proof}

We denote the set of all maps $\mu:[0,T)\to\Prob(L^p(D;U))$ that are dual to some $\cm\in\Corrmeas^{p}([0,T),D;U)$ as $\Prob_T(L^p(D;U))$.

We conclude this section by proving a version of the compactness theorem for time-dependent correlation measures.

\begin{theorem}\label{thm:timedepcmcompactness}
Let $\cm_n \in \Corrmeas^{p}([0,T),D;U)$ for $n=1,2,\dots$ be a sequence of correlation measures such that
\begin{align}
\sup_{n\in\N}\esssup_{t\in[0,T)}\int_{D} \Ypair{\nu^1_{n;t,x}}{|\xi|^p}\,dx \leq c^p < +\infty	\label{eq:timedepuniformlpbound} \\
\lim_{r\to0}\limsup_{n\to\infty}\int_0^{T'}\omega_r^p\bigl(\nu^2_{n,t}\bigr)\,dt = 0 \label{eq:timedepuniformdc}
\end{align}
for some $c>0$ and all $T'\in[0,T)$. Then there exists a subsequence $(n_j)_{j=1}^\infty$ and some $\cm\in\Corrmeas^{p}([0,T),D;U)$ such that
\begin{enumerate}[label=\it (\roman*)]
\item\label{prop:tdwsconv} $\cm_{n_j} \weaksto \cm$ as $j\to\infty$, that is, $\Ypair{\nu^k_{n_j}}{g}_{\Caratheodory^k} \to \Ypair{\nu^k}{g}_{\Caratheodory^k}$ for every $g\in\Caratheodory^k_0([0,T),D;U)$ and every $k\in\N$
\item\label{prop:tdlpbound} $\Ypair{\nu^1_t}{|\xi|^p}_{\Caratheodory^1} \leq c^p$ for a.e.~$t\in[0,T)$
\item\label{prop:tddc} $\int_0^{T'}\omega_r^p\big(\nu^2_t\big)\,dt \leq \liminf_{n\to\infty} \int_0^{T'}\omega_r^p\big(\nu^2_{n,t}\big)\,dt$ for every $r>0$ and $T'\in[0,T)$
\item\label{prop:tdliminf} for $k\in\N$, let $\phi\in L^1_{\textrm{loc}}([0,T)\times D^k)$ and $\kappa\in C(U^k)$ be nonnegative, and let $g(t,x,\xi) := \phi(t,x)\kappa(\xi)$. Then
\begin{equation}\label{eq:tdnuklimitbound}
\Ypair{\nu^k}{g}_{\Caratheodory^k} \leq \liminf_{j\to\infty} \Ypair{\nu_{n_j}^k}{g}_{\Caratheodory^k}.
\end{equation}
\item \label{prop:strongconv_time}
Assume moreover that $D\subset\R^d$ is bounded, $T<\infty$ and that $\cm_n$ have uniformly bounded support, in the sense that
\begin{equation}
\|u\|_{L^p} \leq M \qquad \text{for $\mu^n_t$-a.e.\ } u\in L^p(D;U) \text{ for every } n\in\N, ~ a.e~t \in (0,T),  \label{eq:tLpLinftybound}
\end{equation}
with $\mu_t^n\in \Prob_T(L^p(D;U))$ being dual to $\cm_n$,
then the following \emph{observables converge strongly:}
\begin{equation}\label{eq:tnuklimit}
\lim_{j\to\infty}\int_{D^k} \left|\int_0^T\Ypair{\nu^k_{n_j;t,x}-\nu^k_{t,x}}{g(t,x)} \,dt\right| \,dx = 0
\end{equation}
for every $g \in \Caratheodory^{k,p}_1([0,T),D;U)$.
\end{enumerate}
\end{theorem}

We skip the proof of this theorem as is very similar to that of Theorem \ref{thm:cmcompactness}.
\begin{remark}
A closer look at the convergence statement \eqref{eq:tnuklimit} reveals that we can expect pointwise a.e convergence in space of the ensemble averages of the observable $g \in \Caratheodory^{k,p}_1([0,T),D;U)$. On the other hand, time averaging in \eqref{eq:tnuklimit} seems essential. In other words, we have convergence of time averages of ensemble averages of the observables. 
\end{remark}


\subsection{Statistical solutions}\label{sec:statsoln}
Using correlation measures we can now define statistical solutions of \eqref{eq:cl}. We need the following assumptions on the flux function in \eqref{eq:cl},
\begin{equation}
\label{eq:aflux}
\begin{aligned}
|f(u)| &\leq C(1+|u|^p) \qquad \forall\ u \in U, \\
|f(u) - f(v)| &\leq C|u-v|\max\left(|u|,|v|\right)^{p-1}, \qquad \forall\ u,v \in U.
\end{aligned}
\end{equation}
for some constant $C > 0$ and $1 \leq p < \infty$. The value of $p$ is given by available \emph{a priori} bounds for solutions of \eqref{eq:cl}, for instance, from the entropy condition \eqref{eq:entrcond} (cf.~\eqref{eq:lpbound}). For example, both the shallow water equations and the isentropic Euler equations are $L^2$-bounded, at least for solutions away from vacuum \cite{Dafermos}.

Statistical solutions are correlation measures (or equivalently, probability measures over $L^p$) satisfying the differential equation \eqref{eq:cl} in a certain averaged sense. The full derivation can be found in \cite{FLM17}, and we only provide the definition here.
\begin{definition}
\label{def:statsoln}
Let $\bar{\mu}\in\Prob\big(L^p\big(D;U)\big)$ have bounded support,
\begin{equation}\label{eq:fastdecay}
\|u\|_{L^p(D;U)} \leq M \qquad \text{for $\bar{\mu}$-a.e. } u\in L^p(D;U)
\end{equation}
for some $M>0$. A \emph{statistical solution} of \eqref{eq:cl} with initial data $\bar{\mu}$ is a time-dependent map $\mu:[0,T) \mapsto \Prob(L^p(D;U))$ such that each $\mu_t$ has bounded support, and such that the corresponding correlation measures $(\nu^k_t)_{k\in\N}$ satisfy 
\begin{equation}\label{eq:momentcorrmeas}
\partial_t \Ypair{\nu^k_{t,x}}{\xi_1\otimes\cdots\otimes\xi_k} + \sum_{i=1}^k \nabla_{x_i}\cdot\Ypair{\nu^k_{t,x}}{\xi_1\otimes\cdots\otimes f(\xi_i)\otimes\cdots\otimes\xi_k} = 0
\end{equation}
in the sense of distributions, i.e.,
\begin{align*}
\int_{\R_+}\int_{D^k} \Ypair{\nu^k_{t,x}}{\xi_1\otimes\cdots\otimes\xi_k}:\partial_t \phi + \sum_{i=1}^k \Ypair{\nu^k_{t,x}}{\xi_1\otimes\cdots\otimes f(\xi_i)\otimes\cdots\otimes\xi_k}: \nabla_{x_i}\phi\,dxdt \\
+ \int_{D^k}\Ypair{\bar{\nu}^k_x}{\xi_1\otimes\cdots\otimes\xi_k}:\phi\bigr|_{t=0}\,dx = 0
\end{align*}
for every $\phi\in C_c^\infty\big(D^k\times\R_+,\ U^{\otimes k}\big)$ and for every $k\in\N$. (Here, $\bar{\cm}$ denotes the correlation measure associated with the initial probability measure $\bar{\mu}$.)
\end{definition}
\begin{remark}
If the initial data $\bar{\mu}$ and a resulting statistical solution $\mu_t$ are both atomic, i.e.~$\bar{\mu} = \delta_{\bar{u}}$ and $\mu_t = \delta_{u}$ with $\bar{u} \in L^p(D;U)$ and $u \in L^p((0,T) \times D; U)$, then it is easy to see that a statistical solution in the above sense reduces to a weak solution of \eqref{eq:cl}. Thus, weak solutions are statistical solutions. 
\end{remark}

\begin{remark}
The evolution equation for the first correlation marginal of the statistical solution, i.e.,~for $k=1$ in \eqref{eq:momentcorrmeas}, is equivalent to the definition of a measure-valued solution of \eqref{eq:cl} \cite{diperna,fkmt}. Thus, a statistical solution can be thought of as a measure-valued solution augmented with information about all possible multi-point correlations. Hence, \emph{a priori}, a statistical solution contains significantly more information than a measure-valued solution. 
\end{remark}

\section{Dissipative Statistical solutions and weak-strong uniqueness}
\label{sec:3}
In analogy with weak solutions, it is necessary to impose additional admissibility criteria for statistical solutions in order to ensure uniqueness and stability. In \cite{FLM17}, the authors proposed an entropy condition for statistical solutions of scalar conservation laws. This condition was based on a non-trivial generalization of the Kruzkhov entropy condition to the framework of time-parameterized probability measures on $L^1(D)$. It was shown in \cite{FLM17} that these \emph{entropy statistical solutions} were unique and stable in the 1-Wasserstein metric on $\Prob(L^1(D))$, with respect to perturbations of the initial data.

Although one can extend the entropy condition of \cite{FLM17} to statistical solutions for systems of conservation laws \eqref{eq:cl}, it is not possible to obtain uniqueness and stability of such entropy statistical solutions. Instead, one has to seek alternative notions of stability for systems of conservation laws. 

A possible weaker framework for uniqueness (stability) is that of \emph{weak-strong uniqueness}, see~\cite{Wiedemann_ws,Dafermos} and references therein. Within this framework, one imposes certain \emph{entropy conditions} and proves that the resulting entropy solutions will coincide with a strong (classical) solution if such a solution exists. Weak-strong uniqueness for systems of conservation laws with strictly convex entropy functions is shown in~\cite{Dafermos}. In fact, one can even prove weak-strong uniqueness results for the much weaker notion of entropy or dissipative measure-valued solutions of systems of conservation laws, see~\cite{DST1,BDS1,fkmt}. 

Our aim in this section is to propose a suitable notion of \emph{dissipative statistical solution}s and prove a \emph{weak-strong} uniqueness result for such solutions.  Stability of solutions will be measured in the \emph{Wasserstein distance}, whose definition we recall first.
\begin{definition}
\label{def:wasserstein}
Let $X$ be a separable Banach space and let $\mu, \rho\in\Prob(X)$ have finite $p$th moments, i.e.\ $\int_X |x|^p d\mu(x) < \infty$ and $\int_X |x|^p d\rho(x) < \infty$. The $p$-Wasserstein distance between $\mu$ and $\rho$ is defined as
\begin{equation}\label{eq:wasserdef}
W_p(\mu,\rho) = \left(\inf_{\pi\in\Pi(\mu,\rho)}\int_{X^2} |x-y|^p\,d\pi(x,y)\right)^{\frac{1}{p}};
\end{equation}
where the infimum is taken over the set $\Pi(\mu,\rho)\subset\Prob(X^2)$ of all transport plans from $\mu$ to $\rho$, i.e.\ those $\pi\in\Prob(X^2)$ satisfying
\[
\int_{X^2} F(x)+G(y)\,d\pi(x,y) = \int_X F(x)\,d\mu(x) + \int_X G(y)\,d\rho(y) \qquad \forall\ F,G\in C_b(X)
\]
(see e.g.~\cite{Vil1}).
\end{definition}

As in \cite[Section 4]{FLM17}, our \emph{entropy condition} for statistical solutions will rely on a comparison with probability measures that are convex combinations of Dirac masses, i.e.~$\rho \in \Prob(L^2(D))$ such that  $\rho = \sum_{i=1}^M \alpha_i \delta_{u_i}$ for coefficients $\alpha_i\geq0$, $\sum_i\alpha_i=1$ and functions $u_1,\dots,u_M\in L^2(D)$. From \cite[Lemma 4.2]{FLM17}, we observe that whenever $\rho$ is of this \emph{$M$-atomic form}, there is a one-to-one correspondence between transport plans $\pi\in\Pi(\mu,\rho)$ and elements of the set
\begin{align*}
\Lambda(\alpha,\mu) := \Bigl\{(\mu_1,\dots,\mu_M)\ :\ \mu_1,\dots,\mu_M\in \Prob(L^2(D;U)) \text{ and } \textstyle\sum_{i=1}^M \alpha_i\mu_i = \mu \Bigr\},
\end{align*}
defined for any $\alpha=(\alpha_1,\dots,\alpha_M)\in\R^M$ satisfying $\alpha_i\geq0$ and $\sum_{i=1}^M\alpha_i=1.$ The set $\Lambda(\alpha,\mu)$ is never empty since $(\mu,\dots,\mu) \in \Lambda(\alpha,\mu)$ for any choice of coefficients $\alpha_1,\dots,\alpha_M$. Note that the set $\Lambda(\alpha,\mu)$ depends on the target measure $\rho$ \emph{only} through the weights $\alpha_1,\dots,\alpha_M$.

Using this decomposition of transport plans with respect to M-atomic probability measures, we define the notion of dissipative statistical solution as follows.
\begin{definition}\label{def:dss}
Assume that the system of conservation laws \eqref{eq:cl} is equipped with an entropy function $\eta$. A statistical solution $\mu$ of \eqref{eq:cl} is a \emph{dissipative statistical solution} if 
\begin{enumerate}
\item[(i)] for every choice of coefficients $\alpha_1,\dots,\alpha_M>0$ with $\sum_{i=1}^M \alpha_i=1$ and for every  $(\bar{\mu}_1,\dots,\bar{\mu}_M)\in\Lambda(\alpha,\bar{\mu})$, there exists a function $t \mapsto(\mu_{1,t},\dots,\mu_{M,t})\in\Lambda(\alpha,\mu_t)$, such that each measure $\mu_{i} \in \Prob_T(L^p(D;U))$ is a statistical solution of \eqref{eq:cl} with initial data $\bar{\mu}_i$,
\item[(ii)] for all test functions $0 \leq \theta(t) \in C_c^\infty(\R_+)$,
\begin{equation}
\label{eq:dss3}
\int_{\R_+}\int_{L^p(D,U)}\int_D \eta(u(x)) \theta'(t)\, dx d\mu_t(u) dt + \int_{L^p(D,U)}\int_D \eta(\bar{u}(x)) \theta(0) dx d\bar{\mu}(\bar{u}) \geq 0.
\end{equation}

\end{enumerate}
\end{definition}
We remark that the first condition in the above definition demands that the decomposition of a statistical solution into the components $\mu_{i}$ is still consistent with the underlying conservation law \eqref{eq:cl}. On the other hand, the second condition \eqref{eq:dss3} amounts to requiring that the total entropy of $\mu$ decreases in time. 

First, we investigate the stability of a dissipative statistical solution of \eqref{eq:cl} with respect to statistical solutions built from finitely many classical solutions of~\eqref{eq:cauchy}.
\begin{lemma}\label{lem:ws1}
Let $T>0$, set $p=2$, assume that
\begin{equation}\label{eq:squareflux}
\|f''\|_{L^\infty(\R^N)} < \infty
\end{equation}
{\rm(}where we denoted by $f''$ the Hessian of $f$, i.e. $(f''(u))_{ijk}=\partial_{u^j}\partial_{u^k} f^i(u)$, $i,j,k=1,\dots,N${\rm)},
and assume that the conservation law \eqref{eq:cl} is equipped with an entropy pair $(\eta,q)$ for which
\begin{equation}\label{eq:squareentropy}
c \leq (\eta''(u)v,v) \leq C \qquad \forall\ u\in\R^N,\, v\in \R^N\,\text{with}\quad |v|=1
\end{equation}
{\rm(}where $\eta''$ denotes the Hessian matrix of $\eta${\rm)} for $c,C>0$. Let $\mu\in\Prob_T(L^2(D;U))$ be a dissipative statistical solution of \eqref{eq:cl}, and for $t\in[0,T)$ let $\rho_t = \sum_{i=1}^M \alpha_i\delta_{v_i(t)}$ for coefficients $\alpha_i>0$, $\sum_{i=1}^M\alpha_i=1$, and classical solutions $v_1,\dots,v_M\in W^{1,\infty}(D\times[0,T);U)$ of \eqref{eq:cl}. Then
\begin{equation}\label{eq:dssest}
W_2(\mu_t,\rho_t) \leq e^{Ct} W_2(\mu_0,\rho_0) \qquad \forall\ t\in[0,T),
\end{equation}
where $C=C(R)\geq0$ is a constant only depending on $R:=\max_{i=1,\dots,M}\|v_i\|_{W^{1,\infty}(D\times\R_+,U)}.$
\end{lemma}
\begin{proof}
It is straightforward to verify that $\rho_t$ as defined above is a statistical solution of \eqref{eq:cl} with initial data $\bar{\rho}:=\sum_{i=1}^M\alpha_i\delta_{\bar{v}_i}$, where $\bar{v}_i:=v_i(0)$.

Let $\bar{\mu}^{\ast}=(\bar{\mu}_1^{\ast},\dots,\bar{\mu}_M^{\ast}) \in \Lambda(\alpha,\bar{\mu})$ define a transport plan that minimizes the transport cost between $\bar{\mu}:=\mu_0$ and $\bar{\rho}$, that is,
\begin{equation}
\label{eq:pf1}
W_2(\bar{\mu},\bar{\rho}) = \Biggl(\sum_{i=1}^M \alpha_i \int_{L^2} \|u-\bar{v}_i\|_{L^2}^2 d\bar{\mu}_i^{\ast}(u)\Biggr)^{\frac{1}{2}}.
\end{equation}
(Here and in the remainder of this proof, we denote $L^2 = L^2(D;U)$.) As $\mu_t$ is a dissipative statistical solution, there exists a map $t \mapsto \big(\mu^{\ast}_{1,t},\dots,\mu^{\ast}_{M,t}\big) \in \Lambda(\alpha,\mu_t)$ such that 
\begin{multline}\label{eq:pf2}
\sum_{i=1}^M\alpha_i\Biggl(\int_0^T\int_{L^2}\int_{D} u(x) \partial_t\phi_i(x,t) + f(u(x)) \cdot\nabla_x \phi_i(x,t) dx d\mu^{\ast}_{i,t}(u)dt \\
+ \int_{L^2}\int_D \bar{u} \phi_i(x,0) dxd\bar{\mu}^{\ast}_i(\bar{u})\Biggr)  = 0
\end{multline}
for every $\phi_1,\dots,\phi_M\in C_c^\infty(D\times[0,T))$. 
For each $1 \leq i \leq M$, we have that 
\begin{align*}
&\int_0^T\!\!\!\int_{L^2}\!\int_{D} v_i(x,t) \partial_t \phi_i + f(v_i(t,x))\cdot\nabla_x \phi_i dx d\mu^{\ast}_{i,t}(u) dt + \!
\int_{L^2}\!\int_{D} \bar{v}_i(x) \phi_i(x,0) dx d\bar{\mu}^{\ast}_i(\bar{u}) \\
&= \int_0^T\int_{L^2}\int_{D} \partial_t(v_i \phi_i) + \nabla_x\cdot (f(v_i)\phi_i) dx d\mu^{\ast}_{i,t}(u) dt + \int_{L^2}\int_{D} \bar{v}_i(x) \phi_i(x,0) dx d\bar{\mu}^{\ast}_i(\bar{u}) \\
&\quad - \underbrace{\int_0^T\int_{L^2}\int_{D} \phi_i \big(\partial_t v_i + \nabla_x\cdot f(v_i)\big) dx d\mu^{\ast}_{i,t}(u) dt}_{\text{$=0$, as $v_i$ is a classical solution of \eqref{eq:cl}}} \\
&= - \int_D v_i(x,0) \phi_i(x,0) dx + \int_D \bar{v}_i(x) \phi_i(x,0) dx = 0.
\end{align*}
Multiplying the above with $\alpha_i$ and summing over $i$, we obtain
\begin{equation}\label{eq:dss5}
\begin{split}
\sum_{i=1}^M\alpha_i\Bigg(\int_0^T\int_{L^2}\int_{D} v_i \partial_t\phi_i + f(v_i)\cdot \nabla_x \phi\, dx d\mu_{i,t}^{\ast}(u)dt \\
+\int_{L^2}\int_{D} \bar{v}_i \phi_i(x,0) dxd\bar{\mu}_i^{\ast}(\bar{u})\Bigg)  = 0.
\end{split}
\end{equation}
Subtracting \eqref{eq:dss5} from \eqref{eq:pf2} and choosing as a test function the vector-valued function $\phi_i = \eta'(v_i(x,t)) \theta(t)$ for some scalar test function $0 \leq \theta(t) \in C_c^\infty(\R_+)$ (here, $\eta'$ denotes the vector-valued derivative of $\eta$ with respect to $u$), and using the fact that
\begin{align*}
\partial_t \phi_i &= \eta^{\prime}(v_i)\theta^{\prime}(t) + \theta(t) \eta^{\prime \prime}(v_i) \partial_t v_i = \eta^{\prime}(v_i)\theta^{\prime}(t)  - \theta(t)  f^{\prime}(v_i)\cdot\nabla_x\eta^{\prime}(v_i), \\
\partial_j \phi_i &= \theta(t)  \partial_j \eta^{\prime} (v_i)
\end{align*}
yields
\begin{equation}
\label{eq:pf3}
\begin{split}
0=&~\sum_{i=1}^M\alpha_i\Bigg(\int_0^T\int_{L^2}\int_{D} (u-v_i)\cdot\partial_t\phi_i + (f(u) - f(v_i)) \cdot \nabla_x \phi_i\, dx d\mu^{\ast}_{i,t}(u)dt \\
& +\int_{L^2}\int_{D} (\bar{u}-\bar{v}_i)\cdot\phi_i(x,0)\, dxd\bar{\mu}^{\ast}_i(\bar{u})\Bigg) \\
=&~ \sum_{i=1}^M\alpha_i\Bigg(\int_0^T\int_{L^2}\int_{D} \eta^{\prime}(v_i)\cdot(u-v_i) \theta^{\prime}(t)\, dx d\mu^{\ast}_{i,t}(u)dt \\
&+ \int_{L^2}\int_{D} \eta^{\prime}(\bar{v}_i)\cdot(\bar{u}-\bar{v}_i) \theta(0)\, dxd\bar{\mu}^{\ast}_i(\bar{u}) \\
& + \int_0^T\int_{L^2}\int_{D} \theta(t) \underbrace{\big(f(u) - f(v_i)-f'(v_i)(u-v_i)\big)\cdot \nabla_x \eta^{\prime}(v_i)}_{=:\,\mathcal{Z}(u|v_i)}\, dx d\mu^{\ast}_{i,t}(u)dt\Bigg)
\end{split}
\end{equation}
As $v_i$ is a classical solution of \eqref{eq:cl} and ${\mu}_{1,t}^{\ast},\dots,{\mu}_{M,t}^{\ast}$ are probability measures, we obtain from the entropy conservation of $v_i$ that
\begin{equation}\label{eq:pf4}
\sum_{i=1}^M\alpha_i\Bigg(\int_0^T\int_{L^2}\int_{D}
\eta(v_i) \theta^{\prime}(t) dx d\mu^{\ast}_{i,t}(u) dt + \int_{L^2}\int_{D}
\eta(\bar{v}_i) \theta(0) dx d\bar{\mu}^{\ast}_i(\bar{u})\Bigg) = 0,
\end{equation}
for the same test function $\theta$ used in \eqref{eq:pf3}.

Since $\mu_t$ is a dissipative statistical solution of \eqref{eq:cl}, we have
\begin{align*}
0 &\leq \int_{\R_+}\int_{L^2}\int_D \eta(u(x)) \theta'(t) dx d\mu_t dt + \int_{L^2}\int_D \eta(\bar{u}(x)) \theta(0) dx d\bar{\mu}(\bar{u}) \\
&=\sum_{i=1}^M\alpha_i\Bigg(\int_{\R_+}\int_{L^2}\int_D
\eta(u(x)) \theta'(t) dx d\mu_{i,t}^\ast dt + \int_{L^2}\int_D
\eta(\bar{u}(x)) \theta(0) dx d\bar{\mu}_i^\ast(\bar{u})\Bigg).
\end{align*}
Subtracting \eqref{eq:pf4} and \eqref{eq:pf3} from the above inequality we obtain
\begin{equation}\label{eq:pf5}
\begin{split}
&\sum_{i=1}^M\alpha_i\int_0^T\int_{L^2}\int_{D}
 \underbrace{\big(\eta(u) - \eta(v_i) -  \eta'(v_i)(u-v_i)\big)}_{=:\,\mathcal{H}(u|v_i)}\theta'(t)\,dxd\mu^{\ast}_{i,t}(u) dt \\
&+ \sum_{i=1}^M\alpha_i \int_{L^2}\int_{D}
\big(\eta(\bar{u}) - \eta(\bar{v}_i) - \eta'(\bar{v}_i)(\bar{u}-\bar{v}_i)\big) \theta(0)\, dx d\bar{\mu}^{\ast}_i(\bar{u}) \\
&+ \sum_{i=1}^M\alpha_i\int_0^T\int_{L^2}\int_{D}
\theta(t) \mathcal{Z}(u|v_i)\, dx d\mu^{\ast}_{i,t}(u) dt \geq 0 
\end{split}
\end{equation}
As $\eta$ is strictly convex (the lower bound in \eqref{eq:squareentropy}), we have
\begin{align*}
{\mathcal H}(u|v_i) \geq c|u-v_i|^2.
\end{align*}
Similarly, using \eqref{eq:squareflux} and the fact that $\|v_i\|_{W^{1,\infty}} \leq R$, we obtain that 
\begin{align*}
\max\big(\mathcal{Z}(u|v_i),{\mathcal H}(u|v_i)\big) \leq C |u-v_i|^2.
\end{align*}
Using the above estimates in \eqref{eq:pf5} and choosing the test function $\theta(s) = \chi_{(0,t]}(s)$ (by approximating with smooth functions) yields
\begin{equation}\label{eq:pf6}
\begin{split}
\sum_{i=1}^M\alpha_i\int_{L^2}\int_{D} |u(x) - v_i(x,t)|^2 dx d\mu^{\ast}_{i,t}(u) \leq \sum_{i=1}^M\alpha_i\int_{L^2}\int_{D} |\bar{u}(x) - \bar{v}_i(x)|^2 dx d\bar{\mu}^{\ast}_{i}(\bar{u}) \\
+ C \sum_{i=1}^M\alpha_i\int_{0}^t\int_{L^2}\int_{D} 
|u(x) - v_i(x,s)|^2 dx d\mu^{\ast}_{i,s}(u) ds.
\end{split}
\end{equation}
Applying the integral form of Gr\"onwall's inequality to the above estimate results in
\begin{align*}
\sum_{i=1}^M\alpha_i\int_{L^2}\|u - v_i\|^2_{L^2} dx d\mu^{\ast}_{i,t}(u) 
&\leq e^{Ct} \sum_{i=1}^M\alpha_i\int_{L^2}\|\bar{u} - \bar{v}_i\|^2_{L^2} dx d\bar{\mu}^{\ast}_{i}(\bar{u}) \\
&=  e^{Ct} W_2(\bar{\mu},\bar{\rho})^2,
\end{align*}
the last step following from \eqref{eq:pf1}. As $(\mu_{1,t}^*,\dots,\mu_{M,t}^*) \in \Lambda(\alpha,\mu_t)$, the above inequality implies \eqref{eq:dssest} and concludes the proof. 
\end{proof}

The estimate \eqref{eq:dssest} implies stability of dissipative statistical solutions with respect to probability measures that are convex combinations of Dirac masses, concentrated on classical solutions of \eqref{eq:cl}. We can extend such a stability result to a more general class of strong solutions:
\begin{definition}
\label{def:strongss}
A statistical solution $\mu$ is a \emph{strong statistical solution} of \eqref{eq:cl} if there is some $R>0$ such that for every $n \in \N$, there exists a $\rho_n \in \Prob_T(L^2(D,U))$ of the form $\rho_{n,t} = \sum_{i=1}^{N_n} \alpha_i \delta_{v_i(t)}$ such that each $v_i$ is a classical solution of \eqref{eq:cl} and $v_i \in B_R \subset W^{1,\infty}(D \times (0,T))$, for all $1 \leq i \leq N_n$ and such that
\begin{equation}
\label{eq:sss1}
W_2(\mu_t,\rho_{n,t}) \leq \frac{1}{n} \qquad \forall\ t \in [0,T]
\end{equation}
\end{definition}
We can now prove our main weak-strong uniqueness result:
\begin{theorem}\label{thm:ws}
Let $\bar{\mu}\in\Prob(L^2(D;U))$. Then under the same assumptions as in Lemma \ref{lem:ws1}, if there exists a strong statistical solution $\mu$ of \eqref{eq:cl}, then it is unique in the class of dissipative statistical solutions.
\end{theorem}
\begin{proof}
For a fixed $n\in\N$, Definition \ref{def:strongss} implies the existence of a statistical solution $\rho_t = \sum_{i=1}^{N_n} \alpha_i \delta_{v_i(t)}$ such that 
\begin{equation}\label{eq:sss2}
W_2(\mu_t,\rho_t) \leq \frac{1}{n}.
\end{equation}
Moreover, $v_i(x,t) \in B_R$ is a classical solution of~\eqref{eq:cl} with initial data $\bar{v}_i$ for all $i=1,\dots,N_n$.

Let $\gamma_t$ be another dissipative statistical solution of~\eqref{eq:cl} with $\gamma_0=\bar{\mu}$. By Lemma \ref{lem:ws1} we have
$$
W_2(\rho_t,\gamma_t) \leq \frac{e^{Ct}}{n}
$$
for some $C=C(R)$. Using \eqref{eq:sss2} and the triangle inequality yields
$$
W_2(\mu_t,\gamma_t) \leq \frac{e^{Ct}+1}{n}.
$$
Letting $n \rightarrow \infty$ concludes the proof of uniqueness of strong solutions.
\end{proof}
\begin{remark}
If we assume that the initial data $\bar{\mu}$ is such that ${\rm supp}(\bar{\mu}) \subset B_{R_0}$, with $B_{R_0}$ being the ball of radius $R_0>0$ in $W^{1,\infty}(D;U)$, then by classical results on local well-posedness for \eqref{eq:cl} with strictly convex entropies \cite{Dafermos}, there exist $T(R_0), R(R_0)>0$, such that for every initial data $\bar{v} \in {\rm supp}(\bar{\mu})$, there exists a corresponding classical solution $v \in W^{1,\infty}(D \times [0,T(R_0)])$ and $\|v\|_{W^{1,\infty}(D \times [0,T(R_0)])} \leq R(R_0)$. Moreover, the data-to-solution map $\mathcal{S}_t:\supp\bar{\mu} \to L^2(D;U)$, is well defined for all $0 \leq t \leq T(R_0)$, and continuous because
$$
\|\mathcal{S}_t (\bar{v}) - \mathcal{S}_t (\hat{v})\|_2 \leq e^{R(R_0)t}\|\bar{v}-\hat{v}\|_2, \qquad \forall\ t \in [0,T(R_0)].
$$
Letting $\mu_t = \mathcal{S}_t \#\bar{\mu}$ for all $0 \leq t \leq T(R_0)$, one can verify that $\mu_t$ is indeed a dissipative statistical solution of \eqref{eq:cl}. Moreover, it is strong in the sense of Definition \ref{def:strongss}. Consequently, we can establish that \emph{as long as the underlying initial data is supported on smooth functions, the resulting statistical solutions are locally well-posed. }
\end{remark}

\section{Numerical approximation of statistical solutions.}
\label{sec:4}
In this section, we will {construct} statistical solutions for the system of conservation laws \eqref{eq:cl} by proposing an algorithm to numerically approximate it. We show, under reasonable hypotheses on the underlying numerical schemes, that the approximations constructed by this algorithm converge to a (dissipative) statistical solution of \eqref{eq:cl}. As in \cite{fkmt,FMTacta,FLyeM1}, the algorithm will be based on a finite volume spatio-temporal discretization and Monte Carlo sampling of the underlying probability space. The spatial domain $D$ will everywhere be assumed to be bounded.

\subsection{Multidimensional finite volume framework}\label{sec:fvm}
In this section we briefly describe numerically approximating conservation laws with finite volume and finite difference methods. For a complete review, one can consult~\cite{GR,HEST1,HR,leveque_green}.

We discretize the computational spatial domain as a collection of cells
\[\big\{(x^1_{i^1-\hf}, x^1_{i^1+\hf})\times\cdots\times(x^d_{i^d-\hf}, x^d_{i^d+\hf})\big\}_{(i^1,\ldots,i^d)}\subset D,\]
with corresponding cell midpoints
\[x_{i^1,\ldots, i^d}:=\left(\frac{x^1_{i^1+\hf}+x^1_{i^1-\hf}}{2},\ldots,\frac{x^d_{i^d+\hf}+x^d_{i^d-\hf}}{2}\right).\]
For simplicity we assume that our mesh is equidistant, that is,
\[x^m_{i^m+\hf} - x^m_{i^m-\hf} \equiv \Dx \qquad \forall\ m=1,\dots,d\]
for some $\Dx>0$. For each cell, marked by $\bi=(i^1,\ldots, i^d)$, we let $u^{\Dx}_\bi(t)$ (and equivalently $u^{\Dx}_{i^1,\ldots, i^d}(t)$) denote the averaged value in the cell at time $t\geq 0$. We consider the following semi-discrete scheme,
\begin{equation}\label{eq:semi_d}
\begin{split}
\ddt{}u^{\Dx}_{i^1,\ldots,i^d}(t)+\sum_{m=1}^d\frac{1}{\Dx}\bigg(F^{m,\Dx}\big(u^{\Dx}_{\bi-(q-1)\be_m}(t),\ldots, u^{\Dx}_{\bi+q\be_m}(t)\big) \\
-F^{m,\Dx}\big(u^{\Dx}_{\bi-q\be_m}(t),\ldots, u^{\Dx}_{\bi+(q-1)\be_m}(t)\big)\bigg)= 0  \\
u^{\Dx}_{i^1,\ldots,i^d}(0) = \bar{u}(x_{i^1,\ldots,i^d})
\end{split}
\end{equation}
where $\be_1,\dots,\be_d$ are the canonical unit vectors in $\R^d$, and $F^{m,\Dx}$ is a \emph{numerical flux function} in direction $m=1,\dots,d$. We say that the scheme is a \emph{$(2q+1)$-point scheme,} when the numerical flux function $F^{m,\Dx}$ can be written as a function of $u^{\Dx}_{\bi+j\be_m}(t)$ for $j=-q+1,\ldots,q$. We furthermore assume the numerical flux function is consistent with $f$ and locally Lipschitz continuous, which amounts to requiring that for every bounded set $K\subset \R^N$, there exists a constant $C>0$ such that for $m=1,\ldots,d$,
\begin{equation}
\label{eq:fvm_lipschitz}
\big|F^{m,\Dx}\big(u^{\Dx}_{\bi-(q-1)\be_m}(t),\ldots, u^{\Dx}_{\bi+q\be_m}(t)\big)-f^m(u^{\Dx}_{\bi}\big)\big|\leq C\sum_{j=-q+1}^{q}\big|u^{\Dx}_\bi(t)-u^{\Dx}_{\bi+j\be_m}(t)\big|,
\end{equation}
whenever $u^{\Dx}_{\bi-(q-1)\be_m}(t),\ldots, u^{\Dx}_{\bi+q\be_m}(t) \in K$. For the sake of notational simplicity we will write
\[
F^m_{\bi+ \hf\be_m}(u) = F^{m,\Dx}\big(\Avg{u}{\bi-(q-1)\be_m},\ldots,\Avg{u}{\bi+q\be_m}\big)\qquad \text{for }\bi\in\Z^d,\ 1\leq m \leq d.
\]
We let $\NumericalEvolution{\Dx}_t:L^p(D)\to L^p(D)$ be the spatially discrete numerical evolution operator defined by~\eqref{eq:semi_d}, mapping $\bar{u} \mapsto u^\Delta(t)$. Since $\NumericalEvolution{\Dx}_t$ is the composition of a projection from $L^p$ onto piecewise constant functions and a continuous evolution under an ordinary differential equation, we see that $\NumericalEvolution{\Dx}_t$ is measurable. 

The current form of \eqref{eq:semi_d} is continuous in time, and one needs to employ a time stepping method to discretize the ODE system in time, usually through some strong stability preserving Runge--Kutta method \cite{GST}.

As the operator $\NumericalEvolution{\Dx}$ is a measurable map,  we can define an approximation of a statistical solution of \eqref{eq:cl} with initial data $\bar{\mu}$ by 
\begin{equation}
\label{eq:statapp}
\mu_t^{\Dx}=\pushforward{\NumericalEvolution{\Dx}_t}{\bar{\mu}}.
\end{equation}
Henceforth, $\mu_t^{\Dx}$ is referred to as an \emph{approximate statistical solution.}

\subsection{Convergence of approximate statistical solutions.}
\label{sec:conv}
In this section, we investigate the convergence of the approximate statistical solutions $\mu_t^{\Dx}$ as the mesh is refined, i.e.,~as $\Dx \rightarrow 0$.

\begin{theorem}\label{theo:conv}
Consider the system of conservation laws \eqref{eq:cl} with initial data $\bar{\mu} \in \Prob(L^p(D;U))$ for some $1 \leq p < \infty$, such that $\supp(\bar{\mu}) \subset B_R(0)\subset L^p(D;U)$, with $B_R(0)$ being the ball of radius $R$ and center 0, for some $R > 0$. Assume that the semi-discrete finite volume scheme \eqref{eq:semi_d} satisfies the following conditions:
\begin{itemize}
\item [(i)] \emph{$L^p$ bounds}: 
\begin{equation}
\label{eq:lpb}
\Dx^{d} \sum_{\bi\in \mathbb{Z}^d} \left|u^{\Dx}_\bi(t) \right|^p \leq C \Dx^{d}  \sum_{\bi\in \mathbb{Z}^d} \left|\bar{u}_\bi\right|^p \qquad \forall\ t \in [0,T), \ \forall\ \bar{u} \in L^p(D;U). 
\end{equation}
\item [(ii)] \emph{Weak BV bounds}: There exists $s \geq p$ such that
\begin{equation}\label{eq:wbv}
\Dx^{d} \int_0^T \sum_{m=1}^d \sum_{\bi\in \mathbb{Z}^d} \left|u^{\Dx}_{\bi+\mathbf{e}_m}(t)
- u^{\Dx}_\bi(t)\right|^s dt \leq C\Dx,
\end{equation}
with the constant $C = C(\|\bar{u}\|_p)$ only depending on the $L^p$-norm of the initial data $\bar{u}$. 
\item[(iii)] \emph{Approximate scaling}: There exists a constant $C>0$, possibly depending on the initial data $\bar{\mu}$ but independent of the grid size $\Dx$, such that for every $\ell > 1$
\begin{equation}\label{eq:sfscal}
S^p_{\ell \Dx}(\mu^{\Dx}) \leq C \ell^{\frac1{s}} S^p_{\Dx}(\mu^{\Dx}).
\end{equation}
Here, $S^p_r(\mu)$ is the \emph{structure function} associated with the time parameterized probability measure $\mu_t \in \Prob_T(L^p(D;U))$ {\rm(}equivalently, time-dependent correlation measure $\cm \in \Corrmeas^{p}([0,T),D;U)${\rm)}, defined as
\begin{equation}
\label{eq:sf}
S^p_r(\mu):= \left(\int_0^T \int_{L^p(D)} \int_{D}\intavg_{B_r(x)} |u(x) - u(y)|^p dy dx d\mu_t(u) dt\right)^{\frac{1}{p}}.
\end{equation}
\end{itemize}
Then there is a subsequence $\Dx'\to0$ such that the approximate statistical solutions $\mu^{\Dx'}$ converge strongly to some $\mu \in \Prob_T(L^p(D,U))$, in the sense of Theorem \ref{thm:timedepcmcompactness}(v).
\end{theorem}
\begin{proof}
We will show that that the approximate statistical solutions $\mu_t^{\Dx}$, defined in \eqref{eq:statapp}, satisfy the conditions of Theorem \ref{thm:timedepcmcompactness} and hence converge (up to a subsequence). To this end, we can readily verify from the uniform $L^p$ bounds \eqref{eq:lpb} and the fact that $\supp\bar{\mu}\subset B_R(0)$, that
$$
\supp(\mu^{\Dx}_t) \subset B_{CR}(0) \qquad \forall\ t \in [0,T),
$$
with $C$ being the constant in \eqref{eq:lpb}.

Let $u^\Dx(t)=\NumericalEvolution{\Dx}_t\bar{u}$ be the solution generated by the scheme \eqref{eq:semi_d}. Denoting
\begin{equation}
\label{eq:V}
V_{\Dx}(\bar{u}) =  \int_0^T \sum_{m=1}^d \sum_{\bi\in \mathbb{Z}^d} \left|u^{\Dx}_{\bi+\mathbf{e}_m}(t)
- u^{\Dx}_\bi(t)\right|^p dt,
\end{equation}
we obtain from the weak BV estimate \eqref{eq:wbv} and H\"older's inequality that (recall that $D$ was assumed to be bounded)
\begin{equation}
\label{eq:V1}
V_{\Dx}(\bar{u}) \leq C(T,d,R) \Dx^{\frac{p}{s}-d}.
\end{equation}
The above inequality holds for every $\bar{u} \in B_R(0) \subset L^p(D;U)$. 

Next, for any $r \leq \Dx$, a straightforward but tedious calculation yields
\begin{align*}
\left(S^p_r(\mu^{\Dx}_t)\right)^p &= 
\int_0^T \int_{L^p} \int_D \intavg_{B_r(x)} |u(x) - u(y)|^p dy dx d\mu^{\Dx}_t(u)  dt \\
&=\int_0^T \int_{L^p} \int_D \intavg_{B_r(x)} \left|{\mathcal S}^{\Dx}_t \bar{u}(x) - {\mathcal S}^{\Dx}_t \bar{u}(y) \right|^p  dy dx d\bar{\mu}(\bar{u})  dt && \text{(by \eqref{eq:statapp})} \\
&=\int_{L^p} \int_0^T \int_D \intavg_{B_r(x)} \left|u^{\Dx}(x,t) - u^{\Dx}(y,t) \right|^p  dy dx dt d\bar{\mu}(\bar{u}) \\
&\leq C_d \Dx^{d-1} r \int_{L^p} V_{\Dx}(\bar{u})  d\bar{\mu}(\bar{u}) \\ 
&\leq C(T,d,R)C_d r^{\frac{p}{s}} &&\text{(by \eqref{eq:V1} and $r \leq \Dx$)},
\end{align*}
where $C_d = 3^{d-1}$ results from successive applications of the triangle inequality. Hence, summarizing the above calculation, we obtain that for any $r \leq \Dx$,
\begin{equation}
\label{eq:V2}
S_r^p(\mu_t^{\Dx},T) \leq C r^{\frac{1}{s}},
\end{equation}
for a constant $C$ that depends on the dimension, the support of the initial probability measure and the final time but is independent of the grid size $\Dx$. 

Now, for any $\ell > 1$ and $r = \ell \Dx$, we have 
\begin{align*}
S_{r}^p(\mu_t^{\Dx}) &= S_{\ell \Dx}^p(\mu_t^{\Dx}) \leq C \ell^{\frac1{s}} S_{\Dx}^p(\mu_t^{\Dx}) && \text{(by scaling \eqref{eq:sfscal})} \\
&\leq C \ell^{\frac{1}{s}} \Dx^{\frac{1}{s}} && \text{(by \eqref{eq:V2})} \\
&= C r^{\frac{1}{s}}.
\end{align*}
Here, the constant is independent of $\Dx$. By combining the above estimate with \eqref{eq:V2}, we obtain that
\begin{equation}\label{eq:V3}
S_{r}^p(\mu_t^{\Dx}) \leq C r^{\frac{1}{s}}
\end{equation}
for any $r > 0$. 

Given the independence of the constant in \eqref{eq:V3} with respect to the grid size $\Dx$, we see from \eqref{eq:V3} that the condition of \emph{uniform diagonal continuity} \eqref{eq:timedepuniformdc}  is satisfied. Hence, up to a subsequence still indexed by $\Dx$, $\mu^{\Dx}_t$ converges to some $\mu \in \Prob_T(L^p(D;U))$.
\end{proof}

Several remarks on the assumptions in the above theorem follow.
\begin{remark}
There are many examples of finite volume/difference schemes of the form \eqref{eq:semi_d} which satisfy the uniform $L^p$ bound \eqref{eq:lpb} and the weak BV bound \eqref{eq:wbv}. Assume that the system of conservation laws \eqref{eq:cl} possesses an entropy function $\eta$ that satisfies
\begin{equation}\label{eq:aentropy}
C_1(1+|u|^p) \leq \eta(u) \leq C_2(1+|u|^p) \qquad \forall\ u \in U
\end{equation}
for constants $C_1,C_2 > 0$ and $p\in[1,\infty)$. Then the uniform $L^p$ bound \eqref{eq:lpb} follows for any scheme of form \eqref{eq:semi_d} that satisfies a \emph{discrete entropy inequality},
\begin{equation}\label{eq:denten}
\begin{split}
\ddt{}\eta\big(u^{\Dx}_\bi(t)\big)+\sum_{m=1}^d\frac{1}{\Dx}\left(Q^{m,\Dx}_{\bi+\frac{1}{2}\be_m}(t)-Q^{m,\Dx}_{\bi-\frac{1}{2}\be_m}(t)\right)\leq 0, \\
\end{split}
\end{equation}
with a numerical entropy flux $Q^{m,\Dx}$ that is consistent with the entropy flux $q^m$ in \eqref{eq:entrcond} for $1 \leq m \leq d$. 

In many cases the weak-BV bound \eqref{eq:wbv} also follows from the discrete entropy inequality \eqref{eq:denten}. Examples of schemes which satisfy the discrete entropy inequality \eqref{eq:denten} and the weak BV bound \eqref{eq:wbv} are the so-called entropy stable Lax--Wendroff schemes and the TeCNO schemes of \cite{FMT_TeCNO}. 
\end{remark}

\begin{remark}
The approximate scaling assumption \eqref{eq:sfscal} can be thought of as a weaker version of the so-called  \emph{self-similarity at small scales} assumption of Kolmogorov in his K41 theory for fully
developed turbulence in incompressible fluid flows, see hypothesis H2, equation (6.3), page 75 of \cite{fris1}. Kolmogorov based his hypothesis on the fact that solutions of the incompressible Navier--Stokes (Euler) equations scale exactly. Similar considerations also apply to several prototypical examples of systems of conservation laws \eqref{eq:cl}. In particular, for the compressible Euler equations \eqref{eq:euler} (in any space dimension), it can be readily checked that if $u(x,t)$ is a solution, then $\ell^{\theta} u(\ell x, \ell t)$ is also a solution for any $\theta,\ell >0$. Hence, it is reasonable to hypothesize scaling, analogous to the Kolmogorov hypothesis, for systems of conservation laws.

It is essential to also point out the differences in our hypothesis \eqref{eq:sfscal} to the standard Kolmogorov hypothesis for turbulent incompressible flows. First, our hypothesis pertains only to the numerical solution, generated by the finite volume scheme \eqref{eq:semi_d}. Moreover, we require mere inequalities in the scaling law \eqref{eq:sfscal}, in contrast to the standard Kolmogorov hypothesis of equality.  
\end{remark}

\begin{remark}
Intermittency is widely accepted to be a characteristic of turbulent flows, see \cite{fris1}. It is believed that intermittency stems from the fact that turbulent solutions do not scale exactly as in the Kolmogorov hypothesis. We automatically incorporate a form of intermittency by only requiring an upper bound in \eqref{eq:sfscal}, instead of an equality. Hence, the scaling exponent in \eqref{eq:sfscal} can depend explicitly on the underlying length scale, provided that it is bounded above by $1/s$. This encodes a form of intermittency in the approximate solutions. 
\end{remark}

\begin{remark}
Another approach to incorporating intermittency and relaxing the scaling condition \eqref{eq:sfscal} is to consider a decomposition of the approximate statistical solution $\mu^{\Dx}_t$ into a mean flow and a fluctuation. Defining the mean flow by
\begin{equation}
\label{eq:mf}
\widehat{u}^{\Dx}(x,t) = \langle \nu^{1,\Dx}_{t,x}, \xi \rangle,
\end{equation} 
we see that the mean flow is well defined for almost every $(x,t) \in D \times (0,T)$. Similarly, we can define \emph{fluctuations} of $\tilde{\mu}^\Dx \in \Prob_T(L^p(D;U))$ by its action on all observables $g\in \Caratheodory^{k,p}([0,T],D;U)$,
\begin{equation}
\label{eq:fluct1}
\Ypair{\tilde{\mu}^\Dx}{L_g} = \int_0^T \int_{L^p(D;U)} \int_{D^k} g\big(x,t, u(x) - \widehat{u}^{\Dx}(x,t)\big)\, dx d\tilde{\mu}^\Dx_t(u) dt. 
\end{equation}
We can relax the assumption \eqref{eq:sfscal} by requiring that only the structure function associated with the fluctuation scales approximately, i.e.,
\begin{equation}
\label{eq:sfscalfluct}
S^p_{\ell \Delta}(\tilde{\mu}^\Dx) \leq C \ell^{1/s} S^p_{\Delta}(\tilde{\mu}^\Dx) \qquad \forall\ \ell > 1.
\end{equation}
If we further assume that the mean flow is BV and $L^\infty$, i.e.,
\begin{equation}\label{eq:bv}
\max\left(\norm{\widehat{u}^{\Delta}}_{L^{\infty}((0,T) \times D)},\ \norm{\widehat{u}^{\Delta}}_{L^{\infty}((0,T),BV(D))}\right) \leq C,
\end{equation}
for some constant that is independent of the mesh size $\Dx$, then a straightforward but tedious calculation yields for any $r = \ell\Delta$
\begin{align*}
S^p_{r}(\mu^{\Delta}) &\leq \bar{C}\norm{\hat{u}^{\Delta}}_{L^{\infty}((0,T),BV(D))} r^{\frac{1}{p}} + S^p_{\ell \Delta}(\tilde{\mu}^\Dx) \\
&\leq \bar{C} r^{\frac{1}{p}} + C \ell^{\frac1{s}}  S^p_{\Delta}(\tilde{\mu}^\Dx) && \text{(by \eqref{eq:sfscalfluct})} \\
&\leq \bar{C} r^{\frac{1}{p}} + C \ell^{\frac1{s}} \Dx^{\frac{1}{s}} && \text{(by \eqref{eq:V2})} \\
&= \bar{C} r^{\frac{1}{p}} + C r^{\frac{1}{s}}.
\end{align*} 
Thus, the condition \eqref{eq:timedepuniformdc} in Theorem \ref{thm:timedepcmcompactness} is satisfied in this case. A similar argument can be made by imposing some form of (uniform) H\"older continuity on the mean flow. 
\end{remark}


\subsection{Consistency of the numerical method}
We fix an initial measure $\bar{\mu}\in \Prob(L^p(\D,\Ph))$. For any $u\in L^p(\D,\Ph)$ we define the local average of $u$ as
\[
\Avg{u}{\bi}=\frac{1}{|\mathcal{C}_\bi|}\int_{\mathcal{C}_\bi}u(x)\dd x\qquad \text{for }\bi\in\Z^d,
\]
where $|\mathcal{C}_\bi|$ denotes the $d$-dimensional Lebesgue measure of $\mathcal{C}_\bi$. 
We now state the ``Lax--Wendroff theorem'' for our numerical method, that is, consistency of the method with the PDE.

\begin{theorem}[Lax--Wendroff theorem for statistical solutions]\label{theo:lxw}
	Let the initial data $\bar{\mu}$ have bounded support, $\supp\bar{\mu}\subset B_K(0)\subset L^p(\D,U)$ for some $K>0$. Let $\mu^{\Delta}$ be given by \eqref{eq:statapp} for $\Dx>0$, and assume that for some sequence $\Delta_n\to0$, the sequence $\{\mu^{\Delta_n}\}_{n\in\N}$ converges strongly to $\mu$ in $\Prob_T(L^p(\D,\Ph))$, in the sense of Theorem \ref{thm:timedepcmcompactness}(v). Assume moreover that the following weak BV bound is fulfilled:
	\begin{equation}\label{eq:lxweakbv}
	\Dx^{d}\int_0^T \int_{L^p(\D)} \sum_{m=1}^d \sum_{\bi\in\GridIndexSetD} \left|\AvgS{u}_\bi - \AvgS{u}_{\bi-\be_m}\right|^s \; d\mu^{\Dx}_t(u)\;dt \leq C\Dx,
	\end{equation}
	for some $0<s$.
	Then $\mu_t$ is a statistical solution of \eqref{eq:cauchy}.
\end{theorem}

Given the complicated notation and very technical nature of the proof of the above theorem, we illustrate the main steps of the proof in a very special case, namely $k=2$ for a one-dimensional scalar conservation law ($d=N=1$). The proof in the general case is postponed to \Cref{app:laxwendroff}.

\begin{proof}[Proof for the second moment of a scalar conservation law in one spatial dimension]
	We consider a scalar conservation law ($N=1$) in one spatial dimension ($d=1$). By~\eqref{eq:lpb}, there is some $K_T>0$ such that
	\begin{equation}\label{eq:boundedsupport}
	\supp\mu_t^\Dx \subset B_{K_T}(0) \qquad \forall\ t\in[0,T].
	\end{equation}
	Let $\{u_i^\Delta\}_{i\in\Z}$ be computed by \eqref{eq:semi_d}. Denote $F^\Dx = F^{1,\Dx}$ and, for $u\in L^p(D)$, write
	\[
	F_\iphf^\Dx(u) = F^\Dx\big(u_{i-q+1},\dots,u_{i+q}\big), \qquad i\in\Z.
	\]
	For all pairs of cells $i,j$, we have, by the product rule,
	\begin{align*}
	\ddt{}\big(u^{\Dx}_i(t)u^{\Dx}_j(t)\big)&=-\frac{1}{\Dx}\big(F^\Dx_\iphf(u^{\Dx}(t))-F^\Dx_\imhf(u^{\Dx}(t))\big)u^{\Dx}_j \\ 
	&\quad-\frac{1}{\Dx}u^{\Dx}_i\big(F^\Dx_\jphf(u^{\Dx}(t))-F^\Dx_\jmhf(u^{\Dx}(t)(t))\big).
	\end{align*}
	Hence, for arbitrary $\phi\in C_c^\infty(\R^2\times[0,T))$, we get
	\begin{align*}
	0 &=\int_0^T u^{\Dx}_i(t)u^{\Dx}_j(t)\partial_t\phi(x_i,x_j,t) \dd t + \phi(x_i,x_j,0)u^{\Dx}_j(0)u_i^{\Dx}(0) \\
	&\quad-\int_0^T\bigg(\frac{1}{\Dx}\big(F^\Dx_\iphf(u^{\Dx}(t)) -F^\Dx_\imhf(u^{\Dx}(t))\big)u^{\Dx}_j\\
	&\quad +\frac{1}{\Dx}u^{\Dx}_i\big(F^\Dx_\jphf\big(u^{\Dx}(t)) -F^\Dx_\jmhf(u^{\Dx}(t))\big)\bigg)\phi(x_i,x_j,t) \dd t.
	\end{align*}
	Multiply by $\Dx^2$, sum over all $i,j\in\Z$ and perform a summation-by-parts to obtain
	\begin{align*}
	0 &=\Dx^2\sum_{i,j\in\Z}\left(\int_0^T u^{\Dx}_i(t)u^{\Dx}_j(t)\partial_t\phi(x_i,x_j,t) \dd t + u^{\Dx}_j(0)u_i^{\Dx}(0)\phi(x_i,x_j,0)\right) \\
	&\quad+\Dx^2\sum_{i,j\in\Z}\int_0^T\bigg(F^\Dx_\iphf(u^{\Dx}(t))u^{\Dx}_j(t)\frac{\phi(x_{i+1},x_j,t)-\phi(x_i,x_j,t)}{\Dx}\\
	&\quad+u^{\Dx}_i(t)F^\Dx_\jphf(u^{\Dx}(t))\frac{\phi(x_i,x_{j+1},t)-\phi(x_i,x_j,t)}{\Dx}\bigg) \dd t \\
	&=\Dx^2\sum_{i,j\in\Z}\left(\int_0^T u^{\Dx}_i(t)u^{\Dx}_j(t)\partial_t\phi(x_i,x_j,t) \dd t + u^{\Dx}_j(0)u_i^{\Dx}(0)\phi(x_i,x_j,0)\right) \\
	&\quad+\Dx^2\sum_{i,j\in\Z}\int_0^T\bigg(F^\Dx_\iphf(u^{\Dx}(t))u^{\Dx}_j(t)\partial_1^\Dx\phi(x_i,x_j,t) +u^{\Dx}_i(t)F^\Dx_\jphf(u^{\Dx}(t))\partial_2^\Dx\phi(x_i,x_j,t)\bigg) \dd t
	\end{align*}
	where we have denoted $\partial_1^\Dx\phi(x,y,t)=\frac{\phi(x+\Dx,y,t)-\phi(x,y,t)}{\Dx}$, and similarly for $\partial_2^\Dx\phi$. From the special form \eqref{eq:statapp} of $\mu_t^\Dx$, we therefore have
	\begin{align*}
	0 &=\Dx^2\sum_{i,j\in\Z}\left(\int_0^T\int_{L^p} u_iu_j\partial_t\phi(x_i,x_j,t)\dd\mu_t^\Dx(u) \dd t + \int_{L^p}u_iu_j\phi(x_i,x_j,0)\,d\bar{\mu}(u)\right) \\
	&\quad+\Dx^2\sum_{i,j\in\Z}\int_0^T\int_{L^p}\bigg(F^\Dx_\iphf(u)u_j\partial_1^\Dx\phi(x_i,x_j,t)+u_iF^\Dx_\jphf(u)\partial_2^\Dx\phi(x_i,x_j,t)\bigg)\dd\mu^\Dx_t(u)\dd t.
	\end{align*}
	We write now
	\begin{align*}
	&\quad\Dx^2\sum_{i,j\in\Z}\int_0^T\int_{L^p}F^\Dx_\iphf(u)u_j\partial_1^\Dx\phi(x_i,x_j,t)\dd\mu^\Dx_t(u)\dd t \\
	&= \Dx^2\sum_{i,j\in\Z}\int_0^T\int_{L^p}f(u_i)u_j\partial_1^\Dx\phi(x_i,x_j,t)\dd\mu^\Dx_t(u)\dd t \\
	&\quad+\Dx^2\sum_{i,j\in\Z}\int_0^T\int_{L^p}\big(F^\Dx_\iphf(u)-f(u_i)\big)u_j\partial_1^\Dx\phi(x_i,x_j,t)\dd\mu^\Dx_t(u)\dd t.
	\end{align*}
	The last term vanishes as $\Dx\to0$, since 
	\begin{align*}
	&\quad\left|\Dx^2\sum_{i,j\in\Z}\int_0^T\int_{L^p}\big(F^\Dx_\iphf(u)-f(u_i)\big)u_j\partial_1^\Dx\phi(x_i,x_j,t)\dd\mu^\Dx_t(u)\dd t\right|\\
	&\leq\int_0^T\int_{L^p}\Dx^2\sum_{i,j\in\Z}\big|F^\Dx_\iphf(u)-f(u_i)\big| |u_j| |\partial_1^\Dx\phi(x_i,x_j,t)|\dd\mu^\Dx_t(u)\dd t\\
	&\leq \int_0^T\int_{L^p}\Biggl(\Dx\sum_{i\in\Z}\big|F^\Dx_\iphf(u)-f(u_i)\big|\bigl\|\partial_1^\Dx\phi(x_i,\cdot,t)\bigr\|_{L^{p'}(\R)}\Biggr)\Biggl(\Dx\sum_{j\in\Z}|u_j|^p\Biggr)^{1/p}\dd\mu^\Dx_t(u)\dd t\\
	&\leq K_T \int_0^T\int_{L^p}\Dx\sum_{i\in\Z}\big|F^\Dx_\iphf(u)-f(u_i)\big|\bigl\|\partial_1^\Dx\phi(x_i,\cdot,t)\bigr\|_{L^{p'}(\R)}\dd\mu^\Dx_t(u)\dd t\\
	\intertext{(\textit{by \eqref{eq:boundedsupport}})}
	&\leq K_T\int_0^T\int_{L^p}\Dx\sum_{i\in\Z}\sum_{i'=i-q+1}^{i+q}\big|u_i-u_{i'}\big|\bigl\|\partial_1^\Dx\phi(x_i,\cdot,t)\bigr\|_{L^{p'}(\R)}\dd\mu^\Dx_t(u)\dd t\\
		\intertext{(\textit{by the Lipschitz continuity \eqref{eq:fvm_lipschitz}})}
	&\leq CK_T\int_0^T\int_{L^p}\Dx\sum_{i\in\Z}\big|u_i-u_{i-1}\big|\bigl\|\partial_1^\Dx\phi(x_i,\cdot,t)\bigr\|_{L^{p'}(\R)}\dd\mu^\Dx_t(u)\dd t\\
	&\leq CK_T\underbrace{\Biggl(\int_0^T\int_{L^p}\Dx\sum_{i\in\Z}\big|u_i-u_{i-1}\big|^s\dd\mu^\Dx_t(u)\dd t\Biggr)^{1/s}}_{\to0 \text{ as } \Dx\to0, \text{ by \eqref{eq:lxweakbv}}} \underbrace{\Biggl(\int_0^T\Dx\sum_{i\in\Z}\bigl\|\partial_1^\Dx\phi(x_i,\cdot,t)\bigr\|_{L^{p'}(\R)}^{s'}\dd t\Biggr)^{1/s'}}_{\text{bounded as } \Dx\to0} \\
	&\to 0.
	\end{align*}
	A similar computation holds for the integral involving $F_\jphf^\Dx(u)$. Setting $\Dx=\Dx_n$ then gives
	\begin{align*}
	0&=\lim_{n\to\infty}\Biggl(\Dx_n^2\sum_{i,j\in\Z}\left(\int_0^T\int_{L^p} u_iu_j\partial_t\phi(x_i,x_j,t)\dd\mu_t^{\Dx_n}(u) \dd t + \int_{L^p}u_iu_j\phi(x_i,x_j,0)\dd\bar{\mu}(u)\right) \\
	&\quad+\Dx_n^2\sum_{i,j\in\Z}\int_0^T\int_{L^p}\Bigl(f(u_i)u_j\partial_1^{\Dx_n}\phi(x_i,x_j,t)+u_if(u_j)\partial_2^{\Dx_n}\phi(x_i,x_j,t)\Bigr)\dd\mu^{\Dx_n}_t(u)\dd t\Biggr) \\
	\intertext{(\textit{as $u$ is piecewise constant $\mu^{\Dx_n}_t$-almost surely})}
	&=\lim_{n\to\infty}\Biggl(\int_0^T\int_{L^p}\int_{\R^2} u(x)u(y)\partial_t\phi(x,y,t)\dd x d y\dd\mu_t^{\Dx_n}(u) \dd t + \int_{L^p}u(x)u(y)\phi(x,y,0)\dd x d y\dd\mu^\Dx_0(u) \\
	&\quad+\int_0^T\int_{L^p}\int_{\R^2}\Bigl(f(u(x))u(y)\partial_1\phi(x,y,t)+u(x)f(u(y))\partial_2\phi(x,y,t)\Bigr)\dd x d y\dd\mu^{\Dx_n}_t(u)\dd t\Biggr) \\
	&=\int_0^T\int_{L^p}\int_{\R^2} u(x)u(y)\partial_t\phi(x,y,t)\dd x d y\dd\mu_t(u) \dd t + \int_{L^p}u(x)u(y)\phi(x,y,0)\dd x dy\dd\bar{\mu}(u) \\
	&\quad+\int_0^T\int_{L^p}\int_{\R^2}\Bigl(f(u(x))u(y)\partial_1\phi(x,y,t)+u(x)f(u(y))\partial_2\phi(x,y,t)\Bigr)\dd x d y\dd\mu_t(u)\dd t,
	\end{align*}
	which completes the proof.
\end{proof}

\begin{remark}
We can readily show that the limit statistical solution $\mu_t$ is a {dissipative statistical solution}, assuming that the underlying finite volume method satisfies the discrete entropy inequality \eqref{eq:denten}. To this end, for every choice of coefficients $\alpha_1,\dots,\alpha_M > 0$ with $\sum_{i}^M \alpha_i = 1$ and every $\left(\bar{\mu}_1,\ldots,\bar{\mu}_M \right) \in \Lambda(\alpha,\bar{\mu})$, we construct $\mu^{\Dx}_{i,t} = \NumericalEvolution{\Dx}_t\#\bar{\mu}_i$ as the approximate statistical solution generated by the scheme \eqref{eq:semi_d}. By the convergence theorem \ref{theo:conv}, we can show that each $\mu^{\Dx}_{i,t}$ converges (possibly along a further subsequence), in the topology of Theorem \ref{thm:timedepcmcompactness}, to $\mu_{i,t} \in \Prob_T(L^p(D;U))$ as $\Dx \rightarrow 0$. By Theorem \ref{theo:lxw}, each $\mu_{i}$ is a statistical solution of \eqref{eq:cl} with initial data $\bar{\mu}_i$, and the condition \eqref{eq:dss3} is a straightforward consequence of the discrete entropy inequality \eqref{eq:denten} and the growth condition \eqref{eq:aentropy} with $p=2$.
\end{remark}

\subsection{Monte Carlo algorithm}
While \eqref{eq:statapp} provides an abstract definition of $\mu_t^{\Dx}$, it is not amenable to practical computations, since it requires the computation of the trajectory of the numerical solution operator for almost all possible initial data $\bar{u}\in\supp \bar{\mu}$. We will further approximate $\mu_t^{\Dx}$ by sampling it for a large ensemble of initial data, drawn from the initial probability measure. 

The Monte Carlo algorithm has been shown to be robust in tackling high dimensional problems with low regularity~\cite{mlmc_hyperbolic,mss1}, and has later been demonstrated to perform very well for computing measure valued solutions~\cite{fkmt}.

\begin{Algorithm}[Monte Carlo Algorithm]\label{alg:mc} ~\\
\begin{algorithm}[H]
 \KwData{Initial $\bar{\mu}\in \Prob(L^p(\D,\Ph))$,
 	mesh width $\Dx>0$, numerical evolution operator $\NumericalEvolution{\Dx}$, number of samples $M\in\N$}
 For some probability space $(\Omega,\mathbf{\Omega},\mathbb{P})$, let $\bar{u}_1,\dots,\bar{u}_M:\Omega\to L^p(D;U)$ be independent random variables with distribution $\bar{\mu}$\;
 \For{$m=1,\ldots, M$}{
  Evolve the sample in time, $u_m^\Dx(t)=\NumericalEvolution{\Dx}_t(\bar{u}_m)$
 }
 Estimate statistical solution by the empirical measure
 \begin{equation}
 \label{eq:mc}
 \mu_t^{\Dx,M}(\omega):=\frac{1}{M}\sum_{m=1}^M \delta_{u_m^\Dx(\omega; \cdot, t)}.
 \end{equation}
\end{algorithm}
\end{Algorithm}

In the rest of this paper we will refer to the above algorithm simply as ``the Monte Carlo Algorithm''. Note that for any admissible observable $g \in \Caratheodory^{k,p}_1([0,T],D;U)$, using \eqref{eq:statapp} and \eqref{eq:mc}, we obtain that
\begin{equation}\label{eq:mcapp}
    \begin{split}
    \Ypair{\mu^{\Delta}}{L_g} &= \int_0^T \int_{L^p}\int_{D^k} g(x,t,u) dx d\mu^{\Delta}_t dt 
    = \int_{L^p}\int_0^T\int_{D^k} g\left(x,t,\NumericalEvolution{\Dx}_t\bar{u}(x)\right) dx dt d\bar{\mu} \\
    &\approx \frac{1}{M} \sum_{m=1}^M \left(\int_0^T\int_{D^k} g\left(x,t,\NumericalEvolution{\Dx}_t\bar{u}_m(x)\right) dx dt\right)= \Ypair{\mu^{\Delta,M}}{L_g}.
    \end{split}
\end{equation}

\begin{remark}
	One should note that the probability measure $\mu^{\Delta,M}_T$ is indeed a random probability measure depending on some probability space $\Omega$ from which $\InitialData{u}_1,\ldots\InitialData{u}_M$ are being drawn.
\end{remark}

Using well known results for weak convergence of Monte Carlo~\cite{vanderVaart1996}, we can prove that the Monte Carlo approximation of the statistical solution converges as the number of samples is increased.

\begin{theorem}
Let $\bar{\mu}\in \Prob(L^p(\D,\Ph))$ have bounded support, let $\NumericalEvolution{\Dx}$ be some numerical evolution operator, and let $\mu_t^{\Dx,M}$ be defined through the Monte Carlo Algorithm. Let $\mu_t^{\Dx}$ be defined by \eqref{eq:statapp}. Then for every admissible observable $g \in \Caratheodory^{k,p}_1([0,T],D;U)$, we have
\begin{equation}
\label{eq:mcest}
\E\bigg[\Ypair{\mu_T^{\Dx,M}-\mu_T^{\Dx}}{L_g}^2\bigg]\leq \frac{\Ypair{\mu_T^{\Dx}}{L_g^2}-\Ypair{\mu_T^{\Dx}}{L_g}^2}{M},
\end{equation}
where 
\begin{equation*}
\Ypair{\mu_T^{\Dx}}{L_g^2}:=\int_{L^p}\int_{\D^2}g(x,u(x))g(y,u(y))\,  dx dy\, d\mu_T^{\Dx}(u).
\end{equation*}
\end{theorem}
The proof of the above theorem follows by standard arguments for proving convergence of Monte Carlo approximations. It is analogous to the proof of Monte Carlo convergence to statistical solutions of scalar conservation laws, see Theorem~2 of \cite{FLyeM1}. Note that the right hand side in \eqref{eq:mcest} is bounded on account of the hypothesis \eqref{eq:gboundedt} on admissible observables $g \in \Caratheodory^{k,p}_1([0,T],D;U)$.

\section{Numerical experiments}
\label{sec:5}

For all the numerical experiments in this section, we consider the two-dimensional compressible Euler equations, 
\begin{equation}
\label{eq:euler}
\pdpd{}{t}
\begin{pmatrix}
\rho\\
\rho w^x\\
\rho w^y\\
E
\end{pmatrix}
+
\pdpd{}{x_1}
\begin{pmatrix}
\rho w^x\\
\rho \left(w^x\right)^2+p\\
\rho w^x w^y\\
(E+p)w^x
\end{pmatrix}
+
\pdpd{}{x_2}
\begin{pmatrix}
\rho w^y\\
\rho w^xw^y\\
\rho \left(w^y\right)^2+p\\
(E+p)w^y
\end{pmatrix}
=
0.
\end{equation} 
The system is closed with the equation of state
\[E=\frac{p}{\gamma-1}+\frac{\rho\left(\left(w^x\right)^2+\left(w^y\right)^2\right)}{2}.\]
We set $\gamma=1.4$ for all experiments.

\subsection{Kelvin--Helmholtz problem}\label{sec:kh}
We start with this well-known test case for the development of instabilities in fluid flows, which was also extensively studied in \cite{fkmt}.

The Kelvin--Helmholtz initial data is a shear flow, separating two states of varying density and pressure,
\begin{equation}\label{eq:kh}
u_0(\omega; x_1,x_2)=\begin{cases}u_L\qquad I_1(\omega; x_1)\leq x_2\leq I_2(\omega; x_1)\\
u_R\qquad \text{otherwise.}\end{cases}
\qquad (x_1,x_2)\in D = [0,1]^2
\end{equation}
We assign periodic boundary conditions, and the two states are given as 
$\rho_L=2$, $\rho_R=1$, $w^x_L=-0.5$, $w^x_R=0.5$, $w^y_L=w^x_L=0$ and $p_L=p_R=2.5$. The interfaces between the two states are given as

\begin{equation}\label{eq:kh_pert}
I_i(x, \omega) = \frac{2(i-1)+1}{4}+\epsilon \sum_{j=1}^K a^i_j(\omega)\sin(2\pi(x+b^i_j(\omega))),
\end{equation}
where $K=10$, $\epsilon>0$, and $\{a^i_j\}$ and $\{b^i_j\}$ are uniformly distributed random variables on the interval $[0,1]$. We normalize the $a^i_j$ such that $\sum_j a^i_j=1$.

The initial measure $\bar{\mu}$ is given by the distribution of the random field $u_0$. Note that although $\bar{\mu}$ is a probability measure on the infinite-dimensional space $L^p(D;U)$, it is only concentrated on a 40-dimensional subset of this space.

As was already shown in \cite{fkmt}, there is no convergence for single realizations (samples) of the problem \eqref{eq:euler},\eqref{eq:kh} as the mesh is refined. We observe this behaviour from Figure \ref{fig:kh_single_sample}, where we display the approximate density at time $T=2$, computed with a second-order high-resolution finite volume scheme using an HLLC approximate Riemann solver and WENO reconstruction, together with a second-order SSP Runge--Kutta time integrator, on a sequence of successively refined meshes. As seen from the figure, structures at finer and finer scales are generated upon mesh refinement, impeding convergence. This lack of convergence is also verified from Figure \ref{fig:KHconv1} (A), where the so-called Cauchy rates of the density, i.e.,~quantities of the form
\begin{equation}
    \label{eq:crate}
    \mathrm{Cauchy}_p(\Psi,\Dx,T):= \|\Psi^{\Dx}(\cdot,T) - \Psi^{\frac{\Dx}{2}}(\cdot,T) \|_{L^p(D)},
\end{equation}
with $\Psi^{\Delta}$ being any function computed with \eqref{eq:semi_d} on a mesh with mesh size $\Dx$.
\begin{figure}[h]
	\begin{subfigure}[b]{0.48\textwidth}
		\InputImage{\textwidth}{0.8\textwidth}{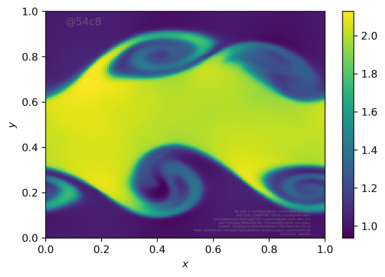}
		\caption{\MeshResolution{128}}
	\end{subfigure}
	\begin{subfigure}[b]{0.48\textwidth}
		\InputImage{\textwidth}{0.8\textwidth}{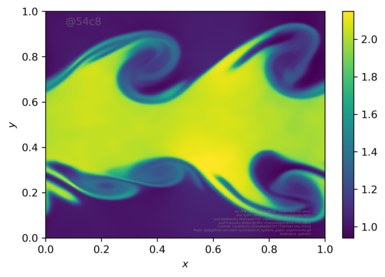}
		\caption{\MeshResolution{256}}
	\end{subfigure}
\begin{subfigure}[b]{0.48\textwidth}
	\InputImage{\textwidth}{0.8\textwidth}{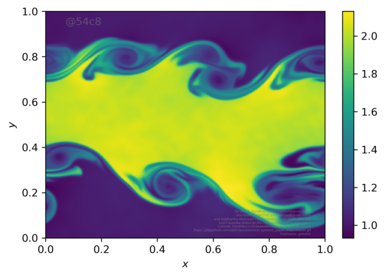}
	\caption{\MeshResolution{512}}
\end{subfigure}
\begin{subfigure}[b]{0.48\textwidth}
	\InputImage{\textwidth}{0.8\textwidth}{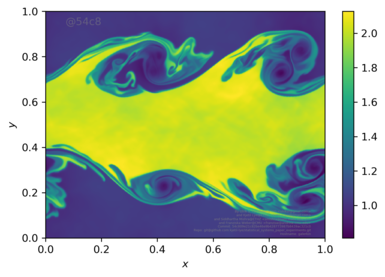}
	\caption{\MeshResolution{1024}}
\end{subfigure}
\caption{The approximate density of the Kelvin--Helmholtz instability \eqref{eq:kh} with a fixed $\omega\in\Omega$ for different mesh resolutions with the same initial data.The scheme used is an HLL3 flux with a WENO2 reconstruction algorithm. In this experiment, $\epsilon=0.01$ and $T=2$.}
\label{fig:kh_single_sample}
\end{figure}
On the other hand, the theory developed in Section \ref{sec:fvm} suggests that observables  $g \in \Caratheodory^{k,p}_1([0,T],D;U)$, for all $k$, should converge on mesh refinement. We start by considering observables with respect of the first marginal $\nu^1$ of the underlying approximate statistical solution. In particular, we consider the mean and variance given by
\begin{equation}
    \label{eq:mandv}
    M^{\Delta}(x,t):= \langle \nu^{1,\Dx}_{x,t}, \xi \rangle, \quad V^{\Delta}(x,t):= \langle \nu^{1,\Dx}_{x,t}, \xi^2 - (M^{\Delta}(x,t))^2 \rangle.  
\end{equation}
The above quantities are defined a.e.~in $D \times [0,T]$ and $\nu^{1,\Dx}$ is the first marginal of the approximate statistical solution $\mu_t^{\Dx,M}$ generated by the Monte Carlo Algorithm. It is straightforward to check that the mean and the variance are admissible observables, in the sense of Convergence Theorem \ref{thm:timedepcmcompactness}. 

We plot the mean and the variance of the density at time $T=2$ in Figure \ref{fig:KHmandv}. As seen from this figure, and in contrast to single samples, the mean and variance clearly converge upon mesh refinement. Moreover, the variance is also concentrated along the so-called \emph{mixing zone}, which spreads out from the initial interface. The convergence of the mean and the variance is further verified from Figure \ref{fig:KHconv1}(B,C), where the Cauchy rates \eqref{eq:crate} of the mean and variance $M^\Dx, V^\Dx$ are displayed as a function of mesh resolution. 
\begin{figure}[h]
	\begin{subfigure}[b]{0.48\textwidth}
		\InputImage{\textwidth}{0.8\textwidth}{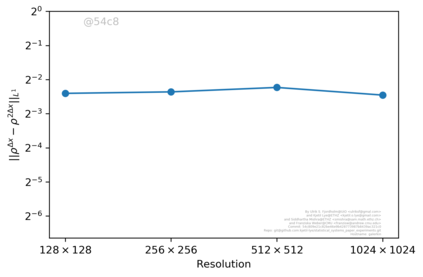}
		\caption{Sample, $\Psi^\Delta =\rho^\Delta$}
	\end{subfigure}
	\begin{subfigure}[b]{0.48\textwidth}
		\InputImage{\textwidth}{0.8\textwidth}{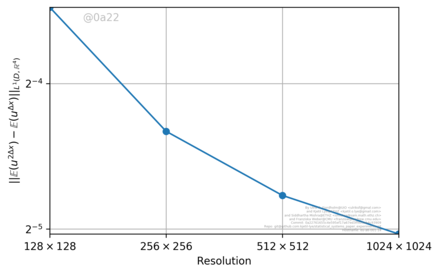}
		\caption{Mean, $\Psi^\Delta = M^\Delta$}
	\end{subfigure}
\begin{subfigure}[b]{0.48\textwidth}
	\InputImage{\textwidth}{0.8\textwidth}{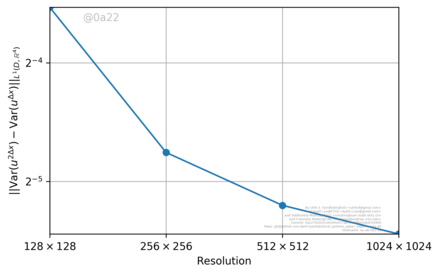}
	\caption{Variance, $\Psi^\Delta = V^\Delta$}
\end{subfigure}
\caption{Cauchy rates \eqref{eq:crate} for the approximate density in the  Kelvin--Helmholtz problem \eqref{eq:kh} for different mesh resolutions. The scheme used is a HLL3 flux with a WENO2 reconstruction algorithm. In this experiment, $\epsilon=0.01$. Here $T=2$.}
\label{fig:KHconv1}
\end{figure}

\begin{figure}[h]
	
	\begin{subfigure}[b]{\textwidth}
		\InputImage{0.24\textwidth}{0.8\textwidth}{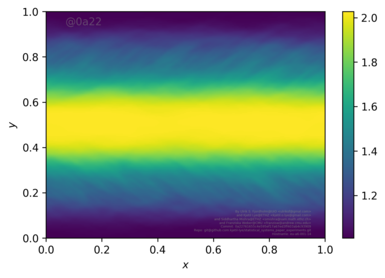}
		\InputImage{0.24\textwidth}{0.8\textwidth}{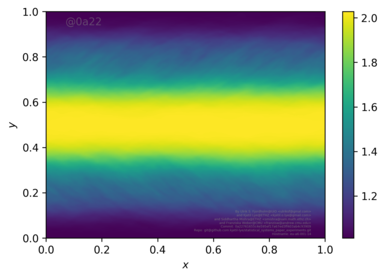}
		\InputImage{0.24\textwidth}{0.8\textwidth}{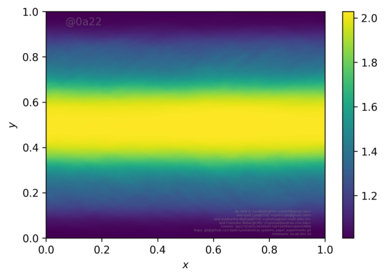}
		\InputImage{0.24\textwidth}{0.8\textwidth}{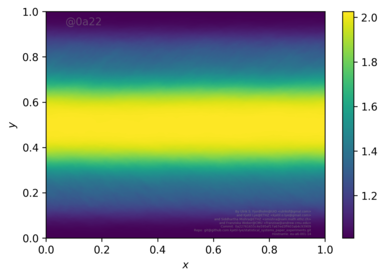}
		\caption{Mean at different resolutions.}
	\end{subfigure}

	\begin{subfigure}[b]{\textwidth}
		\InputImage{0.24\textwidth}{0.8\textwidth}{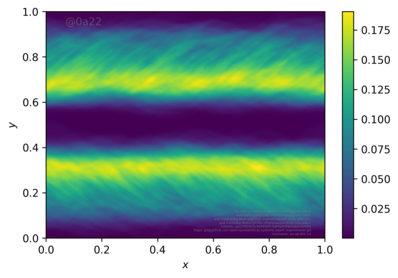}
		\InputImage{0.24\textwidth}{0.8\textwidth}{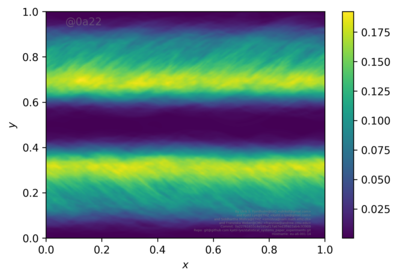}
		\InputImage{0.24\textwidth}{0.8\textwidth}{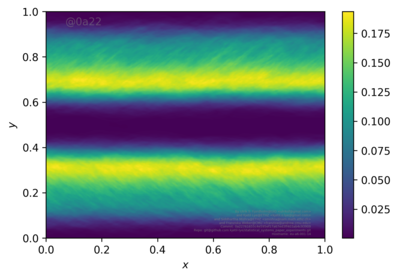}
		\InputImage{0.24\textwidth}{0.8\textwidth}{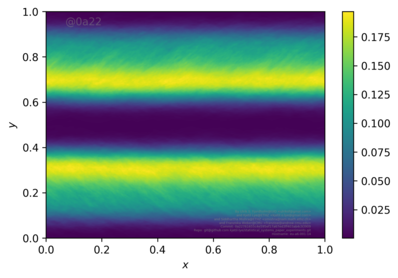}
		\caption{Variance at different resolutions.}
	\end{subfigure}
	\caption{The approximate mean (top row) and variance (bottom row) of the density of the Kelvin--Helmholtz instability \eqref{eq:kh} for mesh resolutions of (from left to right) $128^2$, $256^2$, $512^2$ and $1024^2$ points. The scheme used is an HLL3 flux with a WENO2 reconstruction algorithm. In this experiment, $\epsilon=0.01$. The number of samples used, $M$, was set equal to the resolution $N$, and $T=2$.}
	\label{fig:KHmandv}
\end{figure}
To quantify the convergence of the distribution of $\mu^{\Dx,M}$ we consider the following Cauchy rates, 
\begin{equation}\label{eq:LpWp}
    {\mathcal W}^{k,\Dx}_p(T):= \left( \int_{D^k} \left(W_p\big(\nu^{k,\Dx}_{x,T},\nu^{k,\frac{\Dx}{2}}_{x,T}\big)\right)^p dx \right)^{\frac{1}{p}},
\end{equation}
with $W_p$ being the Wasserstein metric defined in \eqref{eq:wasserdef} and $\nu^{k,\Dx}$, the $k$-th correlation marginal, corresponding to the (approximate) statistical solution $\mu_t^{\Dx,M}$ generated by the Monte Carlo Algorithm. One can check that for all $k \in \N$
\begin{equation}
    \label{eq:wcauchy1}
{\mathcal W}^{k,\Dx}_p(t) \leq C(k,p) W_p\left(\mu^{\Dx}_t, \mu^{\frac{\Dx}{2}}_t \right) \quad \text{for a.e. }t\in (0,T)
\end{equation}
(see Appendix \ref{app:wcauchyproof} for a proof). As the Wasserstein metric metrizes the weak topology on probability measures, we may conclude from Theorems \ref{thm:timedepcmcompactness} and \ref{theo:conv} that under the assumptions of some form of time continuity, the right hand side of \eqref{eq:wcauchy1} goes to zero as $\Dx \rightarrow 0$. This convergence is verified in Figure \ref{fig:KHconv2} (A), where we plot the corresponding Cauchy rates for the distance \eqref{eq:LpWp} with respect to the density (see Appendix~\ref{app:wasserstein} for details about how the Wasserstein distance was computed numerically).

Next, we consider computation of observables with respect to the \emph{second correlation marginal} $\nu^{2,\Delta}_{x,y,t}$ of the approximate statistical solution $\mu^{\Dx,M}_t$. The most interesting observable in this regard is the approximate \emph{structure function} $S^p_r(\mu^{\Dx,M}_t,T)$ \eqref{eq:sf}. This is clearly an admissible observable in the sense of Theorem \ref{thm:timedepcmcompactness}. For computational purposes, it is easier to compute the \emph{time-sections} of the structure function, namely
\begin{equation}    \label{eq:sft}
    \omega_r^p\big(\nu^{2,\Dx,M}_t\big) := \int_D \intavg_{B_r(x)} \Ypair{\nu^{2,\Dx,M}_{x,y,t}}{|\xi_1-\xi_2|^p}\,dydx
\end{equation}
for $t\in[0,T)$. In Figure \ref{fig:KHsf} we plot $\omega^2_{r}\big(\nu^{2,\Dx,M}_T\big)^{1/p}$ for a sequence of mesh sizes $\Dx$. Moreover, we consider three different setups in the figure. In Figure \ref{fig:KHsf} (A) and (B), we set $\epsilon = 0.1$ in \eqref{eq:kh_pert} and $T=2$ and $T=4$, respectively, and in Figure \ref{fig:KHsf} (C) we set $\epsilon = 0.01$ and $T=2$. As seen from the figures, it is clear that the approximate structure functions converge as the mesh is refined. Moreover, the structure functions (approximately) behave as 
\begin{equation}
    \label{eq:sfnum}
  \omega^p_{r}\big(\nu^{2,\Dx,M}_T\big)^{1/p} \sim C(T) r^{\theta_p(T)}.
\end{equation}
The computed values of $\theta_p(T)$ are seen in the legend in Figure \ref{fig:KHsf}. 

The numerical convergence of structure functions is further verified in Figure \ref{fig:KHconv2} (B), where we plot the Cauchy rates \eqref{eq:crate} with $g(r,T) =\omega^2_{r}\big(\nu^{2,\Dx,M}_T\big)$ with $p=1$. In this figure $T=2$ and $\epsilon = 0.01$. Clearly, the structure function (as a function of length scale $r$) converges as the mesh is refined.

\begin{figure}[h]
	\begin{subfigure}[b]{0.48\textwidth}
		\captionsetup{width=.9\linewidth}%
		\InputImage{\textwidth}{0.8\textwidth}{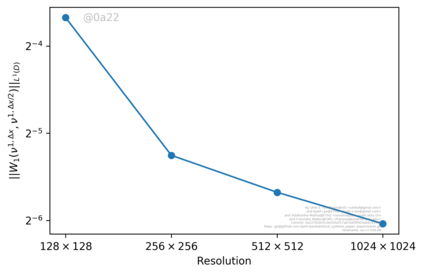}
		\caption{${\mathcal W}_1^{1,\Delta}(T)$ as in \eqref{eq:LpWp} (vertical axis) plotted against $\Delta^{-1}$ (horizontal axis) with $\epsilon = 0.05$ and $T=2$.}
	\end{subfigure}
	\begin{subfigure}[b]{0.48\textwidth}
		\captionsetup{width=.9\linewidth}%
		\InputImage{\textwidth}{0.8\textwidth}{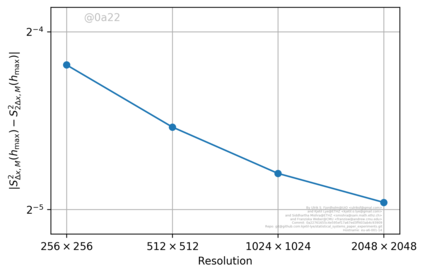}
		\caption{Cauchy rates \eqref{eq:crate} with $\Psi^\Delta(r,T)=\omega^2_{r}\big(\nu^{2,\Dx,M}_T\big)$ with $p=1$ and $r=1/32$, $T=2$ and $\epsilon = 0.01$.}
	\end{subfigure}
\begin{subfigure}[b]{0.48\textwidth}
	\InputImage{\textwidth}{0.8\textwidth}{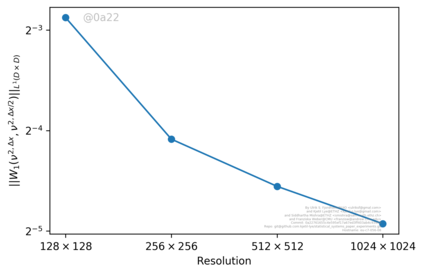}
	\caption{${\mathcal W}_1^{2,\Delta}(T)$ \eqref{eq:LpWp} (vertical axis) vs $\Delta^{-1}$ (horizontal axis) with $\epsilon = 0.05$ and $T=2$.}
\end{subfigure}
\caption{Cauchy rates for the Kelvin--Helmholtz problem.} 
\label{fig:KHconv2}
\end{figure}

\begin{figure}[h]
\begin{subfigure}[b]{0.48\textwidth}
	\InputImage{\textwidth}{0.8\textwidth}{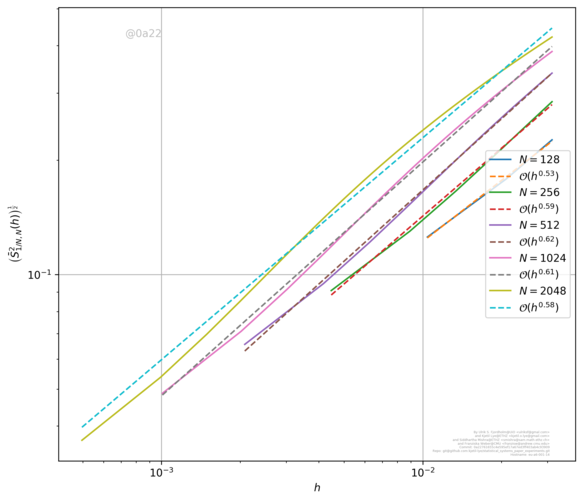}
	\caption{$p=2$, $\epsilon=0.1$, $T=2$}
\end{subfigure}
\begin{subfigure}[b]{0.48\textwidth}
	\InputImage{\textwidth}{0.8\textwidth}{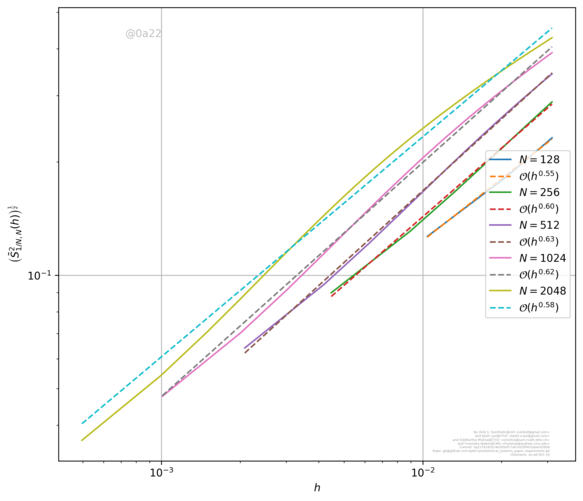}
		\caption{$p=2$, $\epsilon=0.1$, $T=4$}
\end{subfigure}
\begin{subfigure}[b]{0.48\textwidth}
	\InputImage{\textwidth}{0.8\textwidth}{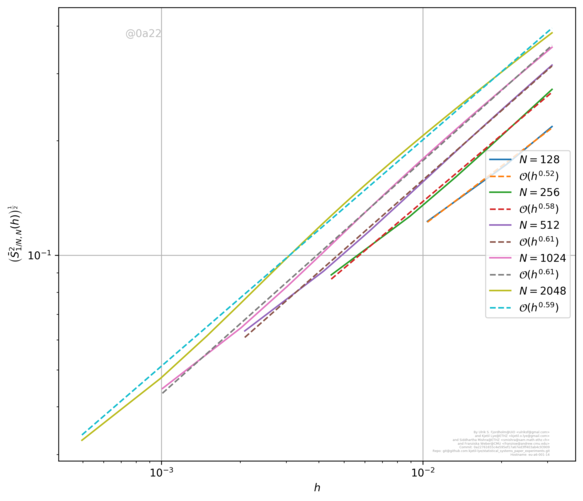}
	\caption{$p=2$, $\epsilon=0.01$, $T=2$}
\end{subfigure}
\caption{Structure functions \eqref{eq:sfnum} for the Kelvin--Helmholtz instability \eqref{eq:kh} for different times $T$, exponents $p$ and perturbation sizes $\epsilon$.  The scheme used is a HLL3 flux with a WENO2 reconstruction algorithm. At each mesh resolution $N$, $M=N$ samples were used.}
\label{fig:KHsf}
\end{figure}

Finally, we compute the Wasserstein Cauchy rates ${\mathcal W}_1^{2,\Delta}(T)$ \eqref{eq:LpWp} with respect to the density, over successively refined mesh sizes and display the result in figure \ref{fig:KHconv2} (c). This figure clearly shows that there is convergence (with respect to mesh resolution and number of Monte Carlo samples) in this metric.

We are interested in computing the two-point correlation marginal $\nu^{2,\Dx,M}_{x,y,t}$ for point pairs $(x,y) \in D$. In this context, we realize this Young measure by empirical histograms, plotted in Figure \ref{fig:KHhist}. In this figure we show the empirical histogram of the two-point correlation Young measure of the density, on successively refined grids, for two different point pairs $x = (0.7,0.7),$ $y =(0.4,0.2)$ (top 2 rows) and $x = (0.7,0.7)$, $y =(0.7,0.8)$ (bottom 2 rows). We see from this figure that the empirical histograms of the two-point correlation marginals converge (visually) on mesh refinement. Moreover, there is a clear difference in the correlation structures at different point pairs.

\begin{figure}[htbp]
	\begin{subfigure}[b]{0.48\textwidth}
		\InputImage{\textwidth}{0.8\textwidth}{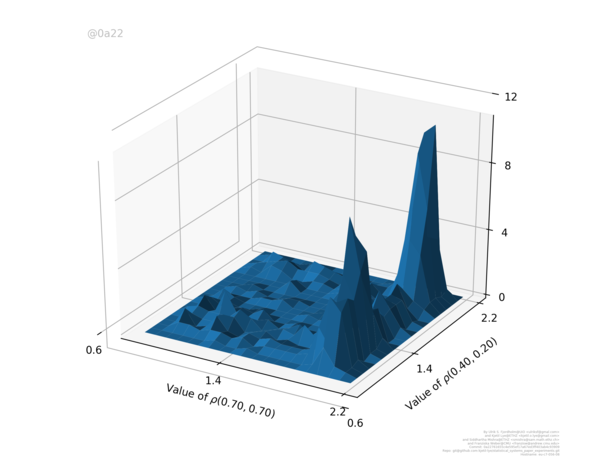}
		\caption{$N=128$}
	\end{subfigure}
	\begin{subfigure}[b]{0.48\textwidth}
		\InputImage{\textwidth}{0.8\textwidth}{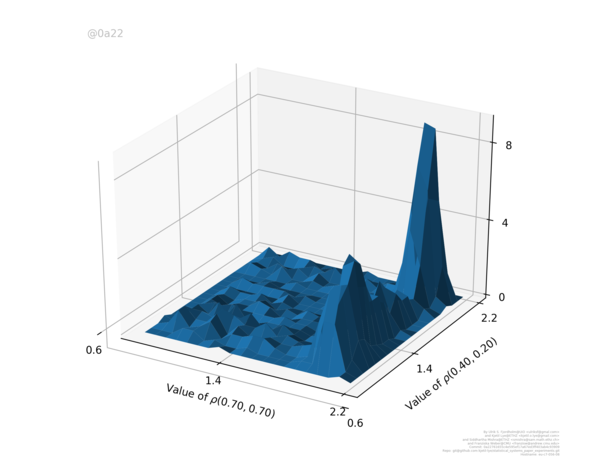}
		\caption{$N=256$}
	\end{subfigure}
	\begin{subfigure}[b]{0.48\textwidth}
		\InputImage{\textwidth}{0.8\textwidth}{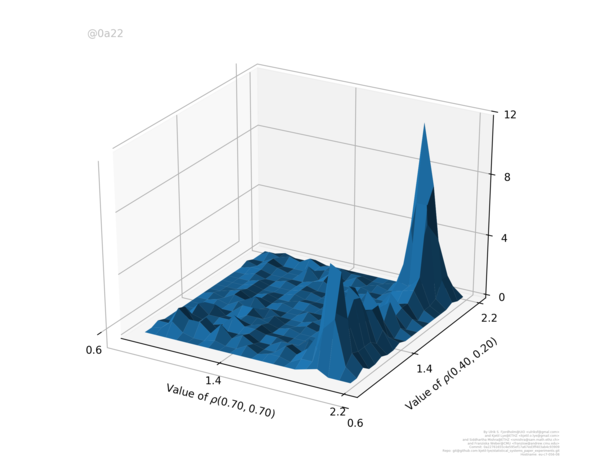}
		\caption{$N=512$}
	\end{subfigure}
	\begin{subfigure}[b]{0.48\textwidth}
		\InputImage{\textwidth}{0.8\textwidth}{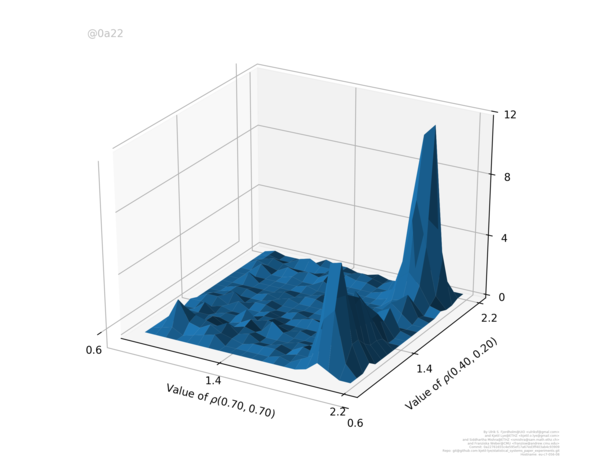}
		\caption{$N=1024$}
	\end{subfigure}
	
	\begin{subfigure}[b]{0.48\textwidth}
		\InputImage{\textwidth}{0.8\textwidth}{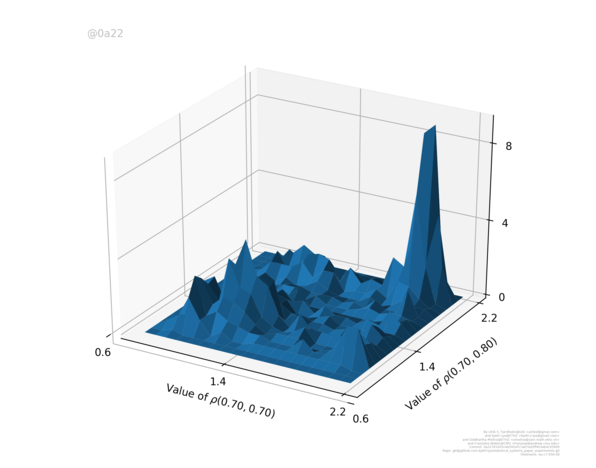}
		\caption{$N=128$}
	\end{subfigure}
	\begin{subfigure}[b]{0.48\textwidth}
		\InputImage{\textwidth}{0.8\textwidth}{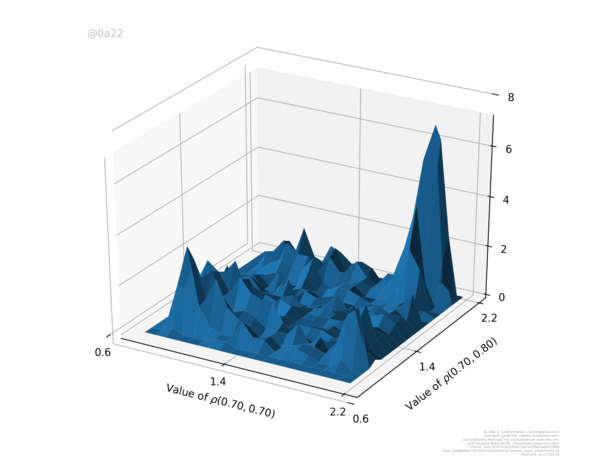}
		\caption{$N=256$}
	\end{subfigure}
	\begin{subfigure}[b]{0.48\textwidth}
		\InputImage{\textwidth}{0.8\textwidth}{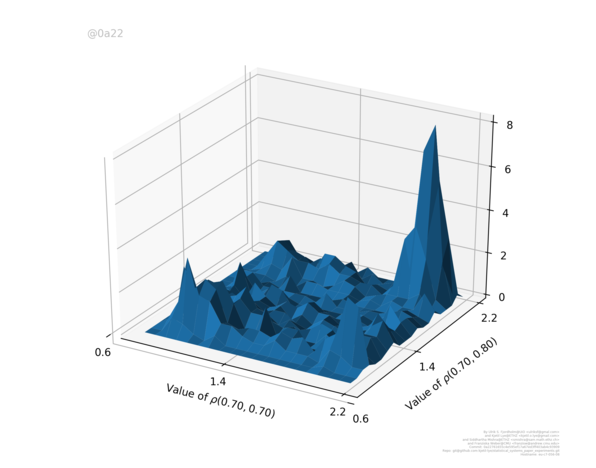}
		\caption{$N=512$}
	\end{subfigure}
	\begin{subfigure}[b]{0.48\textwidth}
		\InputImage{\textwidth}{0.8\textwidth}{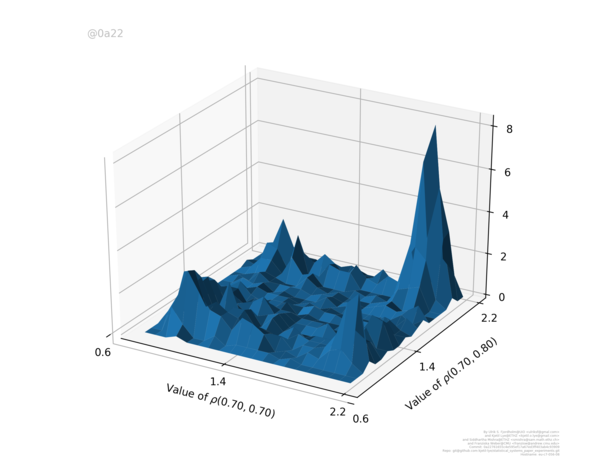}
		\caption{$N=1024$}
	\end{subfigure}
	\caption{Two-dimensional histograms for the correlation measure at $((0.7,0.7), (0.4,0.2))$ (top two rows) and $((0.7,0.7), (0.4,0.2))$ (bottom two rows) for different resolutions for the density in the Kelvin--Helmholtz problem \eqref{eq:kh}. The scheme used is an HLL3 flux with a WENO2 reconstruction algorithm. Here, $T=2$ and $\epsilon=0.05$, and we use $M=1024$ samples.}
	\label{fig:KHhist}
\end{figure}
\subsection{Richtmeyer--Meshkov problem}\label{sec:rm}
Our second test case is the well-studied Richtmeyer--Meshkov problem (see \cite{fkmt} and references therein), which involves a very complicated solution of the compressible Euler equations \eqref{eq:euler}, modeling the complex interaction of strong shocks with unstable interfaces. The underlying initial data is given as
\begin{equation}\label{eq:rm}
p(x)=\begin{cases}
20 & \text{if } |x|<0.1\\
1 & \text{otherwise.}
\end{cases}\qquad \rho(x) = \begin{cases} 2 & \text{if }   |x|< I(x,\omega)\\
1 & \text{otherwise}\end{cases} \quad w^x=w^y=0
\end{equation}
We assign periodic boundary conditions on $D=[0,1]^2$. The interface between the two states is given as
\begin{equation}
I(x, \omega) = 0.25+\epsilon \sum_{j=1}^K a_j(\omega)\sin(2\pi(x+b_j(\omega))),
\end{equation}
where $K=10$, $\epsilon>0$, and $\{a_j\}$ and $\{b_j\}$ are uniform random variables on the interval $[0,1]$. We normalize the $a_j$ such that $\sum_j a_j=1$. The initial probability measure $\bar{\mu}$ is given by the law of the above random field, and lies in $\Prob(L^p(D))$ for every $1 \leq p<\infty$. 

As in the case of the Kelvin--Helmholtz problem, there is no convergence (on mesh refinement) for single samples (realizations). This non-convergence is demonstrated in Figure \ref{fig:RMconv}(A), where the Cauchy rates \eqref{eq:crate} with respect to the density at time $T=5$ are shown. We visualize the density for different mesh resolutions in Figure \ref{fig:RMvis} (top row). As seen from this figure, the solution at this time is very complicated on account of the interaction between the incoming strong shock (which has been reflected, due to periodic boundary conditions) and the unstable interface, which leads to the generation of turbulent small scale eddies. 

On the other hand, and as predicted by the convergence theory developed in Section \ref{sec:fvm}, statistical observables such as the mean and the variance \eqref{eq:mandv} converge on mesh refinement, as shown in Figure \ref{fig:RMvis} (middle and bottom). Furthermore, this figure shows how the small scale structures are averaged out in the mean, whereas the small scale information is encoded in the variance, which is concentrated around the mixing zone. We also verify the convergence of the Wasserstein distance ${\mathcal W}_1^{2,\Delta}(T)$ with respect to the density for successively refined meshes in Figure \ref{fig:RMconv} (B). 

\begin{figure}[h]
	\begin{subfigure}[b]{0.48\textwidth}
		\InputImage{\textwidth}{0.8\textwidth}{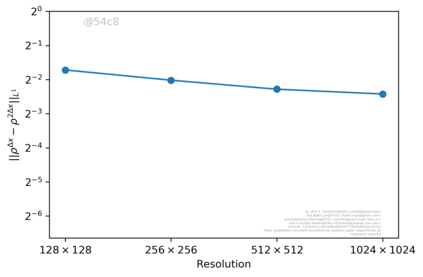}
		\caption{A single sample, $\Psi^\Delta =\rho^\Delta$ in~\eqref{eq:crate}.}
	\end{subfigure}
	\begin{subfigure}[b]{0.48\textwidth}
	\InputImage{\textwidth}{0.8\textwidth}{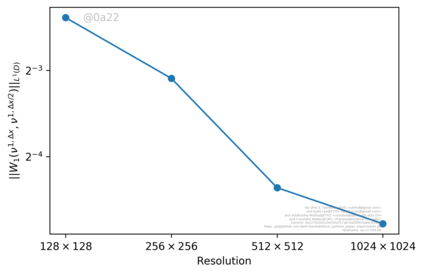}
	\caption{${\mathcal W}_1^{1,\Delta}(T)$, defined in \eqref{eq:LpWp}.}
\end{subfigure}
	\begin{subfigure}[b]{0.48\textwidth}
		\InputImage{\textwidth}{0.8\textwidth}{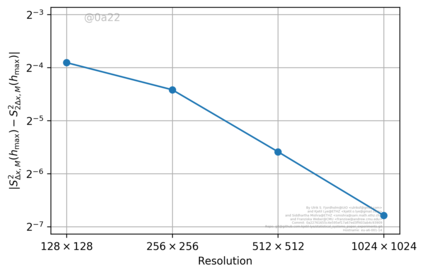}
		\caption{$\Psi^\Delta=\omega^2_{r}(\nu^{2,\Dx,M}_T)$ with $p=1$ and \\$r=1/32$.}
	\end{subfigure}
\begin{subfigure}[b]{0.48\textwidth}
	\InputImage{\textwidth}{0.8\textwidth}{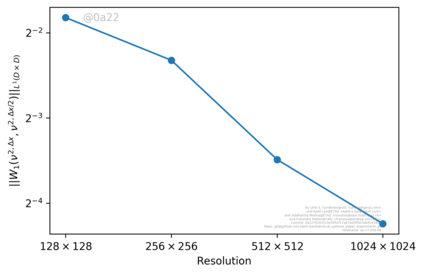}
	\caption{${\mathcal W}_1^{2,\Delta}(T)$ .\\\vspace{\baselineskip}}
\end{subfigure}
\caption{Cauchy rates \eqref{eq:crate} (vertical axis) versus $\Dx$ (horizontal axis) for the Richtmeyer--Meshkov problem using $\epsilon = 0.06$ and $T=5$.
\label{fig:RMconv}}
\end{figure}
\begin{figure}[h]
\centering
	\begin{subfigure}[b]{\textwidth}
		\InputImage{0.24\textwidth}{0.2\textwidth}{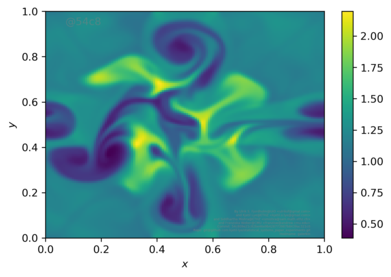}
		\InputImage{0.24\textwidth}{0.2\textwidth}{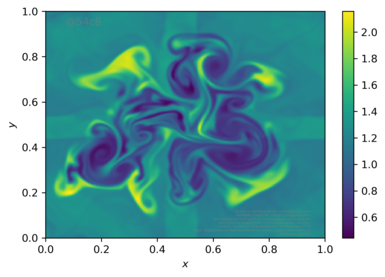}
		\InputImage{0.24\textwidth}{0.2\textwidth}{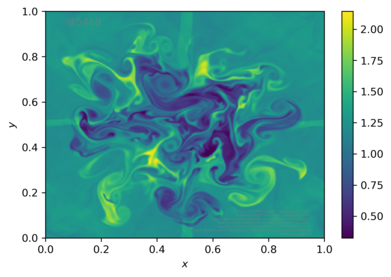}
		\InputImage{0.24\textwidth}{0.2\textwidth}{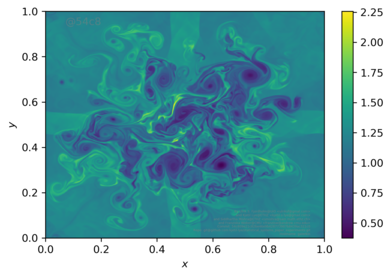}
		\caption{Single sample at different resolutions.}
	\end{subfigure}
	\begin{subfigure}[b]{\textwidth}
		\InputImage{0.24\textwidth}{0.8\textwidth}{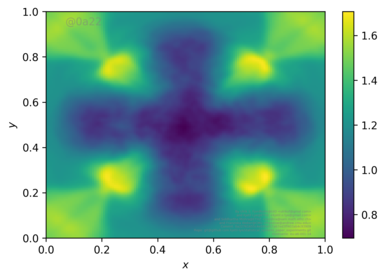}
		\InputImage{0.24\textwidth}{0.8\textwidth}{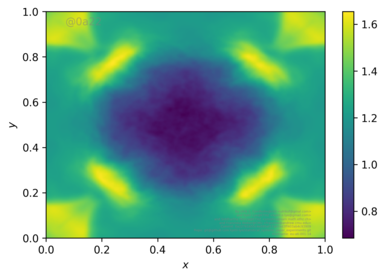}
		\InputImage{0.24\textwidth}{0.8\textwidth}{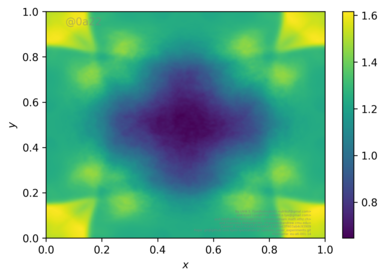}
		\InputImage{0.24\textwidth}{0.8\textwidth}{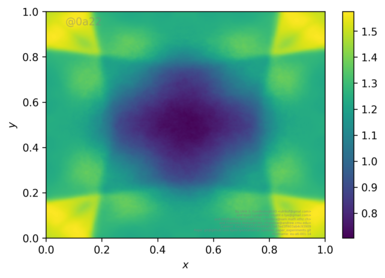}
		\caption{Mean at different resolutions.}
	\end{subfigure}
	\begin{subfigure}[b]{\textwidth}
		\InputImage{0.24\textwidth}{0.8\textwidth}{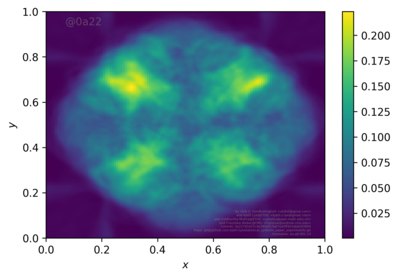}
		\InputImage{0.24\textwidth}{0.8\textwidth}{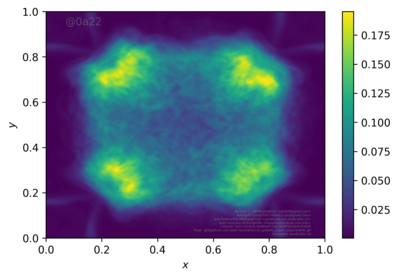}
		\InputImage{0.24\textwidth}{0.8\textwidth}{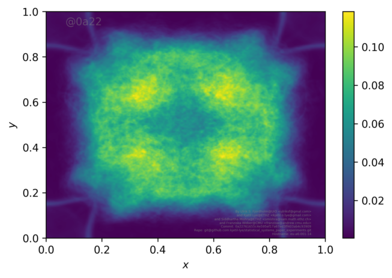}
		\InputImage{0.24\textwidth}{0.8\textwidth}{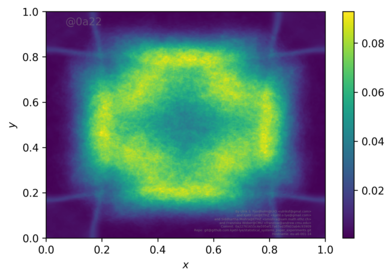}
		\caption{Variance at different resolutions.}
	\end{subfigure}
	\caption{Approximate density for the Richtmeyer--Meshkov problem \eqref{eq:rm} using $\epsilon = 0.06$ and at $T=5$. All results are based on a scheme with the HLLC flux and MC reconstruction, computed at resolutions with (from left to right) $128^2$, $256^2$, $512^2$ and $1024^2$ points.}
	\label{fig:RMvis}
\end{figure}

Next, we compute the time sections of the {structure function} $\omega_{r}^2(\nu^{2,\Dx,M}_t)$ defined in \eqref{eq:sft}. These are shown in Figure \ref{fig:RMsf}, where we have used $T=5$, $r \in \left[1/1024,1/32\right]$, and two different values of the perturbation parameter in \eqref{eq:rm}. As seen from the figure, the structure function clearly converges on mesh refinement. Moreover, it behaves as in \eqref{eq:sfnum}, with exponents shown in Figure \ref{fig:RMsf}. The convergence of the structure function is further verified by plotting the Cauchy rates for the structure function, as a function of the length scale $r$ in Figure \ref{fig:RMconv} (C). 

\begin{figure}[h]
\begin{subfigure}[b]{0.48\textwidth}
	\InputImage{\textwidth}{0.8\textwidth}{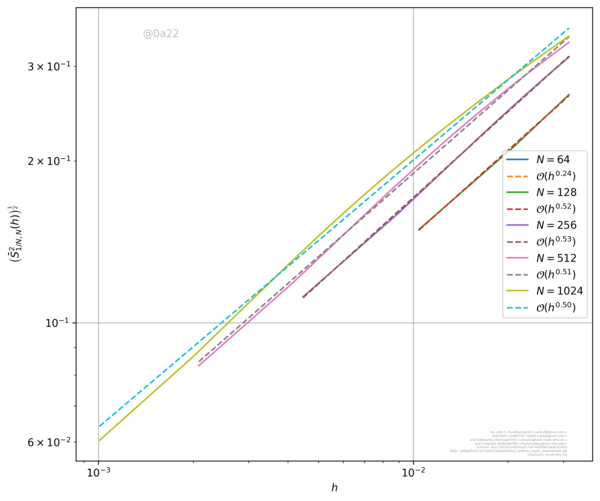}
	\caption{$p=2$, $\epsilon=0.025$ and $T=5$}
\end{subfigure}
\begin{subfigure}[b]{0.48\textwidth}
	\InputImage{\textwidth}{0.8\textwidth}{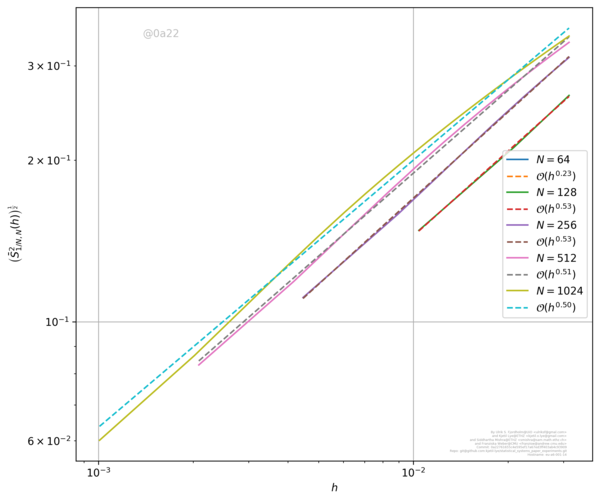}
		\caption{$p=2$, $\epsilon=0.06$ and $T=5$}
\end{subfigure}

\caption{Structure functions for the Richtmeyer--Meshkov problem \eqref{eq:rm} for different perturbation sizes $\epsilon$.  The scheme used is a HLL3 flux with a MC reconstruction algorithm. At each mesh resolution $N$, $M=N$ samples were used.}
\label{fig:RMsf}
\end{figure}

In Figure \ref{fig:RMconv}(D) we plot the Wasserstein distance ${\mathcal W}_1^{2,\Delta}(T)$ for the density and $T=5$, on a sequence of successively refined meshes. As shown in the figure, this distance converges on mesh (and sample) refinement.

Finally, in Figure \ref{fig:RMhist}, we plot histograms that represent the two-point correlation measure for the density at two different point pairs and at time $T=5$. These histograms show that the two-point correlation structure for this initial datum is very different from the correlation structure for the Kelvin--Helmholtz problem (Figure \ref{fig:KHhist}). 
\begin{figure}[htbp]
	\begin{subfigure}[b]{0.48\textwidth}
		\InputImage{\textwidth}{0.8\textwidth}{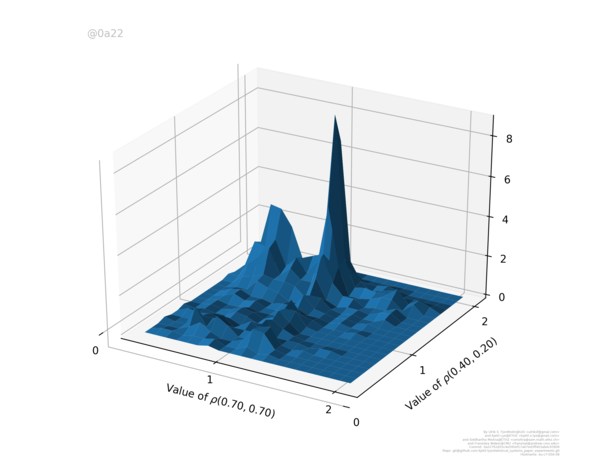}
		\caption{$N=128$}
	\end{subfigure}
	\begin{subfigure}[b]{0.48\textwidth}
		\InputImage{\textwidth}{0.8\textwidth}{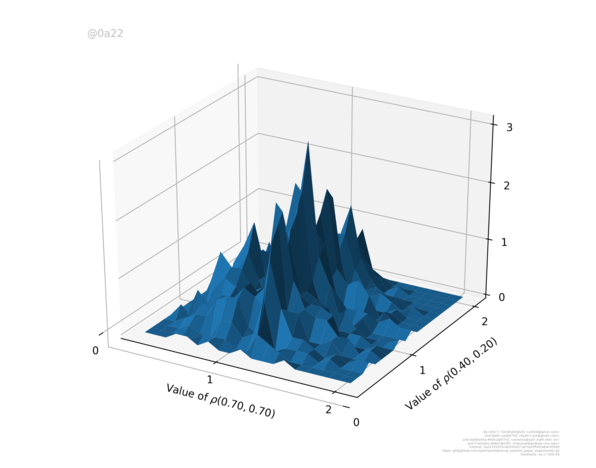}
		\caption{$N=256$}
	\end{subfigure}
	\begin{subfigure}[b]{0.48\textwidth}
		\InputImage{\textwidth}{0.8\textwidth}{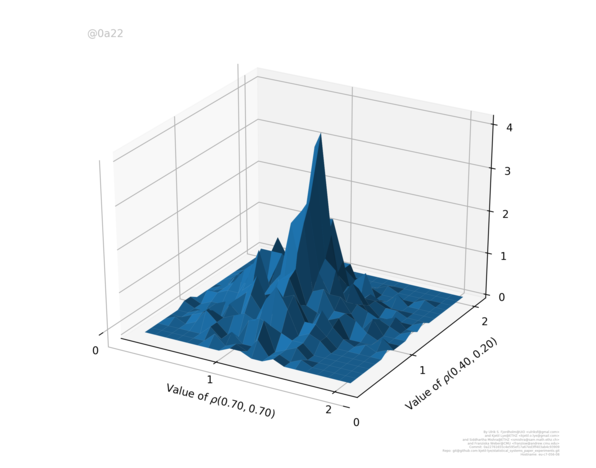}
		\caption{$N=512$}
	\end{subfigure}
	\begin{subfigure}[b]{0.48\textwidth}
		\InputImage{\textwidth}{0.8\textwidth}{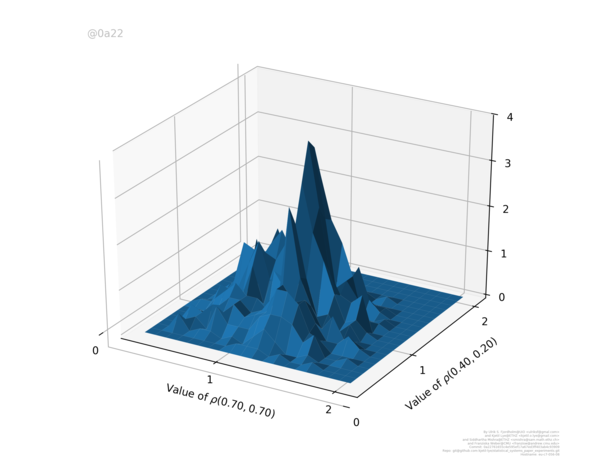}
		\caption{$N=1024$}
	\end{subfigure}
	
	\begin{subfigure}[b]{0.48\textwidth}
		\InputImage{\textwidth}{0.8\textwidth}{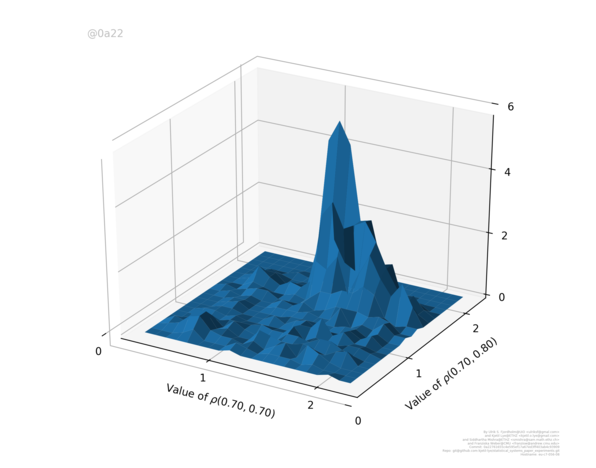}
		\caption{$N=128$}
	\end{subfigure}
	\begin{subfigure}[b]{0.48\textwidth}
		\InputImage{\textwidth}{0.8\textwidth}{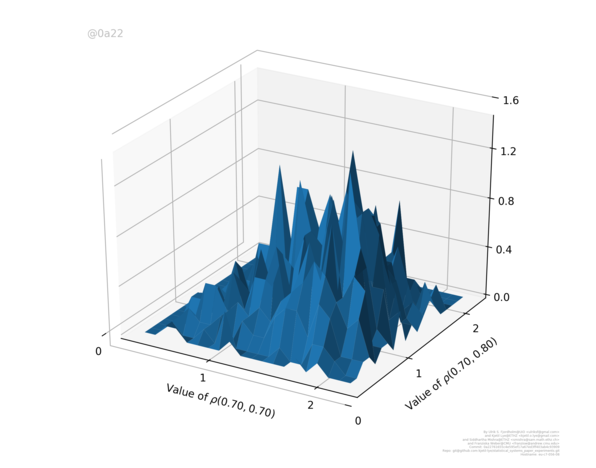}
		\caption{$N=256$}
	\end{subfigure}
	\begin{subfigure}[b]{0.48\textwidth}
		\InputImage{\textwidth}{0.8\textwidth}{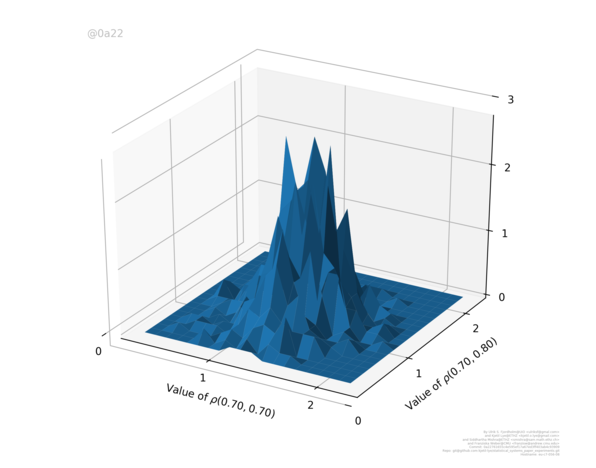}
		\caption{$N=512$}
	\end{subfigure}
	\begin{subfigure}[b]{0.48\textwidth}
		\InputImage{\textwidth}{0.8\textwidth}{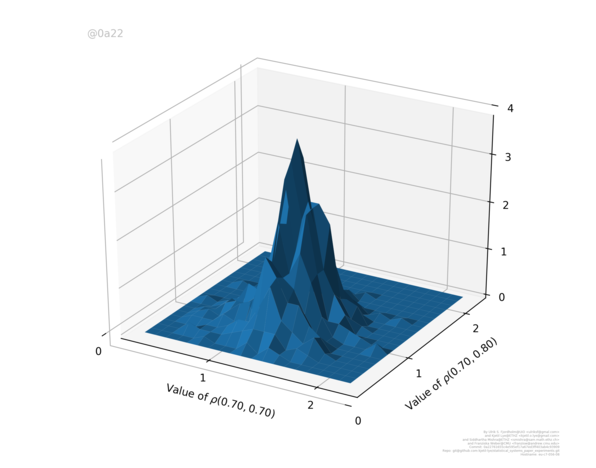}
		\caption{$N=1024$}
	\end{subfigure}
	\caption{Two-dimensional histograms for the correlation measure at $((0.7,0.7), (0.4,0.2))$ (top two rows) and $((0.7,0.7), (0.4,0.2))$ (bottom two rows) for different resolutions for the density in the Richtmeyer--Meshkov problem \eqref{eq:rm}. The scheme used is an HLL3 flux with an MC reconstruction algorithm. Here, $T=5$ and $\epsilon=0.06$, and we use $M=1024$ samples.}
	\label{fig:RMhist}	
\end{figure}

\subsection{Fractional Brownian motion}\label{sec:fbm}
The initial probability measure $\bar{\mu} \in \Prob(L^p(D))$ in the previous two numerical experiments was realized as a probability measure on high- but finite-dimensional subsets. We now consider initial probability measures that are concentrated on genuinely infinite-dimensional subsets of $L^p(D)$. 

We will assume that the initial probability measure for the two-dimensional compressible Euler equations \eqref{eq:euler} will correspond to a \emph{fractional Brownian motion}. Introduced by Mandelbrot et al.~\cite{mandelbrot_fractional}, fractional Brownian motion can be seen as a generalization of standard Brownian motion with a scaling exponent different from $1/2$.

We consider the following initial data,
\begin{align*}
w^{x,H}_0(\omega; x):=B_1^H(\omega; x),\qquad w^{y,H}_0(\omega;x) := B_2^H(\omega; x),\\
\rho_0=4,\qquad p_0=2.5,\qquad \omega\in\Omega,\ x\in [0,1]^2,
\end{align*}
where $B^H_1$ and $B^H_2$ are two independent two dimensional fractional Brownian motions with Hurst index $H\in(0,1)$. Standard Brownian motion corresponds to a Hurst index of $H=1/2$. 

To generate fractional Brownian motion, we use the random midpoint displacement method originally introduced by L\'evy~\cite{Levy1965} for Brownian motion, and later adapted for fractional Brownian motion~\cite{Fournier:1982:CRS:358523.358553,Voss1991}. Consider the uniform grid $0=x_\hf<\dots<x_{N+\hf}=1$ with $x_\iphf=i\Dx$ and $\Dx=\frac{1}{N}$, where $N=2^k+1$ is the number of cells for some $k\in\N$. We first fix the corners
\[w^{x, H, \Delta x}_{1,N}(\omega; 0) =w^{x, H, \Delta x}_{1,1}(\omega; 0)=w^{x, H, \Delta x}_{N,N}(\omega; 0)=w^{x, H, \Delta x}_{N,1}(\omega; 0)=0 \qquad \omega\in\Omega, \]
Recursively update the values on the edges as
\begin{align*}
w^{x, H, \Delta x}_{2^{k-l-1}(2j+1),2^{k-l}i}(\omega; 0)&=\frac{1}{2}\left (w^{x, H, \Delta x}_{2^{k-l}(j+1), 2^{k-l}i}(\omega; 0)+w^{x, H, \Delta x}_{2^{k-l}j, 2^{k-l}i}(\omega; 0)\right) \\
&\quad+\sqrt{\frac{1-2^{2H-2}}{2^{2lH}}}X_{2^l+j, 2^{k-l}i}(\omega)
\end{align*}
and correspondingly for $w^{x, H, \Delta x}_{2^{k-l}j, 2^{k-l-1}(2i+1)}(\omega; 0)$. For the values in the center of the cells, we use the following expression:
\begin{align*}w^{x, H, \Delta x}_{2^{k-l-1}(2j+1),2^{k-l-1}(2j+1)}(\omega; 0)=\frac{1}{4}\left (w^{x, H, \Delta x}_{2^{k-l}(j+1), 2^{k-l}i}(\omega; 0)+w^{x, H, \Delta x}_{2^{k-l}j, 2^{k-l}i}(\omega; 0)\right.\\
\left. w^{x, H, \Delta x}_{2^{k-l}(j+1), 2^{k-l}(i+1)}(\omega; 0)+w^{x, H, \Delta x}_{2^{k-l}(j+1), 2^{k-l}(i+1)}(\omega; 0)\right)\\
+\sqrt{\frac{1-2^{2H-2}}{2^{2lH}}}X_{2^l+j, 2^{k-l}i}(\omega)
\end{align*}
for $l=0,\ldots,k$ and for $i,j=0,\ldots,2^l$. Here $(X_{k,n})_{(k,n)\in\N^2}$ is a collection of normally distributed random variables with mean 0 and variance 1. That is, we bisect every cell and set the middle value to the average of the neighbouring values plus some Gaussian random variable. The same procedure is repeated for $w^{y,H,\Dx}$. See \Cref{fig:brownian_init} for a sample of the initial velocity field with standard Brownian motion, i.e., with $H=0.5$.

The initial probability measure is given by the law of the above random field and the dimension of its support increases with decreasing mesh size. Hence, in the limit $\Dx\to0$ we are approximating a probability measure supported on an infinite dimensional subspace of $L^2(D;U)$

We compute the statistical solutions with Algorithm~\ref{alg:mc}, with the fractional Brownian motion initial data for two different Hurst indices, $H=0.1$ and $H=0.5$. Statistical observables corresponding to the one-point correlation marginal, such as the mean and variance, converge on mesh refinement, as shown in Figure \ref{fig:BMconv1}. We also plot the mean and variance of the density at the highest mesh resolution of $1024^2$ and time $T=0.25$, for the two different Hurst indices, in Figure \ref{fig:BMmandv}. As seen from the figure, there is a clear difference in the spatial structure of the mean and the variance as the Hurst index is changed. Moreover, the spatial structure of these statistical quantities is much more complicated than in the case of the Kelvin--Helmholtz and Richtmeyer--Meshkov initial data, with no clear large scale structures such as shocks. On the other hand, the statistical quantities have more small-scale structures. This is more pronounced in the $H=0.1$ case than for standard Brownian motion.

\begin{figure}
	\begin{subfigure}[b]{0.48\textwidth}
			\InputImage{\textwidth}{0.8\textwidth}{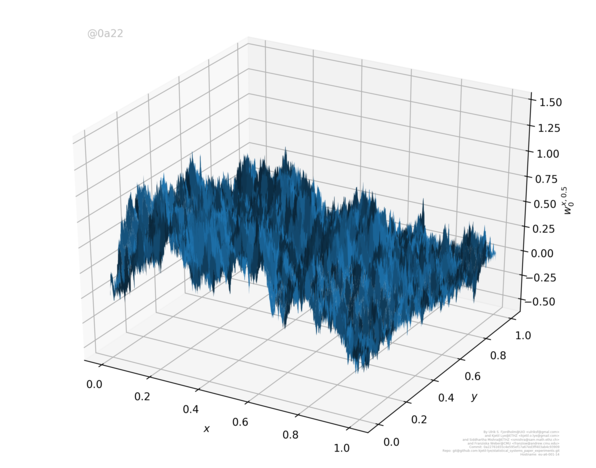}
		\caption{$w^{x,0.5}_0$}
	\end{subfigure}
	\begin{subfigure}[b]{0.48\textwidth}
	\InputImage{\textwidth}{0.8\textwidth}{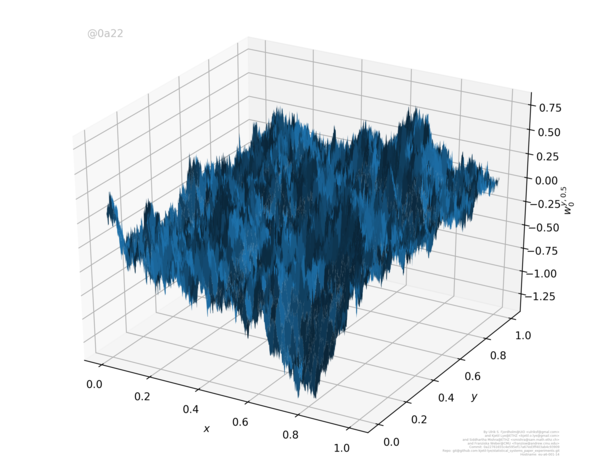}
	\caption{$w^{y,0.5}_0$}
	\end{subfigure}
\caption{Two samples of velocity fields with the Brownian motion initial data.}\label{fig:brownian_init}
\end{figure}

\begin{figure}[h]
	\begin{subfigure}[b]{0.48\textwidth}
		\InputImage{\textwidth}{0.8\textwidth}{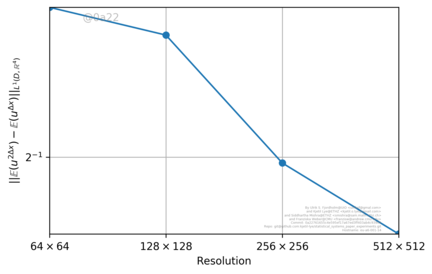}
		\caption{$H=0.1$, mean}
	\end{subfigure}
		\begin{subfigure}[b]{0.48\textwidth}
	\InputImage{\textwidth}{0.8\textwidth}{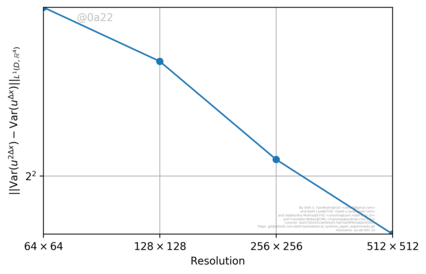}
	\caption{$H=0.1$, variance}
\end{subfigure}
	\begin{subfigure}[b]{0.48\textwidth}
		\InputImage{\textwidth}{0.8\textwidth}{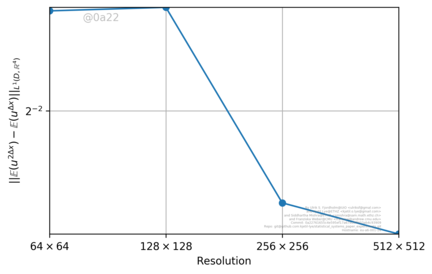}
		\caption{$H=0.5$, mean}
	\end{subfigure}
	\begin{subfigure}[b]{0.48\textwidth}
		\InputImage{\textwidth}{0.8\textwidth}{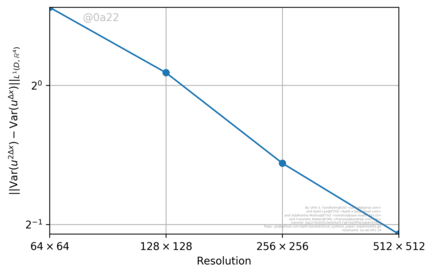}
		\caption{$H=0.5$, variance}
	\end{subfigure}	

\caption{Cauchy rates \eqref{eq:crate} for the mean and the variance (of the density) with respect to fractional Brownian motion initial data at time $T=0.25$. The scheme used is an HLL3 flux with a WENO2 reconstruction algorithm. }
\label{fig:BMconv1}
\end{figure}

\begin{figure}[h]
	\begin{subfigure}[b]{0.24\textwidth}
		\InputImage{\textwidth}{0.8\textwidth}{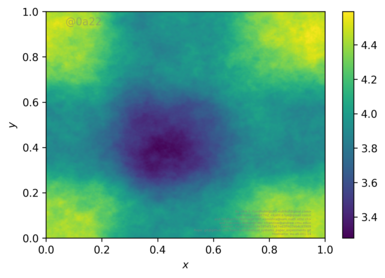}
		\caption{$H=0.1$, mean}
	\end{subfigure}
		\begin{subfigure}[b]{0.24\textwidth}
	\InputImage{\textwidth}{0.8\textwidth}{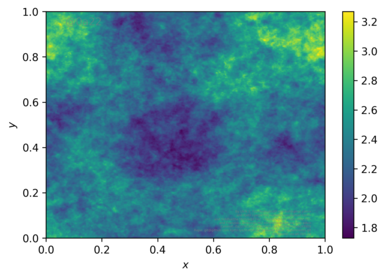}
	\caption{$H=0.1$, variance}
\end{subfigure}
	\begin{subfigure}[b]{0.24\textwidth}
		\InputImage{\textwidth}{0.8\textwidth}{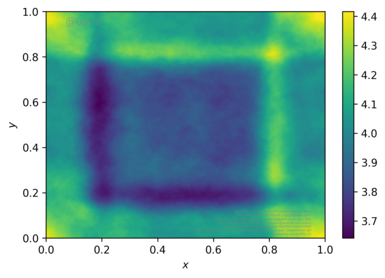}
		\caption{$H=0.5$, mean}
	\end{subfigure}
	\begin{subfigure}[b]{0.24\textwidth}
		\InputImage{\textwidth}{0.8\textwidth}{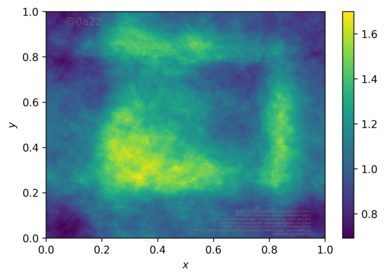}
		\caption{$H=0.5$, variance}
	\end{subfigure}	

\caption{The mean and the variance (of the density) at the highest resolution of $1024^2$ mesh points and $1024$ Monte Carlo samples, for fractional Brownian motion initial data with two different Hurst indices and at time $T=0.25$. The scheme used is an HLL3 flux with a WENO2 reconstruction algorithm. }
\label{fig:BMmandv}
\end{figure}

For $r \in \left[1/1024,1/32\right]$, we plot the (time sections of) the structure function $\omega_{r}^{2}(\nu^{2,\Dx,M}_t)$  at $t=0.25$ in Figure \ref{fig:BMsf}. The structure functions clearly converge on mesh refinement. This is also verified for both Hurst indices in Figure \ref{fig:BMconv2} (A,B), where we plot the Cauchy rates \eqref{eq:crate} for the structure function, with respect to the length scale $r$. Moreover, the structure functions scale as in \eqref{eq:sfnum}. 

In Figure \ref{fig:BMconv2} (C,D), we plot the Cauchy rates with respect to the Wasserstein distance ${\mathcal W}^{2,\Dx}_1(0.25)$, with respect to grid resolution, for both Hurst indices. We verify from this figure that these distances also converge on mesh refinement and sample augmentation.

Finally, in Figure \ref{fig:BMhist}, we plot histograms representing the two-point correlation marginals of the density, computed on the finest grid resolution of $1024^2$, for two different point pairs. The figure shows that the two-point correlation structure is again very different for different Hurst indices, and from the correlation structures for the previous numerical experiments. 

\begin{figure}[h]
	\begin{subfigure}[b]{0.48\textwidth}
		
		\InputImage{\textwidth}{0.8\textwidth}{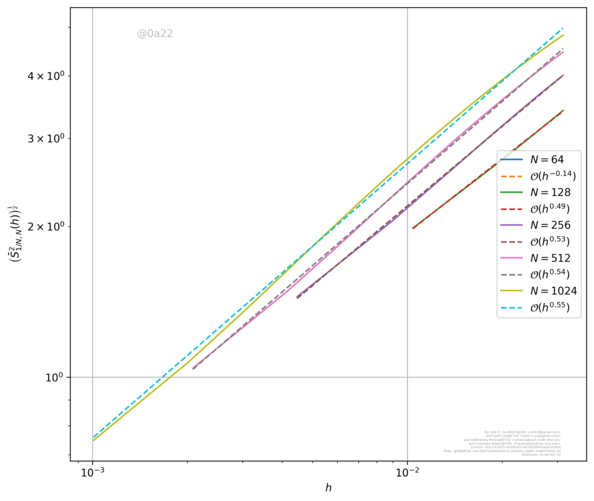}
		\caption{$H=0.1$}
	\end{subfigure}
		\begin{subfigure}[b]{0.48\textwidth}
	\InputImage{\textwidth}{0.8\textwidth}{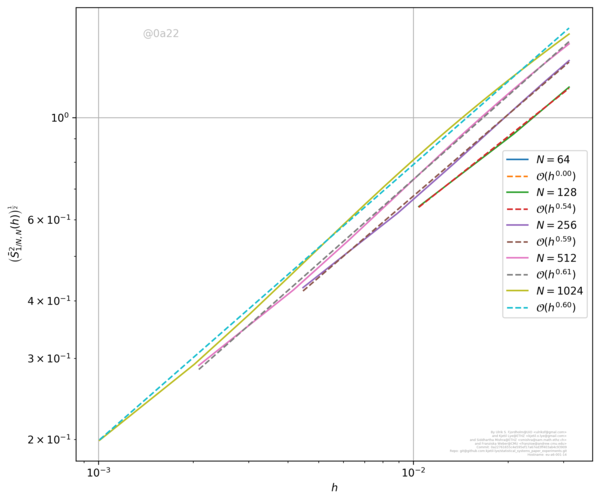}
	\caption{$H=0.5$}
\end{subfigure}

\caption{Structure function \eqref{eq:sft} for $p=2$ for different grid resolutions at time $T=0.25$ for two different Hurst indices, corresponding to the fractional Brownian motion initial data. }
\label{fig:BMsf}
\end{figure}

\begin{figure}[h]
	\begin{subfigure}[b]{0.48\textwidth}
		\InputImage{\textwidth}{0.8\textwidth}{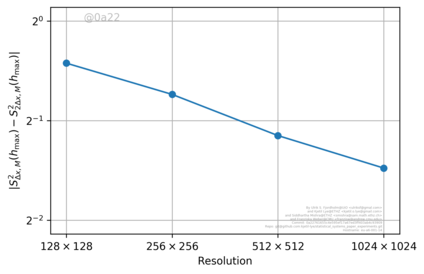}
		\caption{$H=0.1$, $\omega_r^2$}
	\end{subfigure}
		\begin{subfigure}[b]{0.48\textwidth}
	\InputImage{\textwidth}{0.8\textwidth}{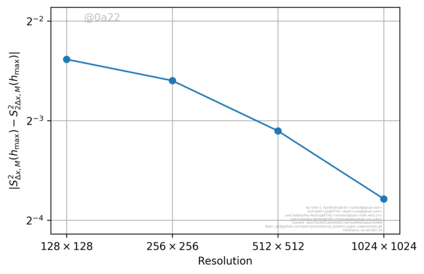}
	\caption{$H=0.5$, $\omega_r^2$}
\end{subfigure}
	\begin{subfigure}[b]{0.48\textwidth}
		\InputImage{\textwidth}{0.8\textwidth}{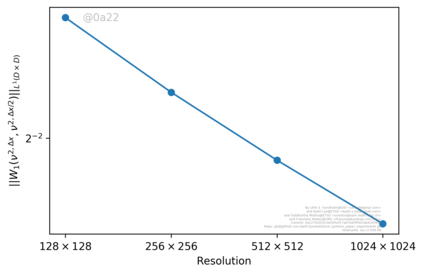}
		\caption{$H=0.1$, ${\mathcal W}^{2,\Dx}_1$}
	\end{subfigure}
	\begin{subfigure}[b]{0.48\textwidth}
		\InputImage{\textwidth}{0.8\textwidth}{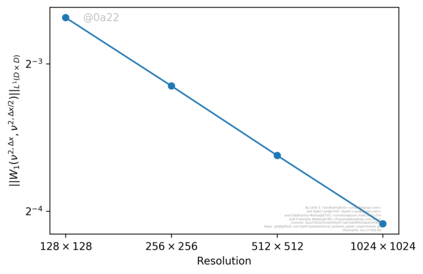}
		\caption{$H=0.5$, ${\mathcal W}^{2,\Dx}_1$ }
	\end{subfigure}	
\caption{Convergence for different two-point statistical observables for the fractional Brownian motion initial data at time $T=0.25$. Top row: (Time sections of) structure function \eqref{eq:sft} with $p=2$. Bottom row: Wasserstein distance \eqref{eq:LpWp} with $p=1$, $k=2$.}
\label{fig:BMconv2}
\end{figure}

\begin{figure}[h]
	\begin{subfigure}[b]{0.48\textwidth}
		\InputImage{\textwidth}{0.8\textwidth}{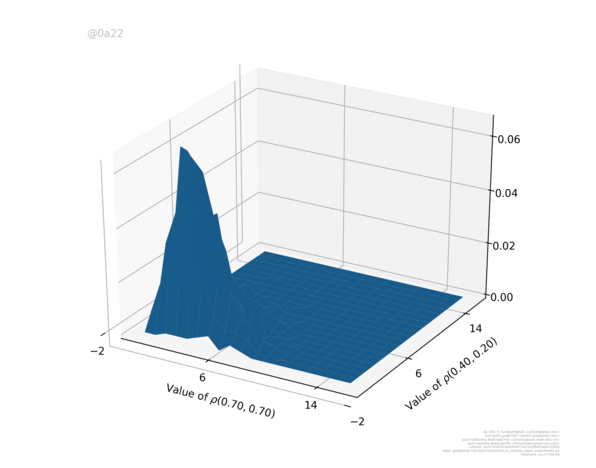}
		\caption{$H=0.1$}
	\end{subfigure}
		\begin{subfigure}[b]{0.48\textwidth}
	\InputImage{\textwidth}{0.8\textwidth}{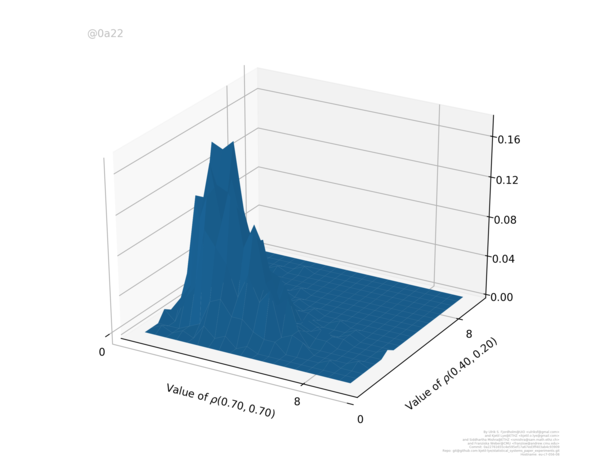}
	\caption{$H=0.5$}
\end{subfigure}
	\begin{subfigure}[b]{0.48\textwidth}
		\InputImage{\textwidth}{0.8\textwidth}{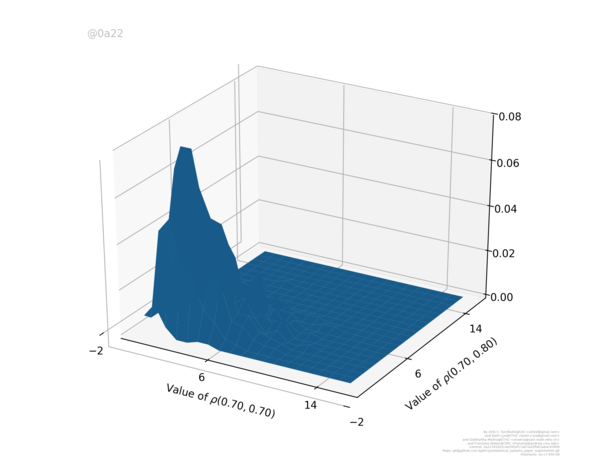}
		\caption{$H=0.1$}
	\end{subfigure}
	\begin{subfigure}[b]{0.48\textwidth}
		\InputImage{\textwidth}{0.8\textwidth}{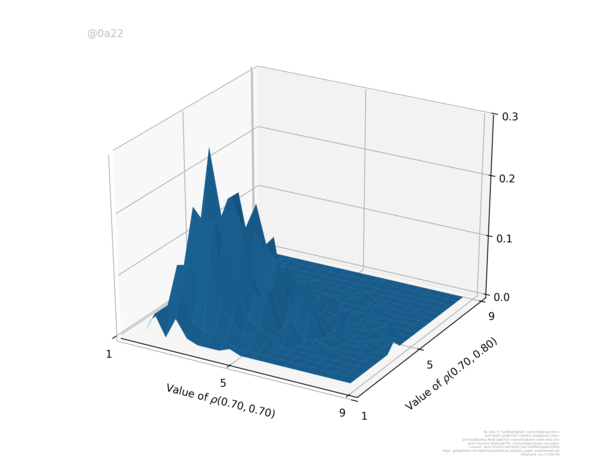}
		\caption{$H=0.5$}
	\end{subfigure}	

\caption{Histograms representing the two-point correlation marginal for the density at time $T=0.25$ and two different Hurst indices, at two different point pairs. Top row: At points $(0.4,0.2)$ and $(0.7,0.8)$. Bottom row: At points $(0.7,0.7)$ and $(0.7,0.8)$. All figures are generated with mesh resolution of $1024^2$ points and with $1024$ samples.}
\label{fig:BMhist}
\end{figure}

\subsection{Stability of the computed statistical solution}\label{sec:stab}
A priori, the computed statistical solution depends on the specifics of the underlying initial probability measure $\bar{\mu}$ as well as on the details of the numerical scheme \eqref{eq:semi_d}, used within the Monte Carlo Algorithm. We investigate the stability of the computed statistical solution with respect to these parameters in the specific case of the Kelvin--Helmholtz problem \eqref{eq:kh}. 

In \cite{fkmt}, the authors had already demonstrated the stability of the computed measure-valued solution with respect to the variations of the underlying numerical method, or to the size and type of perturbations to the Kelvin--Helmholtz initial data. As the computed measure-valued solution in \cite{fkmt} is identical to the first correlation marginal of our computed statistical solution, we can assume that the observables with respect to the first correlation marginal are also stable. Therefore, we investigate the stability of observables with respect to the second correlation marginal $\nu^{2,\Dx,M}$. The results are summarized below.
\begin{itemize}
    \item \emph{Stability with respect to amplitude of perturbations.} We vary the size of the perturbation parameter $\epsilon$ in \eqref{eq:kh_pert} over two orders of magnitude, from $\epsilon = 0.001$ to $\epsilon = 0.1$. The computed (time section of) structure function \eqref{eq:sft} for $p=2$ and at time $T=2$, on the finest resolution of $1024^2$ points and $1024$ Monte Carlo samples is shown in Figure \ref{fig:KHstab} (A, left). As seen from the figure, the computed structure functions are very close to each other and scale as \eqref{eq:sfnum} with $\theta \approx 0.61$. This indicates stability of the computed structure function with respect to the amplitude of perturbations in the initial data. This stability is further verified in Figure \ref{fig:KHstab} (A, right) where we plot the Wasserstein distance $\bigl\|W_1\left(\nu^{2,\epsilon}_{x,T},\nu^{2,\frac{\epsilon}{2}}_{x,T}\right)\bigr\|_{L^1(D^2)}$ with respect to the density, at time $T=2$ for different values of the perturbation parameter. The plot shows (linear) convergence with the decay of the perturbation, indicating stability of the computed statistical solution \emph{vis a vis} perturbation amplitude. 
    
    \item \emph{Stability with respect to type of perturbations.} In all the numerical experiments for the Kelvin--Helmholtz initial data, we have assumed that the random variables $a_j,b_j$ in \eqref{eq:kh_pert} are chosen from a uniform distribution. Here, we choose these random variables from a standard normal distribution. This amounts to varying the corresponding initial probability measure for \eqref{eq:kh}. The consequent change in the structure function \eqref{eq:sft} for two different amplitudes of the perturbation parameter $\epsilon$ in \eqref{eq:kh_pert} are shown in Figure \ref{fig:KHstab} (B, left). The figure clearly shows that the computed structure functions are very close to the ones computed with the uniform distribution. This stability with respect to the type of perturbation is further verified in Figure \ref{fig:KHstab} (B, right) where we plot the $\big\|W_1\left(\nu^{2,\epsilon}_{x,T},\hat{\nu}^{2,\epsilon}_{x,T}\right)\big\|_{L^1(D^2)}$ at time $T=2$. Here, $\nu, \hat{\nu}$ refer to the correlation measures, computed with the uniform and standard normal random variables, respectively. The plot shows convergence with the decay of the perturbation, indicating stability of the computed statistical solution, \emph{vis a vis} perturbation type.
    \item \emph{Stability with respect to choice of numerical scheme.} In order to investigate the stability of the computed statistical solutions to the choice of the underlying numerical scheme in \eqref{eq:semi_d}, we vary the reconstruction procedure, i.e., we use a high-resolution finite volume scheme based on the HLLC flux, but with either MC or WENO2 reconstructions (see e.g.~\cite{leveque_green}). The choice of the reconstruction leads to change in the sub-grid scale numerical viscosity of the overall approximation. We plot the structure function \eqref{eq:sft} in Figure \ref{fig:KHstab} (C, left) and observe a very minor change in the structure function. This issue is investigated further in Figure \ref{fig:KHstab} (C, right) where the Wasserstein distances $\big\|W_1\left(\nu^{2,\Dx}_{\text{WENO2}, x,T},\nu^{2,\Dx}_{\text{MC}, x,T}\right)\big\|_{L^1(D^2)}$ at time $T=2$ are plotted. Here $\nu^{2,\Dx}_{\text{WENO2}}$ is the second correlation marginal, computed with the WENO2 reconstruction procedure, and $\nu^{2,\Dx}_{\text{MC}}$ is the second correlation marginal, computed with the MC reconstruction procedure. We observe convergence of this distance with respect to resolution. This allows us to conclude that the statistical solutions are stable with respect to the choice of the underlying numerical method, at least for this Kelvin--Helmholtz problem. 
    \end{itemize}
    
    \begin{figure}[h]
	\begin{subfigure}[b]{\textwidth}
		\InputImage{0.4\textwidth}{\textwidth}{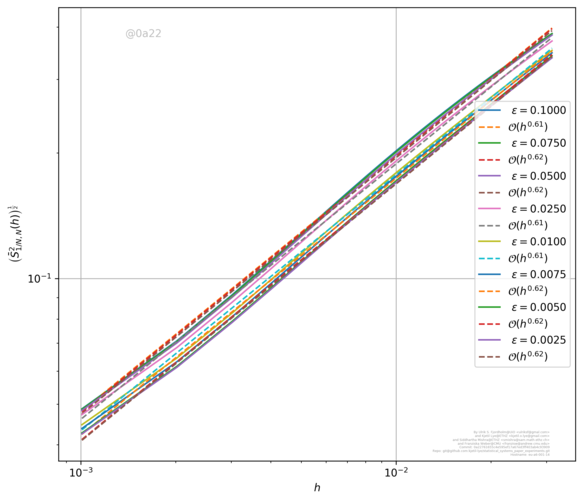}
		\InputImage{0.5\textwidth}{\textwidth}{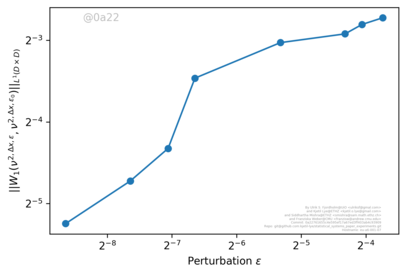}
		\caption{Different amplitudes of perturbation.}
	\end{subfigure}
		\begin{subfigure}[b]{\textwidth}
	\InputImage{0.4\textwidth}{0.8\textwidth}{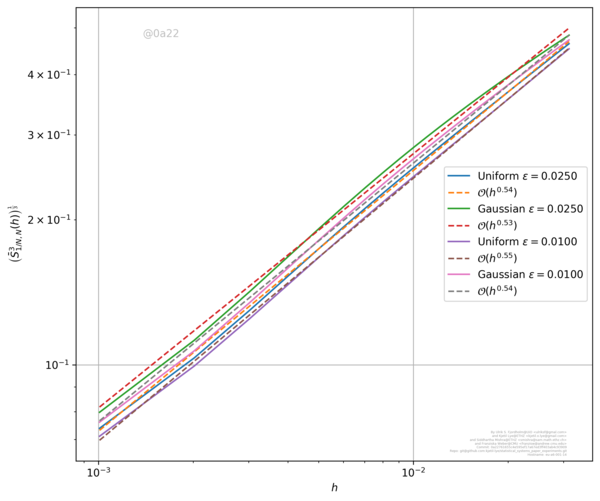}
	\InputImage{0.5\textwidth}{0.8\textwidth}{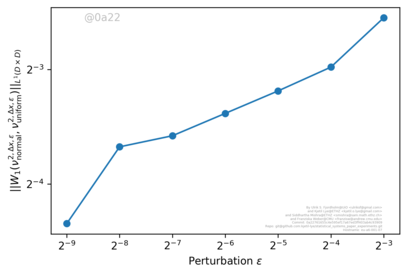}
	\caption{Different distributions of parameters.}
\end{subfigure}
	\begin{subfigure}[b]{\textwidth}
		\InputImage{0.4\textwidth}{0.8\textwidth}{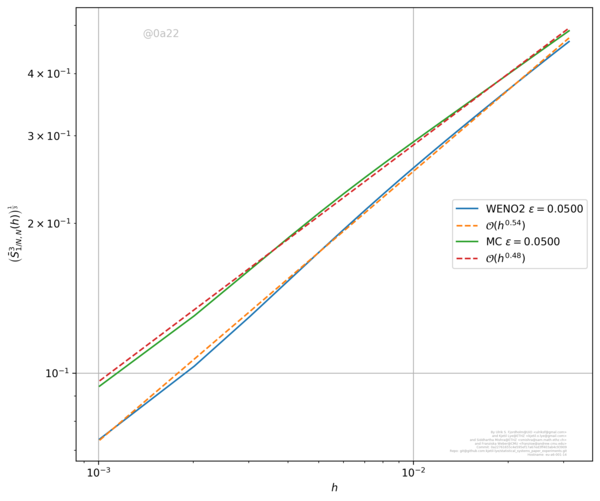}
		\InputImage{0.5\textwidth}{0.8\textwidth}{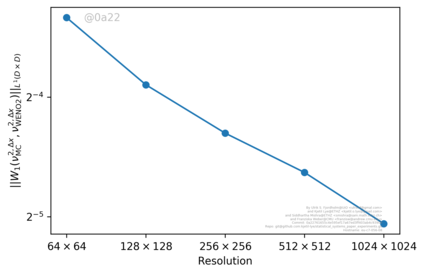}
		\caption{Different numerical schemes.}
	\end{subfigure}
\caption{Stability of the statistical solution with respect to variations of different parameters in the Kelvin--Helmholtz problem \eqref{eq:kh}. Left column: The structure function \eqref{eq:sft} for $p=2$. Right column: Different Wasserstein distances. All computations are at time $T=2$, computed on a fine grid of $1024^2$ points and with $1024$ Monte Carlo samples.}
\label{fig:KHstab}
\end{figure}

\subsection{Statistical steady state and regularity}\label{sec:reg}
In addition to the four numerical experiments reported in the last section, i.e., Kelvin--Helmholtz, Richtmeyer--Meshkov, and fractional Brownian motion with two different Hurst indices of $H=0.1$ and $H=0.5$, we have performed two further numerical experiments. Both of them consider the two-dimensional compressible Euler equations with the following initial data:
\begin{itemize}
    \item Fractional Brownian motion initial data with Hurst index $H=0.75$.
    \item Shock-vortex interaction initial data, see Section 6.3.2 of \cite{FMT_TeCNO} and references therein.
\end{itemize}
The Monte Carlo Algorithm is used to compute the approximate statistical solution $\mu_t^{\Dx,M}$ for these additional sets of initial data. 

We focus on the (time-sections of) the structure function \eqref{eq:sft} and find that in all six numerical experiments, the structure function behaved as \eqref{eq:sfnum}. The exponent $\theta_p(t)$ as a function of time, for $p=1,2,3$ and for each numerical experiment is shown in Figure \ref{fig:reg1}. We observe the following from this figure.
\begin{itemize}
    \item First, the exponent $\theta_p(t)$ reaches a steady state rather quickly, when compared to the dynamic behaviour of the solution. In other words, Figure \ref{fig:reg1} seems to suggest that statistical equilibrium is reached significantly faster than the (deterministic) steady state for individual realizations. Hence, the system evolves dynamically for each sample, while the whole ensemble has already reached statistical equilibrium. The time scale at which this statistical equilibrium is reached depends on the specifics of the initial data. 
    \item For all the experiments except the shock-vortex interaction, there is a very interesting behavior of the structure function \eqref{eq:sft} with respect to time. In particular, the exponent $\theta_p(t)$ for $p=2,3$ clearly increases with time, indicating that the non-linear evolution \emph{statistically regularizes} the solution in some manner. The exception is for the shock-vortex interaction where this exponent remains constant with time. This can be explained by the fact that the shock-vortex interaction results in a solution whose total variation (TV) norm is bounded. Hence, one can readily verify that $\theta_p(t) = \frac{1}{p}$, which is approximately realized in the computations. On the other hand and, as shown in Figure \ref{fig:BV}, the (average) BV-norm blows up for all the remaining test cases. Hence, in these problems, there is a complex interaction of structures at different length scales that leads to a subtle statistical regularity. 
    \item Last, but not least, we observe that in all the numerical experiments we have considered, the structure functions scale as \eqref{eq:sfnum}. Hence, the approximate scaling assumption \eqref{eq:sfscal} in Theorem \ref{theo:conv} is always observed to be satisfied.  
\end{itemize}

\begin{figure}
	\begin{subfigure}[b]{0.49\textwidth}
		\InputImage{\textwidth}{0.8\textwidth}{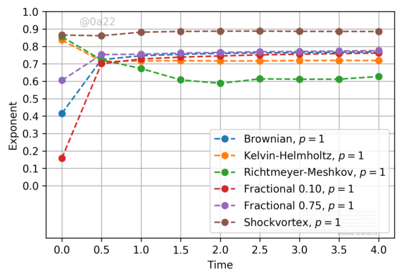}
		\caption{$p=1$}
	\end{subfigure}
	\begin{subfigure}[b]{0.49\textwidth}
	\InputImage{\textwidth}{0.8\textwidth}{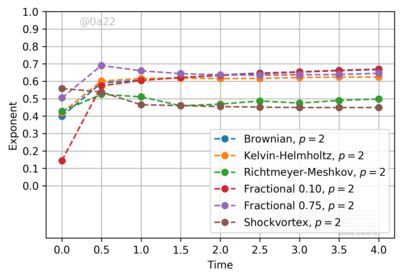}
	\caption{$p=2$}
\end{subfigure}
	\begin{subfigure}[b]{0.49\textwidth}
	\InputImage{\textwidth}{0.8\textwidth}{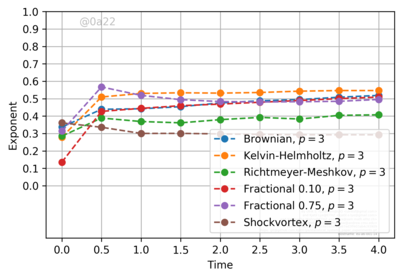}
	\caption{$p=3$}
\end{subfigure}
\caption{The evolution of the approximate scaling exponents of the structure functions as a function of time.}
\label{fig:reg1}

\end{figure}

\begin{figure}
	\InputImage{0.7\textwidth}{0.6\textwidth}{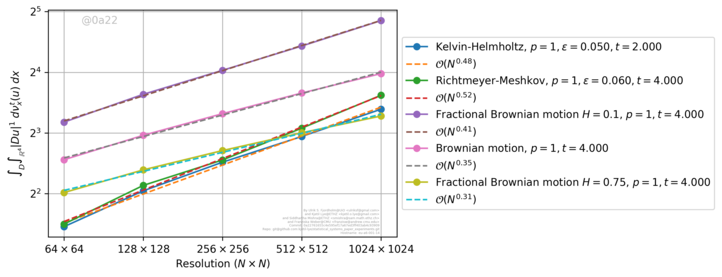}	
	\caption{BV norm as a function of resolution for different initial data.}
	\label{fig:BV}
\end{figure}

\subsection{Reproducing the numerical experiments}
All experiments were carried out using the open source Alsvinn simulator~\cite{alsvinn}. For a full description on how the experiments were carried out, along with the raw data and post processing scripts, consult 
\begin{center}
\url{https://github.com/kjetil-lye/systemspaper_experiments}
\end{center} 
and
\begin{center}
\url{https://github.com/kjetil-lye/statistical_systems_paper_experiments}
\end{center}

\section{Discussion}\label{sec:6}
We consider hyperbolic systems of conservation laws \eqref{eq:cauchy}. Although the standard solution framework of entropy solutions has been shown to be well-posed for scalar conservation laws and one-dimensional systems, it is now clearly established that entropy solutions for multi-dimensional systems are not unique, nor are they amenable to numerical approximation. On the other hand, numerical evidence presented in \cite{fkmt,FMTacta} and references therein suggests that a statistical notion of solutions might be more appropriate for \eqref{eq:cauchy}, even if the initial data and other underlying parameters are deterministic. 

Entropy measure-valued solutions \cite{diperna,fkmt} are a possible solution framework. Although global existence and numerical approximation results for entropy measure-valued solutions are available, it is well-known that measure-valued solutions are not unique, even for scalar conservation laws. This is largely on account of lack of information about multi-point spatial correlations. 

Inspired by the need to incorporate correlations, the authors of \cite{FLM17} proposed a framework of statistical solutions for hyperbolic systems of conservation laws. Statistical solutions are time-parameterized probability measures on the space of $p$-integrable functions. They were shown in \cite{FLM17} to be equivalent to adding information about all possible multi-point correlations to the measure-valued solution. The time-evolution of these measures is prescribed by a system of nonlinear tensorized moment-transport equations \eqref{eq:momentcorrmeas}. Under an additional entropy condition, the well-posedness of statistical solutions for scalar conservation laws was shown in \cite{FLM17} and the numerical approximation of statistical solutions for scalar conservation laws was considered in \cite{FLyeM1}.

Our main aim in this paper was to propose a numerical algorithm to approximate statistical solutions of multi-dimensional hyperbolic systems of conservation laws. To this end, we combined a high-resolution finite volume method \eqref{eq:semi_d} with a Monte Carlo sampling procedure to obtain Algorithm \ref{alg:mc}, that computes statistical solutions. 

The task of proving convergence of the approximations $\mu_t^{\Dx,M}$, generated by the Monte Carlo Algorithm, was rather intricate. First, we had to completely characterize an appropriate topology on the space of time-parameterized probability measures on $p$-integrable functions. This topology is based on the topology induced by the underlying correlation measures. We showed that this induced topology is equivalent to the weak topology on the space of probability measures on $L^p$, for any fixed time, but it also induces appropriate extensions when time is varied. The resulting compactness theorem \ref{thm:timedepcmcompactness} delineates the class of admissible observables that converge in this topology. Essentially, this theorem boils down to the convergence of time averages of (multi-point) statistical quantities of interest such as the mean, variance, multi-point correlation functions and structure functions \eqref{eq:sf}. We believe that this topology on time-parameterized probability measures on $L^p$ and novel sufficient conditions for ensuring convergence in it, might have independent applications in probability theory and stochastic analysis.

Next, we proved in Theorem \ref{theo:conv} that, under certain assumptions on the underlying finite volume schemes, the approximate statistical solutions converge, upon mesh refinement, in the aforementioned topology to a time-parameterized probability measure on $L^p$. A Lax--Wendroff theorem was proved, showing that the limit measure is indeed a statistical solution. Finally, a standard Monte Carlo convergence argument was used to guarantee convergence, under sample augmentation, of the approximations generated by the Monte Carlo Algorithm.

The assumptions in Theorem \ref{theo:conv} include $L^p$ stability and the weak BV bound \eqref{eq:wbv}, which are satisfied by many existing high-resolution entropy stable finite volume schemes, such as the TeCNO schemes of \cite{FMT_TeCNO}. On the other hand, we also required a subtle (approximate) scaling assumption \eqref{eq:sfscal} that is an analogue of the well-known scaling assumptions in Kolmogorov's theories for homogeneous, isotropic turbulence \cite{fris1}. Although this assumption was verified in all the numerical experiments, we were unable to prove it here. Hence, we have provided a conditional existence result for statistical solutions of multi-dimensional systems of conservation laws in this paper.  

Moreover, we proposed an entropy condition and a notion of dissipative statistical solutions (Definition \ref{def:dss}) and proved a weak-strong uniqueness result for these dissipative statistical solutions. This provides us with a conditional uniqueness result for statistical solutions, i.e., if they exist, then strong statistical solutions (Definition \ref{def:strongss}) are unique. In particular, we obtain short time existence and uniqueness results for dissipative statistical solutions. 

We present extensive numerical experiments for the two-dimensional compressible Euler equations to illustrate our Monte Carlo Algorithm. The results validate the convergence analysis and demonstrate convergence, on mesh refinement and sample augmentation, for statistical quantities of interest such as the mean, variance, structure functions, one-point probability density functions and two-point joint probability density functions. Moreover, we observe convergence in appropriate Wasserstein distances \eqref{eq:LpWp} for the multi-point correlation measures. Summarizing the results of the numerical experiments, we conclude that one observes \emph{convergence of all interesting statistical observables} in our framework. This should be contrasted to the state of the art, where deterministic quantities do not converge on mesh refinement \cite{fkmt}. Thus, we provide rigorous justification of the computability of statistical quantities of interest in the context of multi-dimensional systems of conservation laws.

Furthermore, we discover from the numerical experiments that 
\begin{itemize}
    \item The computed statistical solutions are remarkably stable with respect to different variations. In particular, we varied the amplitude of initial perturbations, the type of initial perturbations leading to different probability measures on $L^p$, and also the underlying numerical method. In all these cases, we observed that the computed statistical solutions were stable with respect to these perturbations. This observed stability augurs very well for identifying further constraints or admissibility criteria in order to obtain uniqueness of statistical solutions.
    \item The correlation structure of the statistical solutions seem to reach an equilibrium at significantly shorter time scales than the actual flow. This behavior is clearly seen in the variation of structure functions over time (Figure \ref{fig:reg1}). Thus, statistical stationary might be reached much faster than the actual evolution would suggest. 
    \item There seems to be a subtle gain in regularity, as measured by the decay exponent of the structure function, for the statistical solutions (Figure \ref{fig:reg1}). Qualitatively, it seems that mixing of the underlying structures leads to a gain in regularity. This observed regularity needs to be studied further.
    
\end{itemize}
The results in this paper can be extended in different directions. On the theoretical side, a key question is whether the scaling assumption \eqref{eq:sfscal} can be proved for some numerical approximations or relaxed in an appropriate manner. This will pave the way for a unconditional global existence result for statistical solutions. 

As formulated, (dissipative) statistical solutions are not necessarily unique. If we start with a deterministic initial data, i.e.~setting $\bar{\mu} = \delta_{\bar{u}}$, then we can apply the construction of \cite{CDL1,CDK1} to obtain infinitely many entropy solutions corresponding to the initial data. Each deterministic solution defines a statistical solution. Thus, infinitely many statistical solutions are possible for the same (deterministic) initial data. On the other hand, numerical experiments strongly suggest that the computed statistical solutions are stable. Thus, we need to find further admissibility criteria to single out an unique statistical solution. Moreover, the observed gain in regularity might provide additional constraints to obtain uniqueness.

On the computational side, the Monte Carlo algorithm \ref{alg:mc} can be very expensive, even prohibitively expensive in three space dimensions. Hence, it is imperative to consider alternatives to accelerate it. Alternatives such as the multi-level Monte Carlo method and quasi-Monte Carlo method are considered in a forthcoming publication. Another alternative would be to use deep learning, such as in \cite{LMR1}, to accelerate the Monte Carlo algorithm. 

\section*{Acknowledgements.}
The research of SM was partially supported by the ERC consolidator grant ERC COG NN. 770880 COMANFLO. Numerical results in this paper were obtained from computations performed on the Swiss National supercomputing center (CSCS) Lugano, through production projects s665 and s839. FW was partially supported by NFR project no.~250302.

\appendix

\section{Proof of \Cref{thm:weakconvergenceequivalence}}\label{app:weakconvergenceequivalenceproof}
In order to characterize weak convergence of probability measures on $L^p(D;U)$, we will use a construction from \cite[Section 5.1, pp.~106--107]{ambrosio_gradient_flows} which we summarize here. Assume that $X$ is Polish (i.e.,~a complete and separable metric space) and let $X_0\subset X$ be a dense, countable subspace. For fixed $q_1,q_2,q_3\in\Q$ with $q_2,q_3\in(0,1)$ and every $v\in X_0$, define $H:X\to\R$ by
\begin{equation}\label{eq:defhatfunction}
H(u) := \min\big(q_1+q_2 d(u,v),\ q_3\big).
\end{equation}
The collection $\widehat{\Lambda}_0$ of all such ``hat functions'' $H$ is clearly countable. Let 
\begin{equation}\label{eq:defc0}
\Lambda_0 := \big\{ q\min(H_1,\dots,H_m)\ : \ q\in\Q,\ H_1,\dots,H_m\in\widehat{\Lambda}_0 \text{ for any } m\in\N\big\}.
\end{equation}
Then $\Lambda_0$ is also countable, and it can be shown that weak convergence of a sequence $(\mu_n)_n$ is equivalent to
\begin{equation}\label{eq:c0convergence}
\Ypair{\mu_n}{F} \to \Ypair{\mu}{F} \qquad \forall\ F\in \Lambda_0
\end{equation}
(see the aforementioned reference).

We prove first Lemma \ref{lem:lgcont}.
\begin{proof}[Proof of Lemma \ref{lem:lgcont}]
If $g\in\Caratheodory^{k,p}(D;U)$ and $u\in L^p(D;U)$ then
\begin{align*}
|L_g(u)| &\leq \int_{D^k} |g(x,u(x_1),\dots,u(x_k))|\,dx \\
&\leq \sum_{\alpha\in\{0,1\}^k} \int_{D^k} \phi_{|\hat{\alpha}|}(x_{\hat{\alpha}})|u(x_1)|^{\alpha_1p}\cdots |u(x_k)|^{\alpha_kp}\,dx \\
&= \sum_{\alpha\in\{0,1\}^k} \|\phi_{|\hat{\alpha}|}\|_{L^1(D^{|\hat{\alpha}|})}\|u\|_{L^p(D)}^{|\alpha|p} \\
&= \sum_{i=0}^k \binom{k}{i}\|\phi_i\|_{L^1(D^i)}\|u\|_{L^p(D)}^{p(k-i)} < \infty.
\end{align*}
If $g\in\Caratheodory^{k,p}_1$ and $u,v\in L^p(D;U)$ then ($\hat{u}^j(x):=(u(x_1),\dots, u(x_{j-1}),u(x_{j+1}),\dots,u(x_k))$)
\begin{align*}
|L_g(u)-L_g(v)| &\leq C\sum_{i=1}^k \int_{D^k} |u(x_i)-v(x_i)|\max\big(|u(x_i)|,|v(x_i)|\big)^{p-1}\prod_{j\neq i} h(\hat{x}^j,\hat{u}^j({x})\big) dx\\
&\leq Ck\|u-v\|_{L^p}\big\|\max(|u|,|v|)\big\|_{L^p}^{(p-1)/p}\big(\|\phi_1\|_{L^1}+\|u\|_{L^p}^p\big)^{k-1},
\end{align*}
whence $L_g$ is continuous (and in fact locally Lipschitz).
\end{proof}

\begin{lemma}\label{lem:polynomialcp}
Let $F\in\big(\Lgfuncs^p\big)^k$ for some $k\in\N$, i.e.\ $F(u) = (F_1(u),\dots,F_k(u))$ for $F_1,\dots,F_k$ lying either in $\Lgfuncs^p$ or in $\Lgfuncs^p_1$. Let $P : \R^k \to \R$ be a polynomial in $k$ variables. Then $P\circ F$ lies in $\Lgfuncs^p$ or in $\Lgfuncs^p_1$, respectively.
\end{lemma}
\begin{proof}
$\Lgfuncs^p$ is clearly closed under scalar multiplication, as $\alpha L_g = L_{\alpha g}$ for any $\alpha\in\R$. It is therefore enough to show that $\Lgfuncs^p$ ($\Lgfuncs^p_1$, resp.) is closed under addition and multiplication. Let $g_i:D^{k_i}\times U^{k_i}\to\R$ for $i=1,2$ be Carath\'eodory functions satisfying \eqref{eq:gbounded}, \eqref{eq:glipschitz}. Then
\begin{align*}
L_{g_1}(u)L_{g_2}(u) &= \int_{D^{k_1}} g_1(x,u(x_1),\dots,u(x_{k_1}))\,dx\int_{D^{k_2}} g_2(x,u(x_1),\dots,u(x_{k_2}))\,dx \\
&= \int_{D^{k}} g(x,u(x_1),\dots,u(x_k))\,dx = L_g(u),
\end{align*}
where $k=k_1+k_2$ and 
\[
g(x,\xi) = g_1(x_1,\dots,x_{k_1},\xi_1,\dots,\xi_{k_1})g_2(x_{k_1+1},\dots,x_k,\xi_{k_1+1},\dots,\xi_k).
\]
The function $g$ is readily seen to satisfy \eqref{eq:gbounded}, and using \eqref{eq:gbounded} it can be seen that $g$ also satisfies \eqref{eq:glipschitz} whenever $g_1,g_2$ do so. This shows that $\Lgfuncs^p$ ($\Lgfuncs^p_1$, resp.) is closed under multiplication. To show that $L_{g_1}+L_{g_2}\in \Lgfuncs^p$ (or $\Lgfuncs^p_1$, respectively), assume that, say, $k_1\leq k_2$. If $k_1=k_2$ then clearly $L_{g_1}+L_{g_2}\in\Lgfuncs^p$ (or $\Lgfuncs^p_1$, respectively), so assume that $k_1<k_2$. Let $k_0=k_2-k_1$ and let $\phi\in L^1(D^{k_0})$ satisfy $\phi\geq0$ and $\int_{D^{k_0}}\phi(x)\,dx = 1$. Define $g:D^{k_2}\times U^{k_2}\to\R$ by
\[
g(x,\xi) = g_1(x_1,\dots,x_{k_1},\xi_1,\dots,\xi_{k_1})\phi(x_{k_1+1},\dots,x_{k_2}) + g_2(x,\xi).
\]
It is now straightforward to verify that $L_{g_1}+L_{g_2}=L_g$ and that $g$ satisfies \eqref{eq:gbounded}, \eqref{eq:glipschitz}.
\end{proof}

\begin{proof}[Proof of \Cref{thm:weakconvergenceequivalence}]
We may assume that $B$ is closed. By Lemma \ref{lem:lgcont}, every $L_g\in\Lgfuncs^p_1$ is continuous and bounded on $B$. For any $L_g\in \Lgfuncs^p_1$ there is, by the Tietze extension theorem \cite[Section 4.2]{Folland}, an $F\in C_b(L^p(D;U))$ satisfying $F=L_g$ on $B$, so we see that if $\mu_n\weakto\mu$ then $\Ypair{\mu_n}{L_g}=\Ypair{\mu_n}{F}\to\Ypair{\mu}{F}=\Ypair{\mu}{L_g}$.

Assume conversely that $\Ypair{\mu_n}{F}\to\Ypair{\mu}{F}$ for all $F\in \Lgfuncs^p_1$. We claim that the ``hat functions'' $H:L^p(D;U)\to\R$ defined in \eqref{eq:defhatfunction} can be approximated uniformly on $B$ by functions in $\Lgfuncs^p_1$. Let $v\in L^p(D;U)$ and $q_1,q_2,q_3\in\Q$ be as in \eqref{eq:defhatfunction}, and let $R := \sup_{u\in B} \|u-v\|_{L^p}^p < \infty$. Let $\psi(s) = \min\big(q_1+q_2 |s|^{1/p},\ q_3\big)$, which is a continuous, bounded function on $\R$. Let $\eps>0$ and let $P:\R\to\R$ be a polynomial such that $\|\psi-P\|_{C([-R,R])}<\eps$. The function $u\mapsto \|u-v\|_{L^p}^p$ clearly lies in $\Lgfuncs^p_1$, so Lemma \ref{lem:polynomialcp} implies that also $u\mapsto P\big(\|u-v\|_{L^p}^p\big)$ can be written as a function $L_g\in\Lgfuncs^p_1$. For any $u\in B$ we then have
\[
|H(u)-L_g(u)| = \big|\psi\big(\|u-v\|_{L^p}^p\big) - P\big(\|u-v\|_{L^p}^p\big)\big| < \eps.
\]
This proves the claim.

Next, let $F\in \Lambda_0$, where $\Lambda_0$ is given by \eqref{eq:defc0}. Let $\eps>0$, let $g_1,\dots,g_m$ be such that $\|H_i-L_{g_i}\|_{C_b(B)} < \eps$ for $i=1,\dots,m$, and let $P:\R^m\to\R$ be a polynomial such that 
\[
\sup_{r\in[-\eps,1+\eps]^m}\big|\min(r_1,\dots,r_m)-P(r)\big|<\eps.
\]
By Lemma \ref{lem:polynomialcp} we can write $P\circ \big(L_{g_1},\dots,L_{g_m}\big) = L_g$ for some $L_g\in\Lgfuncs^p_1$. We conclude that
\begin{align*}
\sup_{u\in B}\big|F(u) - L_g(u)\big| &= \sup_{u\in B}\big|\min(H_1(u),\dots,H_m(u)) - P(L_{g_1}(u),\dots,L_{g_m}(u))\big| \\ &\leq \sup_{u\in B}\big|\min(H_1(u),\dots,H_m(u)) - \min(L_{g_1}(u),\dots,L_{g_m}(u))\big| \\
&\quad + \sup_{u\in B}\big|\min(L_{g_1}(u),\dots,L_{g_m}(u)) - P(L_{g_1}(u),\dots,L_{g_m}(u))\big| \\
&\leq 2\eps,
\end{align*}
by the 1-Lipschitz continuity of the min-function and the approximation properties of $g_1,\dots,g_m$ and $P$. We can conclude that \eqref{eq:c0convergence} holds, and hence $\mu_n\rightharpoonup\mu$ weakly.
\end{proof}

\section{The compactness theorem}\label{app:proofcompactness}
Throughout this appendix we will use the cutoff functions
\[
\theta(s) := \begin{cases}
1 & |s|\leq 1 \\
2-|s| & 1<|s|<2 \\
0 & |s|\geq2,
\end{cases}
\qquad \theta_R(v):=\theta(|v|/R),
\qquad \zeta_R(v) := v\theta_R(v)
\]
(for some $R\geq1$), defined for $s\in\R$ and $v\in U$, where $U\subset \R^N$ as before. Note that $\zeta_R(v)=v$ for $|v|<R$ and $\zeta_R(v)=0$ for $|v|>2R$, and that $\|\zeta_R\|_{\Lip} \leq 2$.

For a function $u\in L^p(D)$ we define the modulus of continuity
\begin{align*}
\omega_r^p(u) &:= \int_D\intavg_{B_r(0)}|u(x+z)-u(x)|^p\,dzdx.
\end{align*}

\begin{lemma}\label{lem:cutoff}
If $u\in L^q(D,U)$ for some $q\in[1,\infty)$ then
\[
\big\|\zeta_R\circ u-u\big\|_{L^q(D)} \leq 3\omega_r^q(u)^{1/q}
\]
whenever $R \geq \frac{1}{|B_r|^{1/q}}\|u\|_{L^q(D)}$.
\end{lemma}
\begin{proof}
Denote $u_r(x) := \intavg_{B_r(0)}u(x+z)\,dz$. Then $|u_r(x)| \leq \frac{1}{|B_r|^{1/q}}\|u\|_{L^q(D)}$ for every $x\in D$, so that if $R \geq \frac{1}{|B_r|^{1/q}}\|u\|_{L^q(D)}$ then $u_r = \zeta_R\circ u_r$. It follows that
\begin{align*}
\big\|\zeta_R\circ u-u\big\|_{L^q(D)} &\leq \big\|\zeta_R\circ u-\zeta_R\circ u_r\big\|_{L^q(D)} + \|u_r-u\|_{L^q(D)} \\
&\leq \big(\|\zeta_R\|_{\rm Lip}+1\big)\|u_r-u\|_{L^q(D)} \\
&\leq 3\omega_r^q(u)^{1/q},
\end{align*}
where we have used the fact that $\|u_r-u\|_{L^q} \leq \omega_r^q(u)^{1/q}$.
\end{proof}

\begin{proof}[Proof of Theorem \ref{thm:cmcompactness}]
For every $k\in\N$, the sequence $(\nu^k_n)_{n\in\N}\subset\Caratheodory^{k*}_0(D;U)$ is bounded with norm $\|\nu^k_n\|_{\Caratheodory^{k*}} \equiv 1 < \infty$, and so has a weak*-convergent subsequence. Thus, we can extract a diagonal subsequence $(n_j)_{j\in\N}$ such that $\nu^{k}_{n_j} \weaksto \nu^k \in \Caratheodory^{k*}_0(D;U)$ for every $k\in\N$. For the sake of notational simplicity we denote $n = n_j$ for the remainder of this proof. 

We start by showing \ref{prop:liminf}. Let $\kappa$ and $\phi$ be nonnegative functions, as prescribed. We may assume that $\liminf_{n\to\infty}\Ypair{\nu^k_n}{g}_{\Caratheodory^k} < \infty$, for otherwise there is nothing to prove. For $R>0$, let $\kappa_R\in C_c(U^k)$ and $\phi_R\in L^1(D^k)$ be the functions $\kappa_R(\xi)=\kappa(\xi)\theta_R(\xi)$ and $\phi_R(x)=\phi(x)\theta_R(x)$. Then $\kappa_R\to\kappa$ and $\phi_R\to\phi$ pointwise almost everywhere as $R\to\infty$. If $g_R(x,\xi):=\phi_R(x)\kappa_R(\xi)$ then clearly $g_R\to g$ as $R\to\infty$ almost everywhere, and $g_R\in \Caratheodory^k_0(D;U)$ for every $R$. We then find that
\begin{align*}
\Ypair{\nu^k}{g}_{\Caratheodory^k} &= \lim_{R\to\infty} \Ypair{\nu^k}{g_R}_{\Caratheodory^k} & \text{(Fatou's lemma)}\\
&= \lim_{R\to\infty} \liminf_{n\to\infty} \Ypair{\nu^k_n}{g_R}_{\Caratheodory^k} & \text{(since $\nu^k_n \weaksto \nu^k$)} \\
&\leq \lim_{R\to\infty} \liminf_{n\to\infty} \Ypair{\nu^k_n}{g}_{\Caratheodory^k} & \text{(since $g_R \leq g$ for all $R$)} \\
&= \liminf_{n\to\infty} \Ypair{\nu^k_n}{g}_{\Caratheodory^k},
\end{align*}
which is \eqref{eq:nuklimitbound}.

Since $\nu^k_{n;x} \geq 0$ for all $k,n$ and a.e.\ $x\in D^k$, we have $\nu^k_x\geq0$ for all $k$ and a.e.\ $x\in D^k$. To see that $\nu^k_x(U^k)=1$ for a.e.\ $x\in D^k$, let $A_1,\dots,A_k\subset D$ be bounded Borel sets of positive Lebesgue measure and define $A = A_1\times\dots\times A_k$. Then since $\nu^k_{n;x}$ are probability measures and converge weak* as we have just shown,
\begin{equation*}
1 \geq \intavg_A \Ypair{\nu^k_{n;x}}{\theta_R}\ dx \to \intavg_A \Ypair{\nu^k_{x}}{\theta_R}\ dx \qquad \text{as } n\to\infty.
\end{equation*}
Conversely,
\begin{align*}
\intavg_A 1-\Ypair{\nu^k_{x}}{\theta_R}\ dx 
&= \lim_{n\to\infty} \intavg_A 1-\Ypair{\nu^k_{n;x}}{\theta_R}\ dx
= \lim_{n\to\infty} \intavg_A \Ypair{\nu^k_{n;x}}{1-\theta_R}\ dx \\
&\leq \limsup_{n\to\infty} \frac{1}{R^p}\intavg_A \Ypair{\nu^k_{n;x}}{|\xi|^p}\ dx \\
&\leq \limsup_{n\to\infty} \frac{1}{R^p} \sum_{i=1}^k\intavg_A\Ypair{\nu^k_{n;x}}{|\xi_i|^p}\ dx \\
&\leq \frac{1}{R^p}\sum_{i=1}^k\frac{c^p}{|A_i|},
\end{align*}
where we have first used Chebyshev's inequality and then the uniform $L^p$ bound \eqref{eq:uniformlpbound}. Passing to $R\to\infty$ in the two estimates above shows that
\begin{equation}\label{eq:numassone}	
\intavg_A \nu^k_{x}(U^k)\ dx = 1.
\end{equation}
Since $A_1,\dots,A_k$ were arbitrary, it follows that $\nu^k_x(U^k) = 1$ for a.e.\ $x\in D^k$.

The limit $\cm$ clearly satisfies the symmetry condition. 
For consistency, we need to show that for any $f\in C_0(U^{k-1})$ and $\phi\in C_c(D^k)$, we have
\begin{equation}\label{eq:symmetry}
\int_{D^k} \phi(x)\Ypair{\nu^k_x}{f(\xi_1,\dots,\xi_{k-1})}\ dx = \int_{D^k} \phi(x)\Ypair{\nu^{k-1}_{x_1,\dots,x_{k-1}}}{f(\xi_1,\dots,\xi_{k-1})}\ dx
\end{equation}
Define $\bar{f}_R\in C_0(U^k)$ by $\bar{f}_R(\xi_1,\dots,\xi_k) = f(\xi_1,\dots,\xi_{k-1})\theta_R(\xi_k).$ We estimate the difference between the left- and right-hand sides of \eqref{eq:symmetry} by $E_1+E_2+E_3+E_4$, where
\begin{align*}
E_1 &= \left|\int_{D^k}\phi(x)\Ypair{\nu^k_x}{f(\xi_1,\dots,\xi_{k-1})-\bar{f}_R(\xi)}\ dx\right| \\
E_2 &= \left|\int_{D^k}\phi(x)\left(\Ypair{\nu^k_x}{\bar{f}_R} - \Ypair{\nu^k_{n;x}}{\bar{f}_R}\right)\ dx\right| \\
E_3 &= \left|\int_{D^k}\phi(x)\left(\Ypair{\nu^k_{n;x}}{\bar{f}_R} - \Ypair{\nu^{k-1}_{n;\hat{x}_k}}{f}\right)\ dx\right| \\
E_4 &= \left|\int_{D^k}\phi(x)\left(\Ypair{\nu^{k-1}_{n;\hat{x}_k}}{f} - \Ypair{\nu^{k-1}_{\hat{x}_k}}{f}\right)\ dx\right|
\end{align*}
where $\hat{x}_k = (x_1,\dots,x_{k-1})$. 
The fact that $E_1\to0$ as $R\to\infty$ follows from the Dominated Convergence Theorem. For every $R>0$ we have $E_2$, $E_4 \to 0$ as $n\to\infty$ by the weak* convergence of $\nu^k_n$ and $\nu^{k-1}_n$, respectively. For $E_3$ we can write
\begin{align*}
E_3 &= \left|\int_{D^k}\phi(x)\Ypair{\nu^k_{n;x}}{f\bigl(\xi_1,\dots,\xi_{k-1}\bigr)\bigl(\theta_R(\xi_k)-1\bigr)}\ dx\right| \\
&\leq \|f\|_{C_0(U^{k-1})}\|\phi\|_{C_0(D^k)}\int_{\supp(\phi)}\Ypair{\nu^1_{n;x_k}}{1-\theta_R(\xi)}\ dx \\
&\leq \|f\|_{C_0(U^{k-1})}\|\phi\|_{C_0(D^k)}\frac{C}{R^p} \to 0
\end{align*}
as $R\to\infty$ (as in the proof of \eqref{eq:numassone}). Thus, for a given $\eps>0$ we can first select $R$ such that $E_1,E_3<\eps$, and then select $n$ such that $E_2,E_4<\eps$. This proves \eqref{eq:symmetry}.

The $L^p$ bound \ref{prop:lpbound} follows from \eqref{eq:nuklimitbound} with $k=1$, $\phi\equiv 1$ and $\kappa(\xi) = |\xi|^p$, since $\Ypair{\nu^1}{|\xi|^p}_{\Caratheodory^1} \leq \liminf_n \Ypair{\nu^1_n}{|\xi|^p}_{\Caratheodory^1} \leq c^p$, by \eqref{eq:uniformlpbound}. To prove the diagonal continuity property \ref{prop:dc} of $\cm$ we use \eqref{eq:nuklimitbound} with $k=2$, $\phi(x,y) = \frac{\Indicator{B_r(x)}(y)}{|B_r(x)|}$ and $\kappa(\xi_1,\xi_2) = |\xi_1-\xi_2|^p$, and note that $\omega_r^p(\nu^2) = \Ypair{\nu^2}{g}_{\Caratheodory^2}$. The diagonal continuity of $\nu^2$ then follows from $\omega_r^p(\nu^2) \leq \liminf_{n\to\infty} \omega_r^p(\nu^2_n) \to 0$ as $r\to0$, by \eqref{eq:uniformdc}.

We prove \ref{prop:strongconv} for $k=1$ and give a sketch of the general case. Let $g\in\Caratheodory^{1,p}_1(D;U)$, that is, $g$ is a Carath\'eodory function such that $|g(x,\xi)|\leq \phi_1(x) + \phi_0|\xi|^p$ for some $\phi_0>0$ and nonnegative $\phi_1\in L^1(D)$, as well as $|g(x,\xi)-g(y,\zeta)| \leq C\max(|\xi|,|\zeta|)^{p-1}|\xi-\zeta|$ for some $C>0$. Define
\begin{align*}
u_n(x) := \Ypair{\nu^1_{n;x}}{g(x,\xi)}, && u_n^R(x) := \Ypair{\nu^1_{n;x}}{\theta_R(|\xi|^p)g(x,\xi)}, \\
u(x) := \Ypair{\nu^1_{x}}{g(x,\xi)}, && u^R(x) := \Ypair{\nu^1_{x}}{\theta_R(|\xi|^p)g(x,\xi)}.
\end{align*}
The sequence $\{u_n\}_{n\in\N}$ satisfies
\begin{equation}\label{eq:udiagcontargument}
\begin{split}
\int_D\intavg_{B_r(0)} &|u_n(x+z)-u_n(x)|\,dzdx \\
&\leq \int_D\intavg_{B_r(0)}\big|\Ypair{\nu^2_{n,x,x+z}}{g(x+z,\xi_2)-g(x,\xi_1)}\big|\,dzdx \\
&\leq C\int_D\intavg_{B_r(0)}\Ypair{\nu^2_{n,x,x+z}}{\max(|\xi_1|,|\xi_2|)^{p-1}|\xi_2-\xi_1|}\,dz dx \\
&\leq C\left(\int_D\intavg_{B_r(0)}\Ypair{\nu^2_{n,x,x+z}}{\max(|\xi_1|,|\xi_2|)^p}\,dz dx\right)^{(p-1)/p} \\
&\qquad \qquad \qquad\quad \left(\int_D\intavg_{B_r(0)}\Ypair{\nu^2_{n,x,x+z}}{|\xi_2-\xi_1|^p}\,dz dx\right)^{1/p} \\
&\leq 2Cc^{p-1} \omega_r^p\big(\nu^2_n\big)^{1/p},
\end{split}
\end{equation}
so together with the compactness of $D$ and the uniform bound $\|u_n\|_{L^1(D)} \leq \|\phi_1\|_{L^1(D)} + \phi_0c^p$ we can apply Kolmogorov's Compactness Theorem \cite[Theorem A.8]{HR},~\cite{Sudakov1957} to conclude that $\{u_n\}_{n\in\N}$ is precompact in $L^1(D)$. Hence, there exists a subsequence $\{u_{n_l}\}_{l\in\N}$ and some $\bar{u}\in L^1(D)$ such that $u_{n_l}\to\bar{u}$ as $l\to\infty$ in $L^1(D)$. From the weak* convergence of $\nu^1_n$ we know that $u_n^R \weakto u^R$ as $n\to\infty$ weakly in $L^1(D)$ for every $R>0$. Lebesgue's dominated convergence theorem implies that $u^R \to u$ as $R\to\infty$ in $L^1(D)$. 

We claim that $u^R_n \to u_n$ as $R\to\infty$ in $L^1(D)$, \textit{uniformly} in $n$. Indeed, choosing $R>0$ such that $R\geq \frac{1}{|B_r|^{1/p}}M$ (where $M$ is the constant in \eqref{eq:LpLinftybound}), we get
\begin{align*}
&\int_D |u_n(x)-u_n^R(x)|\,dx = \int_D \big|\Ypair{\nu^1_{n;x}}{(1-\theta_R(|\xi|^p))g(x,\xi)}\big|\,dx \\
&\leq \int_D \Ypair{\nu^1_{n;x}}{(1-\theta_R(|\xi|^p))\phi_1(x)}\,dx + \phi_0\int_D \Ypair{\nu^1_{n;x}}{(1-\theta_R(|\xi|^p))|\xi|^p}\,dx \\
&= \underbrace{\int_{L^p}\int_D (1-\theta_R(|u(x)|^p))\phi_1(x)\,dxd\mu_n(u)}_{=:F_1}\\
&\hphantom{=}\quad + \phi_0\underbrace{\int_{L^p}\int_D (1-\theta_R(|u(x)|^p))|u(x)|^p\,dxd\mu_n(u)}_{=:F_2}.
\end{align*}
The first term can be bounded by
\begin{align*}
F_1 &\leq \int_{L^p}\int_{\{x\in D\ :\ |u(x)|^p>R\}}\phi_1(x)\,dxd\mu_n(u) \\
&\leq \int_{L^p}\sup\left\{\int_{D'}\phi_1(x)\,dx\ :\ D'\subset D,\ |D'|\leq \|u\|_{L^p}^p/R\right\}d\mu_n(u) \\
&\leq \sup\left\{\int_{D'}\phi_1(x)\,dx\ :\ D'\subset D\ :\ |D'|\leq c^p/R\right\}
\end{align*}
where we in the second inequality used Chebyshev's inequality and in the third inequality the uniform $L^p$ bound \eqref{eq:uniformlpbound}. The above converges to 0 as $R\to\infty$, uniformly in $n$. For the second term we have
\begin{align*}
F_2 &= \int_{L^p}\int_D \big||u(x)|^p - \zeta_R(|u(x)|^p)\big|\,dxd\mu_n(u) = \int_{L^p}\big\||u|^p - \zeta_R\circ|u|^p\big\|_{L^1(D)}\,d\mu_n(u) \\
&\leq 3\int_{L^p} \omega_r^1(|u|^p)\,d\mu_n(u) \leq 6pc^{p-1}\int_{L^p} \omega_r^p(u)^{1/p}\,d\mu_n(u) \\
&\leq 6pc^{p-1}\omega_r^p(\nu^2_n)^{1/p}
\end{align*}
where the first inequality follows from Lemma \ref{lem:cutoff} with $q=1$, the second inequality follows an estimate similar to \eqref{eq:udiagcontargument}, and the third inequality is H\"older's inequality. The final term above vanishes as $R\to\infty$ uniformly in $n$, by the uniform diagonal continuity assumption \eqref{eq:uniformdc}. It follows that for any $\psi\in L^\infty(D)$,
\begin{align*}
\Big|\int_D \psi(u-\bar{u})\,dx\Big| &\leq \|\psi\|_{L^\infty}\Big(\|u-u^R\|_{L^1}+\|u_{n_l}^R-u_{n_l}\|_{L^1} + \|u_{n_l}-\bar{u}\|_{L^1}\Big) \\
&\quad+ \Big|\int_D \psi\big(u^R-u_{n_l}^R\big)\,dx\Big|,
\end{align*}
all of which vanish as $R\to\infty$ and $l\to\infty$. We conclude that $\bar{u}=u$, whence $u_{n_l}\to u$ as $l\to\infty$. By the uniqueness of the limit $u$, we get convergence of the whole sequence: $u_n \to u$ in $L^1(D)$ as $n\to\infty$.

For general $k\in\N$ we prove only that $u_n(x):=\Ypair{\nu^k_{n;x}}{g(x,\xi)}$ satisfies a bound of the form \eqref{eq:udiagcontargument}, and leave the rest to the reader. We write first
\begin{align*}
&\int_{D^k}\intavg_{B_r(0)^k}|u_n(x+z)-u_n(x)|\,dzdx \\
&\leq \int_{D^k}\intavg_{B_r(0)^k}\Ypair{\nu^{2k}_{n;x,x+z}}{\big|g\big(x+z,\xi_{k+1},\dots,\xi_{2k}\big)-g\big(x,\xi_1,\dots,\xi_k\big)\big|}\,dzdx \\
&\leq C\sum_{l=1}^k \int_{D^k}\intavg_{B_r(0)^k}\Ypair{\nu^{2k}_{n;x,x+z}}{|\xi_{l+k}-\xi_l|\max(|\xi_l|,|\xi_{l+k}|)^{p-1}h(\hat{x}^l, \hat{\xi}^l)}\,dzdx
\end{align*}
(cf.~Definition \ref{def:caratheodory} with $\hat{\xi}^l=(\xi_1,\dots,\xi_{l-1},\xi_{l+1},\dots,\xi_k)$). Consider, say, the last summand above:
\begin{align*}
\int_{D^k}&\intavg_{B_r(0)^k}\Ypair{\nu^{2k}_{n;x,x+z}}{|\xi_{2k}-\xi_k|\max(|\xi_k|,|\xi_{2k}|)^{p-1}h(\hat{x}^k, \hat{\xi}^k)}\,dzdx \\
&=\int_{D^k}\intavg_{B_r(0)}\Ypair{\nu^{k+1}_{n;x,x_k+z_k}}{|\xi_{k+1}-\xi_k|\max(|\xi_k|,|\xi_{k+1}|)^{p-1}h(\hat{x}^k, \hat{\xi}^k)}\,dz_kdx \\
&=\int_{L^p(D)}\int_{D^k}\intavg_{B_r(0)}\big|u(x_k+z_k)-u(x_k)\big|\max(|u(x_k)|,|u(x_k+z_k)|)^{p-1}\\
&\qquad\qquad\qquad\qquad\quad h(x_1,\dots,x_{k-1}, u(x_1),\dots,u(x_{k-1}))\,dz_kdxd\mu_n(u) \\
&=\int_{L^p(D)}\left(\int_{D}\intavg_{B_r(0)}|u(x_k+z_k)-u(x_k)|\max(|u(x_k)|,|u(x_k+z_k)|)^{p-1}\,dz_kdx_k\right) \\
&\qquad\qquad \ \, \Biggl(\int_{D^{k-1}} h(x_1,\dots,x_{k-1}, u(x_1),\dots,u(x_{k-1}))\,d\hat{x}^k\Biggr)d\mu_n(u) \\
&\leq 2\int_{L^p(D)}\left(\int_{D}\intavg_{B_r(0)}|u(x_k+z_k)-u(x_k)|^p\,dz_kdx_k\right)^{1/p}\left(\int_D |u(x_k)|^p\,dx_k\right)^{(p-1)/p} \\
&\qquad\qquad\quad C\big(1+\|u\|_{L^p}^{p(k-1)}\big)\,d\mu_n(u) \\
&\leq 2C\omega_r^p\big(\nu^2_n\big)^{1/p}M^{p-1}(1+M^{p(k-1)})
\end{align*}
where the second-last inequality follows from H\"older's inequality and the boundedness of $L_h$ (cf.~the proof of Lemma \ref{lem:lgcont}), and the last inequality follows from \eqref{eq:LpLinftybound}. By the uniform diagonal continuity of $\cm_n$, the above vanishes uniformly as $r\to0$. The rest of the proof follows as in the proof for $\nu^1$.
\end{proof}

\section{Time-dependent correlation measures}\label{app:time-dep-cm}

\begin{proof}[Proof of Lemma \ref{lem:cmtimeslice}]
By assumption, $(t,x)\mapsto \Ypair{\nu^k_{t,x}}{g}$ is measurable for all $g\in C_0(U^k)$, so by separability of $\Caratheodory^k_0(D;U)$, the map $t \mapsto \int_{D^k} \Ypair{\nu^k_{t,x}}{g(x)}\,dx$ is also measurable for any $g\in\Caratheodory^k_0(D;U)$. Let $E_0\subset \Caratheodory^k_0(D;U)$ be a countable, dense subset of the unit sphere in $\Caratheodory^k_0(D;U)$, and let
\[
\mathcal{T} := \bigcap_{g\in E_0} \Big\{\text{Lebesgue points for } t \mapsto \int_{D^k} \Ypair{\nu^k_{t,x}}{g(x)}\,dx\Big\},
\]
a set whose complement $[0,T) \setminus \mathcal{T}$ has zero Lebesgue measure. The set $E := \textrm{span}\,E_0$ is dense in $\Caratheodory^k_0(D;U)$, and every $t\in\mathcal{T}$ is still a Lebesgue point for $t \mapsto \int_{D^k} \Ypair{\nu^k_{t,x}}{g(x)}\,dx$ whenever $g\in E$.  For every $t\in\mathcal{T}$ we define the functional $\rho(t) : E \to \R$ by
\[
\rho(t)(g) := \int_{D^k}\Ypair{\nu^k_{t,x}}{g(x)}\,dx.
\]
Then $\rho(t)$ is linear: For any $g,h\in E$ and $\alpha\in\R$ we have
\begin{align*}
\rho(t)(\alpha g+h)
&= \lim_{h\to0^+}\intavg_{t-h}^{t+h}\int_{D^k}\Ypair{\nu^k_{s,x}}{\alpha g(x)+h(x)}\,dxds \\
&= \lim_{h\to0^+}\left(\alpha\intavg_{t-h}^{t+h}\int_{D^k}\Ypair{\nu^k_{s,x}}{g(x)}\,dxds + \intavg_{t-h}^{t+h}\int_{D^k}\Ypair{\nu^k_{s,x}}{h(x)}\,dxds\right) \\
&= \alpha\rho(t)(g) + \rho(t)(h).
\end{align*}
Moreover, $\rho(t)$ is continuous, as
\begin{align*}
|\rho(t)(g)| 
&\leq \lim_{h\to0^+} \intavg_{t-h}^{t+h}\int_{D^k}|\Ypair{\nu^k_{t,x}}{g(x)}|\,dxdt
\leq \lim_{h\to0^+} \intavg_{t-h}^{t+h}\int_{D^k}\|g(x)\|_{C_0}\,dxdt \\
&= \|g\|_{\Caratheodory^k_0(D;U)}.
\end{align*}
It follows that for every $t\in\mathcal{T}$ the functional $\rho(t)$ can be extended uniquely to a continuous linear functional on $\Caratheodory^k_0(D;U)$. The remaining claims in the lemma are now readily checked.
\end{proof}


\section{Proof of \Cref{theo:lxw}}
\label{app:laxwendroff}
We write
\[
\AvgEvolved{\InitialData{u}}{t}{\Dx}{\bi}=(\AvgS{\NumericalEvolution{\Dx}_t\InitialData{u}})_\bi
\]
where $\NumericalEvolution{\Dx}$ is the numerical evolution operator; see Section \ref{sec:fvm}.
\begin{proof}
	Let $\mu^{\Dx}_t$ be the approximate statistical solution, defined in \eqref{eq:statapp}. Let $k\in \N$, and let $\phi\in C_c^\infty(D^k)$ be a test function. Fix  $\bi=(\bi_1,\ldots, \bi_k)\in(\GridIndexSetD)^{k}$,  By changing the order of integration, we have
	\begin{multline*}
		\int_{\TD}\int_{L^1(\D,\Ph)}\AvgS{u}_{\bi_1}\cdots \AvgS{u}_{\bi_k}\partial_t\phi(\vec{x}_\bi,t)\;d\mu^{\Dx}_t(u) dt=\int_{\TD}\int_{L^1(\D,\Ph)}\AvgEvolved{\InitialData{u}}{t}{\Dx}{\bi_1} \cdots\AvgEvolved{\InitialData{u}}{t}{\Dx}{\bi_k}\partial_t\phi(\vec{x}_\bi,t)\;d\InitialData{\mu}(\InitialData{u}) \;dt\\
		=\int_{L^1(\D,\Ph)}\int_{\TD}\AvgEvolved{\InitialData{u}}{t}{\Dx}{\bi_1}\cdots \AvgEvolved{\InitialData{u}}{t}{\Dx}{\bi_k}\partial_t\phi(\vec{x}_\bi,t)\; dt\; d\InitialData{\mu}(\InitialData{u}),
	\end{multline*}
where
\[\vec{x}_\bi:=\begin{pmatrix}x_{\bi_1}&\cdots&x_{\bi_k}\end{pmatrix}.\]
	Since $\{\AvgS{\InitialData{u}}^{\cdot,\Dx}_{\bj}\}_{\bj\in\GridIndexSetD}$ solves \eqref{eq:semi_d} with initial data $\left\{\AvgS{\InitialData{u}}_{\bj}\right\}_{\bj\in\GridIndexSetD}$ for every $\InitialData{u}\in L^1(\D, \Ph)$, we get that 
	\begin{multline*}
	\int_{L^1(\D,\Ph)}	\int_{\TD}\AvgEvolved{\InitialData{u}}{t}{\Dx}{\bi_1}\cdots \AvgEvolved{\InitialData{u}}{t}{\Dx}{\bi_k}\partial_t\phi(\vec{x}_\bi,t)\\
		-\frac{1}{\Dx}\sum_{n=1}^k\sum_{m=1}^d\AvgEvolved{\InitialData{u}}{t}{\Dx}{\bi_1}\cdots\Bigg(F^m_{\bi_n+\hf \be_{m}}(\NumericalEvolution{\Dx}(t)(\InitialData{u}))-F^m_{\bi_n-\hf \be_{m}}(\NumericalEvolution{\Dx}(t)(\InitialData{u}))\Bigg)\cdots \AvgEvolved{\InitialData{u}}{t}{\Dx}{\bi_k}\phi(\vec{x}_\bi,t)\; dt \\
		\qquad\qquad\qquad+\phi(\vec{x}_\bi,0)\AvgEvolved{\InitialData{u}}{0}{\Dx}{\bi_1}\cdots\AvgEvolved{\InitialData{u}}{0}{\Dx}{\bi_k}\;d\InitialData{\mu}(\InitialData{u})=0.
	\end{multline*}
	%
	%
	%
	%
	Multiplying by $\Dx^{kd}$ and summing over $\GridIndexSetDK{k}$, we get
	\begin{multline*}
		\sum_{\bi\in\GridIndexSetDK{k}}\Dx^{kd}\Bigg(	\int_{L^1(\D,\Ph)}\int_{\TD}\AvgEvolved{\InitialData{u}}{t}{\Dx}{\bi_1}\cdots \AvgEvolved{\InitialData{u}}{t}{\Dx}{\bi_k}\partial_t\phi(\vec{x}_\bi,t)\\
		-\frac{1}{\Dx}\sum_{n=1}^k\sum_{m=1}^d\AvgEvolved{\InitialData{u}}{t}{\Dx}{\bi_1}\cdots\Bigg(F^m_{\bi_n+\hf \be_{m}}(\NumericalEvolution{\Dx}(t)(\InitialData{u}))-F^m_{\bi_n-\hf \be_{m}}(\NumericalEvolution{\Dx}(t)(\InitialData{u}))\Bigg)\cdots \AvgEvolved{\InitialData{u}}{t}{\Dx}{\bi_k}\phi(\vec{x}_\bi,t)\; dt \\
		\qquad\qquad\qquad+\phi(\vec{x}_\bi,0)\AvgEvolved{\InitialData{u}}{0}{\Dx}{\bi_1}\cdots\AvgEvolved{\InitialData{u}}{0}{\Dx}{\bi_k}\;d\InitialData{\mu}(\InitialData{u})\Bigg)=0.
	\end{multline*}
	%
	%
	%
	%
	For $1\leq n \leq k$ and $1\leq m\leq d$, summation by parts gives
	\begin{align*}
		\sum_{\bi\in\GridIndexSetDK{k}}\Dx^{dk}\int_{\TD}\int_{L^1(\D,\Ph)}\AvgS{{u}}_{\bi_1}\cdots\Bigg(F^m_{\bi_n+\hf \be_{m}}(u))-F^m_{\bi_n-\hf \be_{m}}(u))\Bigg)\cdots\AvgS{{u}}_{\bi_k} \phi(\vec{x}_\bi,t)\;dt\;d {\mu_t^{\Dx}}({u})\\
		=	\sum_{\bi\in\GridIndexSetDK{k}}\Dx^{dk}\int_{\TD}\int_{L^1(\D,\Ph)}\AvgS{{u}}_{\bi_1}\cdots F^m_{\bi_n+\hf \be_{m}}(u)\cdots \AvgS{{u}}_{\bi_k}\Bigg(\phi(\vec{x}_\bi,t)-\phi(\vec{x}_{\bi+\be_{n,m}},t)\Bigg)\;dt\;d {\mu_t^{\Dx}}({u})
	\end{align*}
	%
	We furthermore have that

	\begin{align*}
	\sum_{\bi\in\GridIndexSetDK{k}}\Dx^{dk}\int_{\TD}\int_{L^1(\D,\Ph)}\AvgS{{u}}_{\bi_1}\cdots F^m_{\bi_n+\hf \be_{m}}(u)\cdots \AvgS{{u}}_{\bi_k}\Bigg(\frac{\phi(\vec{x}_\bi,t)-\phi(\vec{x}_{\bi+\be_{n,m}},t)}{\Dx}\Bigg)\;d \mu^{\Dx}_t(u)\;dt\\
	 =\sum_{\bi\in\GridIndexSetDK{k}}\Dx^{dk}\int_{\TD}\int_{L^1(\D,\Ph)}\AvgS{{u}}_{\bi_1}\cdots f^m(\AvgS{{u}}_{\bi_n})\cdots \AvgS{{u}}_{\bi_k}\Bigg(\frac{\phi(\vec{x}_\bi,t)-\phi(\vec{x}_{\bi+\be_{n,m}},t)}{\Dx}\Bigg)\;d^{\Dx}_t{\mu}({u})\;dt\\
		-\sum_{\bi\in\GridIndexSetDK{k}}\Dx^{dk}\int_{\TD}\int_{L^1(\D,\Ph)}\AvgS{{u}}_{\bi_1}\cdots \Bigg(f^m(\AvgS{{u}}_{\bi_n})-F^m_{\bi_n+\hf \be_{m}}(u)\Bigg)\cdots \AvgS{{u}}_{\bi_k}\Bigg(\frac{\phi(\vec{x}_\bi,t)-\phi(\vec{x}_{\bi+\be_{n,m}},t)}{\Dx}\Bigg)\;d {\mu^{\Dx}_t}({u})\;dt
	\end{align*}
	%
	%
	%
	%
	and by the Lipschitz continuity \eqref{eq:fvm_lipschitz}, we have
	\begin{align*}
		\sum_{\bi\in\GridIndexSetDK{k}}\Dx^{dk}\int_{\TD}\int_{L^1(\D,\Ph)}\AvgS{{u}}_{\bi_1}\cdots \Bigg(f^m(\AvgS{{u}}_{\bi_n})-F^m_{\bi_n+\hf \be_{m}}(u)\Bigg)\cdots \AvgS{{u}}_{\bi_k}\Bigg(\frac{\phi(\vec{x}_\bi,t)-\phi(\vec{x}_{\bi+\be_{n,m}},t)}{\Dx}\Bigg)\;d {\mu_t^{\Dx}}({u})\;dt\\
		\leq
		\sum_{\bi\in\GridIndexSetDK{k}}\Dx^{dk}\int_{\TD}\int_{L^1(\D,\Ph)}\Bigg|\AvgS{{u}}_{\bi_1}\cdots \Bigg(f^m(\AvgS{{u}}_{\bi_n})-F^m_{\bi_n+\hf \be_{m}}(u)\Bigg)\cdots \AvgS{{u}}_{\bi_k}\Bigg(\frac{\phi(\vec{x}_\bi,t)-\phi(\vec{x}_{\bi+\be_{n,m}},t)}{\Dx}\Bigg)\Bigg|\;d {\mu_t^{\Dx}}({u})\;dt\\
		\leq C 
		\sum_{\bi\in\GridIndexSetDK{k}}\Dx^{dk}\sum_{q=-p+1}^{p}\int_{\TD}\int_{L^1(\D,\Ph)}\Big|\AvgS{{u}}_{\bi_1}\Big|\cdots \Big|\AvgS{{u}}_{\bi_n+q\be_m}-\AvgS{{u}}_{\bi_n}\Big|\cdots \Big|\AvgS{{u}}_{\bi_k}\Big|\Big|\frac{\phi(\vec{x}_\bi,t)-\phi(\vec{x}_{\bi+\be_{n,m}},t)}{\Dx}\Big|\;d {\mu_t^{\Dx}}({u})\;dt\\
		\leq C 
		\sum_{\bi\in\GridIndexSetDK{k}}\Dx^{dk}\int_{\TD}\int_{L^1(\D,\Ph)}\Big|\AvgS{{u}}_{\bi_1}\Big|\cdots \Big|\AvgS{{u}}_{\bi_n}-\AvgS{{u}}_{\bi_n-\be_m}\Big|\cdots \Big|\AvgS{{u}}_{\bi_k}\Big|\Big|\frac{\phi(\vec{x}_\bi,t)-\phi(\vec{x}_{\bi-\be_{n,m}},t)}{\Dx}\Big|\;d {\mu_t^{\Dx}}({u})\;dt
	\end{align*}
	%
	%
	%
	%
	By H\"older's inequality we get
	\begin{align*}
	\sum_{\bi\in\GridIndexSetDK{k}}\Dx^{dk}\int_{\TD}\int_{L^1(\D,\Ph)}\Big|\AvgS{{u}}_{\bi_1}\Big|\cdots \Big|\AvgS{{u}}_{\bi_n}-\AvgS{{u}}_{\bi_n-\be_m}\Big|\cdots \Big|\AvgS{{u}}_{\bi_k}\Big|\Big|\frac{\phi(\vec{x}_\bi,t)-\phi(\vec{x}_{\bi-\be_{n,m}},t)}{\Dx}\Big|\;d {\mu_t^{\Dx}}({u})\;dt\\
		\leq C\Bigg(\Dx^{kd}\int_{\TD}\int_{L^1(\D,\Ph)}\sum_{\bi\in\GridIndexSetDK{k}}\Big|\AvgS{{u}}_{\bi_n}-\AvgS{{u}}_{\bi_n-\be_m}\Big|^s\;d{\mu^{\Dx}_t}({u})\;dt\Bigg)^{1/s}\\
		\cdot\Bigg(\Dx^{kd}\sum_{\bi\in\GridIndexSetDK{k}}\int_{\TD}\int_{L^1(\D,\Ph)}\Big|\AvgS{{u}}_{\bi_1}\Big|^{\frac{s}{s-1}}\cdots \Big|\AvgS{{u}}_{\bi_k}\Big|^{\frac{s}{s-1}}\Big|\frac{\phi(\vec{x}_\bi,t)-\phi(\vec{x}_{\bi-\be_{n,m}},t)}{\Dx}\Big|^{\frac{s}{s-1}}\;d{\mu^{\Dx}_t}({u})\;dt\Bigg)^{\frac{s-1}{s}}.
	\end{align*}
	%
	%
	%
	%
	By the weak BV assumption \eqref{eq:lxweakbv}, we get
	\begin{align*}
	\lim_{\Dx\to 0}\Dx^{kd}\int_{\TD}\int_{L^1(\D,\Ph)}\sum_{\bi\in\GridIndexSetDK{k}}\Big|\AvgS{{u}}_{\bi_n}-\AvgS{{u}}_{\bi_n-\be_m}\Big|^s\;d{\mu^{\Dx}_t}({u})\;dt\\
	=\lim_{\Dx\to 0}\Dx^{d}\int_{\TD}\int_{L^1(\D,\Ph)}\sum_{\bi\in\GridIndexSetD}\Big|\AvgS{{u}}_{\bi_n}-\AvgS{{u}}_{\bi_n-\be_m}\Big|^s\;d{\mu^{\Dx}_t}({u})\;dt=0.
	\end{align*}
	%
	%
	%
	%
	We  furthermore note that since $\phi\in C_c^\infty$, we have
	\[\Dx^{kd}\sum_{\bi\in\GridIndexSetDK{k}}\int_{\TD}\int_{L^1(\D,\Ph)}\Big|\AvgS{{u}}_{\bi_1}\Big|^{\frac{s}{s-1}}\cdots \Big|\AvgS{{u}}_{\bi_k}\Big|^{\frac{s}{s-1}}\Big|\frac{\phi(\vec{x}_\bi,t)-\phi(\vec{x}_{\bi-\be_{n,m}},t)}{\Dx}\Big|^{\frac{s}{s-1}}\;d{\mu^{\Dx}_t}({u})\;dt\leq C.\]

	%
	%
	%
	%
	%
	Hence we get 
	\begin{multline*}
		0=\lim_{\Dx}\sum_{\bi\in\GridIndexSetDK{k}}\Dx^{kd}\Bigg(	\int_{L^1(\D,\Ph)}\int_{\TD}\AvgEvolved{\InitialData{u}}{t}{\Dx}{\bi_1}\cdots \AvgEvolved{\InitialData{u}}{t}{\Dx}{\bi_k}\partial_t\phi(\vec{x}_\bi,t)\\
		-\frac{1}{\Dx}\sum_{n=1}^k\sum_{m=1}^d\AvgEvolved{\InitialData{u}}{t}{\Dx}{\bi_1}\cdots\Bigg(F^m_{\bi_n+\hf \be_{m}}(\NumericalEvolution{\Dx}(t)(\InitialData{u}))-F^m_{\bi_n-\hf \be_{m}}(\NumericalEvolution{\Dx}(t)(\InitialData{u}))\Bigg)\cdots \AvgEvolved{\InitialData{u}}{t}{\Dx}{\bi_k}\phi(\vec{x}_\bi,t)\; dt \\
		\qquad\qquad\qquad+\phi(\vec{x}_\bi,0)\AvgEvolved{\InitialData{u}}{0}{\Dx}{\bi_1}\cdots\AvgEvolved{\InitialData{u}}{0}{\Dx}{\bi_k}\;d\InitialData{\mu}(\InitialData{u})\Bigg)\\
		=\int_{\TD}\int_{L^1(\D, \Ph)}\int_{(\R^d)^k}u(x_1)\cdots u(x_k)\partial_t \phi(\vec{x},t)\\
		+\sum_{n=1}^k\sum_{m=1}^du(x_1)\cdots f^m(u(x_n))\cdots u(x_k)\partial_{nm}\phi(\vec{x},t)\;d\vec{x}\;d\mu^t(u)\;dt \\
		+\int_{L^1(\D, \Ph)}\int_{(\R^d)^k}\InitialData{u}(x_1)\cdots\InitialData{u}(x_k)\phi(\vec{x},0)\;d\vec{x}\;d\InitialData{\mu}(\InitialData{u}),
	\end{multline*}
	which completes the proof.
\end{proof}

\section{Proof of \eqref{eq:wcauchy1}}\label{app:wcauchyproof}
\begin{proposition}
Let $D\subset\R^d$ be bounded and let $\mu,\tilde{\mu}\in\Prob(L^p(D;U))$ both have bounded $p$th moment. Then
\begin{equation}
    \label{eq:wassdist}
\left(\int_{D^k} W_p\big(\nu^k_x,\tilde{\nu}^k_x\big)^p\,dx\right)^{1/p} \leq C W_p(\mu,\tilde{\mu})
\end{equation}
for some $C=C(p,k,|D|)$, where $(\nu^k)_{k\in\N}, (\tilde{\nu}^k)_{k\in\N}$ are the correlation measures corresponding to $\mu,\tilde{\mu}$.
\end{proposition}
\begin{proof}
Let $\pi\in\Prob(L^p(D;U)^2)$ be an optimal transport plan for $W_p(\mu,\tilde{\mu})$, and let $(\sigma^k)_{k\in\N}\in\Corrmeas^p(D;U^2)$ be its corresponding correlation measure. We claim that the marginals of $\sigma^k_x$ are $\nu^k_x$ and $\tilde{\nu}^k_x$, respectively, for a.e.~$x\in D^k$. Indeed, if $g\in\Caratheodory^k_0(D;U)$ then, denoting $\bar{g}(x,\xi,\zeta):=g(x,\xi)$,
\begin{align*}
\Ypair{\sigma^k}{\bar{g}} &= \Ypair{\pi}{L_{\bar{g}}} 
= \int_{L^p(D;U)^2} L_{\bar{g}}(u,v)\,d\pi(u,v) = \int_{L^p(D;U)^2} L_g(u)\,d\pi(u,v) \\
&= \int_{L^p(D;U)} L_g(u)\,d\mu(u) = \Ypair{\nu^k}{g}.
\end{align*}
A similar computation holds for the second marginal. Since $g$ is arbitrary, the claim follows. We can therefore estimate
\begin{align*}
\int_{D^k} W_p\big(\nu^k_x,\tilde{\nu}^k_x\big)^p\,dx &\leq \int_{D^k} \int_{U^{2k}}|\xi-\zeta|^p\,d\sigma_x^k(\xi,\zeta)\,dx \\
&= \sum_{l=1}^k\int_{D^k} \int_{U^{2k}}|\xi_l-\zeta_l|^p\,d\sigma_x^k(\xi,\zeta)\,dx \\
&= k\int_{D^k} \int_{U^2}|\xi_1-\zeta_1|^p\,d\sigma_{x_1}^1(\xi,\zeta)\,dx \\
&= k|D|^{k-1}\int_{D} \int_{U^2}|\xi-\zeta|^p\,d\sigma_{x}^1(\xi,\zeta)\,dx \\
&= k|D|^{k-1}\int_{L^p(D;U)^2}\int_{D} |u(x)-v(x)|^p\,dx\,d\pi(u,v) = W_p(\mu,\tilde{\mu})^p.
\end{align*}
\end{proof}

\section{Computing the Wasserstein distance for sums of Diracs}\label{app:wasserstein}
In \Cref{sec:5}, we computed the Wasserstein distance between different numerical solutions. The numerical solutions are always sums of Diracs, which simplifies the computations greatly. 

For two $k$-th correlation marginals $\nu^{k,1}$ and $\nu^{k,2}$, we are interested in the distance

\[\|W_1(\nu^{k,1}_{\cdot, T}, \nu^{k,2}_{\cdot,T})\|_{L^1(D^k)}.\]

\subsection{Computing the Wasserstein distance between the first correlation marginals}
To compute the first correlation marginals, we use the function\\ \verb|scipy.stats.wasserstein_distance| in the \verb|scipy| module for Python~\cite{scipy} to compute the Wasserstein distance. The function computes the Wasserstein distance by going through the CDF, consult the \verb|scipy| documentation for more information.

We approximate the spatial integral as a sum over all the volume averages. 

\subsection{Computing the Wasserstein distance between the second correlation marginals}
For the second correlation marginals, we use the function \verb|ot.emd| in the \verb|POT| module for Python~\cite{POT}. This function uses the Hungarian algorithm~\cite{hungarian} to compute the Wasserstein distance between sums of Diracs.

The spatial integral is approximated by taking 10 spatial points in each direction, in other words, we use $10\; 000$ evaluation of the Wasserstein distance.


\bibliography{biblo}{}
\bibliographystyle{abbrv}

\end{document}